\documentclass[reqno,a4paper]{amsart}

\usepackage{amsthm}
\usepackage{amsmath}
\usepackage[all]{xy} \SelectTips{eu}{}
\usepackage{latexsym,amsfonts,amssymb}
\usepackage{hyperref}
\hypersetup{colorlinks=true, linkcolor=teal}
\usepackage{cite}
\hypersetup{citecolor=brown}
\usepackage{leftidx}
\usepackage[dvipsnames,svgnames,table]{xcolor}
\usepackage{mathptmx}
\usepackage{nicefrac}
\usepackage{array}
\usepackage{mathrsfs}
\usepackage{rotating}
\usepackage{tikz-cd}

\usepackage[pagewise]{lineno}
\usepackage{setspace}
\makeatletter
\renewcommand{\@biblabel}[1]{%
  \makebox[4em][l]{[#1]}%
}
\makeatother

\usepackage{vmargin}
\setpapersize{A4}
\setmargrb{25mm}{20mm}{25mm}{20mm}

\renewcommand{\theequation}{$\smash{\sharp}\mspace{0.5mu}$\arabic{equation}}

\newcommand{\numberseries}{\mdseries}   

\newlength{\thmtopspace}                
\newlength{\thmbotspace}                
\newlength{\thmheadspace}               
\newlength{\thmindent}                  

\newtheoremstyle{bfupright head,slanted body}
                {\thmtopspace}{\thmbotspace}
                {\slshape}{\thmindent}{\bfseries}{.}{\thmheadspace}
                {{\numberseries \thmnumber{{\bf #2} }}\thmnote{#3}}

\newtheoremstyle{bfupright head,upright body}
                {\thmtopspace}{\thmbotspace}
                {\upshape}{\thmindent}{\bfseries}{.}{\thmheadspace}
                {{\numberseries \thmnumber{{\bf #2} }}\thmnote{#3}}

\newtheoremstyle{bfit head,upright body}
                {\thmtopspace}{\thmbotspace}
                {\upshape}{\thmindent}{\upshape}{.}{\thmheadspace}
                {{\numberseries\thmnumber{{\bf #2} }}
                {\bfseries\itshape\thmnote{\negthickspace#3}}}

\newtheoremstyle{it head,upright body}
                {\thmtopspace}{\thmbotspace}
                {\upshape}{\thmindent}{\upshape}{.}{\thmheadspace}
                {{\numberseries\thmnumber{{\bf #2} }}
                {\itshape\thmnote{\negthickspace#3}}}

\newtheoremstyle{fixed bf head,slanted body}
                {\thmtopspace}{\thmbotspace}{\slshape}
                {\thmindent}{\bfseries}{\!.}{\thmheadspace}
                {{\thmname{#1}\numberseries\thmnumber{ {\bf #2}}}\thmnote{ (#3)} }

\newtheoremstyle{fixed bf head,upright body}
                {\thmtopspace}{\thmbotspace}{\upshape}
                {\thmindent}{\bfseries}{\!.}{\thmheadspace}
                {{\thmname{#1}\numberseries\thmnumber{ {\bf #2}}}\thmnote{ (#3)} }

\newtheoremstyle{independent paragraph}
                {\thmtopspace}{\thmbotspace}
                {\upshape}{\thmindent}{\upshape}{}{0pt}
                {\thmnote{#3 }}

\newtheoremstyle{subparagraph}
                {\thmbotspace}{\thmbotspace}
                {\upshape}{\thmindent}{\upshape}{}{0pt}
                {\thmnote{#3 }}

\newtheoremstyle{notes}
                {\thmtopspace}{\thmbotspace}
                {\ttfamily}{\thmindent}{\ttfamily\small }{}{0pt}
                {\thmnote{#3 }}
                
\newtheoremstyle{numbered paragraph}
                {\thmtopspace}{\thmbotspace}{\upshape}
                {\thmindent}{\upshape}{}{\thmheadspace}
                {{\numberseries \thmnumber{\bf #2.}}}

\theoremstyle{bfupright head,slanted body}
\newtheorem{res}{}[section]             \newtheorem*{res*}{}

\theoremstyle{bfit head,upright body}
                 \newtheorem*{com*}{}

\theoremstyle{bfupright head,upright body}
               \newtheorem*{bfhpg*}{}

\theoremstyle{it head,upright body}
               \newtheorem*{ithpg*}{}

\theoremstyle{fixed bf head,slanted body}
\newtheorem{thm}[res]{Theorem}          \newtheorem*{thm*}{Theorem}
\newtheorem{prp}[res]{Proposition}      \newtheorem*{prp*}{Proposition}
\newtheorem{cor}[res]{Corollary}        \newtheorem*{cor*}{Corollary}
\newtheorem{lem}[res]{Lemma}            \newtheorem*{lem*}{Lemma}

\theoremstyle{fixed bf head,upright body}
\newtheorem{dfn}[res]{Definition}       \newtheorem*{dfn*}{Definition}
      \newtheorem*{obs*}{Observation}
\newtheorem{rmk}[res]{Remark}           \newtheorem*{rmk*}{Remark}
\newtheorem{exa}[res]{Example}          \newtheorem*{exa*}{Example}
         \newtheorem*{exe*}{Exercise}
\newtheorem{stp}[res]{Setup}            \newtheorem{stp*}{Setup}
         \newtheorem{ntn*}{Notation}
     \newtheorem{con*}{Construction}
      \newtheorem*{conj*}{Conjecture}

\theoremstyle{numbered paragraph}

\theoremstyle{subparagraph}

\theoremstyle{notes}

\newlength{\thmlistleft}        
\newlength{\thmlistright}       
\newlength{\thmlistpartopsep}   
\newlength{\thmlisttopsep}      
\newlength{\thmlistparsep}      
\newlength{\thmlistitemsep}     

\newcounter{eqc} 
  {\end{list}}%

\newcounter{prt}
\newenvironment{prt}{\begin{list}{\upshape (\alph{prt})}%
    {\usecounter{prt}%
      \setlength{\leftmargin}{\thmlistleft}%
      \setlength{\labelwidth}{\thmlistleft}%
      \setlength{\rightmargin}{\thmlistright}%
      \setlength{\partopsep}{\thmlistpartopsep}%
      \setlength{\topsep}{\thmlisttopsep}%
      \setlength{\parsep}{\thmlistparsep}%
      \setlength{\itemsep}{\thmlistitemsep}}}%
  {\end{list}}%

\newcounter{rqm}
\newenvironment{rqm}{\begin{list}{\upshape (\arabic{rqm})}%
    {\usecounter{rqm}%
      \setlength{\leftmargin}{\thmlistleft}%
      \setlength{\labelwidth}{\thmlistleft}%
      \setlength{\rightmargin}{\thmlistright}%
      \setlength{\partopsep}{\thmlistpartopsep}%
      \setlength{\topsep}{\thmlisttopsep}%
      \setlength{\parsep}{\thmlistparsep}%
      \setlength{\itemsep}{\thmlistitemsep}}}%
  {\end{list}}%

\newcounter{rqmm}
  {\end{list}}%

  {\end{list}}%

  {\end{list}}%

  \newcommand{\proofoftag}[2][:]{(#2)#1}

\newcommand{\pgref}[1]{\ref{#1}}

\renewcommand{\eqref}[1]{(\pgref{eq:#1})}

\newcommand{\lemref}[2][Lemma ]{#1\pgref{lem:#2}}

\newcommand{\prpref}[2][Proposition ]{#1\pgref{prp:#2}}

\newcommand{\stpref}[2][Setup~]{#1\pgref{stp:#2}}

\makeatletter
\def\@nobreak@#1{\mathchoice%
  {\nobreakdef@\displaystyle\f@size{#1}}%
  {\nobreakdef@\nobreakstyle\tf@size{\firstchoice@false #1}}%
  {\nobreakdef@\nobreakstyle\sf@size{\firstchoice@false #1}}%
  {\nobreakdef@\nobreakstyle\ssf@size{\firstchoice@false #1}}%
  \check@mathfonts}%
\def\nobreakdef@#1#2#3{\hbox{{%
                    \everymath{#1}%
                    \let\f@size#2\selectfont%
                    #3}}}%
\makeatother

\DeclareFontFamily{T1}{cmex}{}
\DeclareFontShape{T1}{cmex}{m}{n}{<-> s * [0.89] cmex10}{}
\DeclareSymbolFont{cmlargesymbols}{T1}{cmex}{m}{n}
\DeclareMathSymbol{\mycoprod}{\mathop}{cmlargesymbols}{"60} \let\coprod\mycoprod
\DeclareMathSymbol{\myprod}{\mathop}{cmlargesymbols}{"51} \let\prod\myprod

\DeclareSymbolFont{usualmathcal}{OMS}{cmsy}{m}{n}
\DeclareSymbolFontAlphabet{\mathcal}{usualmathcal}

\DeclareSymbolFont{letters}{OML}{txmi}{m}{it}

\DeclareMathSymbol{\alpha}{\mathord}{letters}{"0B}
\DeclareMathSymbol{\beta}{\mathord}{letters}{"0C}
\DeclareMathSymbol{\gamma}{\mathord}{letters}{"0D}
\DeclareMathSymbol{\sigma}{\mathord}{letters}{"0E}
\DeclareMathSymbol{\epsilon}{\mathord}{letters}{"0F}
\DeclareMathSymbol{\zeta}{\mathord}{letters}{"10}
\DeclareMathSymbol{\eta}{\mathord}{letters}{"11}
\DeclareMathSymbol{\theta}{\mathord}{letters}{"12}
\DeclareMathSymbol{\iota}{\mathord}{letters}{"13}
\DeclareMathSymbol{\kappa}{\mathord}{letters}{"14}
\DeclareMathSymbol{\lambda}{\mathord}{letters}{"15}
\DeclareMathSymbol{\mu}{\mathord}{letters}{"16}
\DeclareMathSymbol{\nu}{\mathord}{letters}{"17}
\DeclareMathSymbol{\xi}{\mathord}{letters}{"18}
\DeclareMathSymbol{\pi}{\mathord}{letters}{"19}
\DeclareMathSymbol{\rho}{\mathord}{letters}{"1A}
\DeclareMathSymbol{\sigma}{\mathord}{letters}{"1B}
\DeclareMathSymbol{\tau}{\mathord}{letters}{"1C}
\DeclareMathSymbol{\upsilon}{\mathord}{letters}{"1D}
\DeclareMathSymbol{\phi}{\mathord}{letters}{"1E}
\DeclareMathSymbol{\chi}{\mathord}{letters}{"1F}
\DeclareMathSymbol{\psi}{\mathord}{letters}{"20}
\DeclareMathSymbol{\omega}{\mathord}{letters}{"21}
\DeclareMathSymbol{\varepsilon}{\mathord}{letters}{"22}
\DeclareMathSymbol{\vartheta}{\mathord}{letters}{"23}
\DeclareMathSymbol{\varpi}{\mathord}{letters}{"24}
\DeclareMathSymbol{\varrho}{\mathord}{letters}{"25}
\DeclareMathSymbol{\varsigma}{\mathord}{letters}{"26}
\DeclareMathSymbol{\varphi}{\mathord}{letters}{"27}

\DeclareMathSymbol{\Gamma}{\mathord}{letters}{"00}
\DeclareMathSymbol{\Theta}{\mathord}{letters}{"02}
\DeclareMathSymbol{\Xi}{\mathord}{letters}{"04}
\DeclareMathSymbol{\Pi}{\mathord}{letters}{"05}
\DeclareMathSymbol{\Upsilon}{\mathord}{letters}{"07}
\DeclareMathSymbol{\Phi}{\mathord}{letters}{"08}
\DeclareMathSymbol{\Psi}{\mathord}{letters}{"09}

\DeclareMathSymbol{\upGamma}{\mathalpha}{operators}{"00}
\DeclareMathSymbol{\upDelta}{\mathalpha}{operators}{"01}
\DeclareMathSymbol{\upTheta}{\mathalpha}{operators}{"02}
\DeclareMathSymbol{\upLambda}{\mathalpha}{operators}{"03}
\DeclareMathSymbol{\upXi}{\mathalpha}{operators}{"04}
\DeclareMathSymbol{\upPi}{\mathalpha}{operators}{"05}
\DeclareMathSymbol{\upSigma}{\mathalpha}{operators}{"06}
\DeclareMathSymbol{\upUpsilon}{\mathalpha}{operators}{"07}
\DeclareMathSymbol{\upPhi}{\mathalpha}{operators}{"08}
\DeclareMathSymbol{\upPsi}{\mathalpha}{operators}{"09}
\DeclareMathSymbol{\upOmega}{\mathalpha}{operators}{"0A}

\DeclareMathAlphabet\PazoBB{U}{fplmbb}{m}{n}%

\newcommand{\Hom}[3]{\operatorname{Hom}_{#1}(#2,#3)}
\newcommand{\Ext}[4]{\operatorname{Ext}_{#1}^{#2}(#3,#4)}

\newcommand{\Ker}[1]{\operatorname{Ker}\mspace{1mu}#1}
\newcommand{\Coker}[1]{\operatorname{Cok}\mspace{1mu}#1}

\newcommand{\Prj}[1]{\operatorname{Prj}\mspace{2mu}#1}
\newcommand{\Prjp}[1]{\operatorname{Prj}(#1)}

\newcommand{\Hlettertwo}{\mathbb{H}}

\newcommand{\sHH}[3]{%
\Hlettertwo^{#1}_{#2,\,#3}}

\newcommand{\stalkco}[1]{S\mspace{-2mu}\langle{#1}\rangle}

\newcommand{\Sq}[1]{S_{\mspace{-4mu}\smash{#1}}}

\newcommand{\Fq}[1]{F_{\mspace{-2mu}\smash{#1}}}

\newcommand{\lMod}[1]{{}_{#1}\mspace{-1mu}\operatorname{Mod}}

\newcommand{\lMode}[2]{{}_{#1}\mspace{-1mu}\operatorname{Mod}^{\,#2}}

\newcommand{\lPrj}[1]{{}_{#1}\mspace{-1mu}\operatorname{Prj}}
\newcommand{\lprj}[1]{{}_{#1}\mspace{-1mu}\operatorname{prj}}
\newcommand{\lInj}[1]
{{}_{#1}\mspace{-1mu}\operatorname{Inj}}

\newcommand{\lsmod}[1]{{}_{#1}\mspace{-1mu}\operatorname{mod}}

\DeclareMathOperator{\Inj}{Inj}

\newcommand{\lPPrj}[1]{{}_{#1}\mspace{-1mu}\operatorname{PPrj}}

\newcommand{\exact}[1]{{}_{#1}\mathscr{E}}

\newcommand{\supp}[1]{\operatorname{supp}\mspace{1mu}#1}

\newcommand{\QSD}[2]{\mathbf{D}_{#1}(#2)}
\newcommand{\QSH}[2]{\mathbf{K}_{#1}(#2)}
\newcommand{\QST}[2]{\mathbf{K}^{\theta}_{#1}(#2)}
\newcommand{\QSPD}[2]{\mathbf{D}^{\text{pure}}_{#1}(#2)}
\newcommand{\qb}[1]{{}_{#1}\mspace{-1mu}\mathcal{B}}

\begin{document}

\title{Compact generation and perfect objects in homotopy categories \\ of exact model structures}

\author{Anastasios Slaftsos}

\address{Anastasios Slaftsos, Dipartimento di Matematica ``Tullio Levi-Civita'', Universit\'a degli Studi di Padova, via Trieste 63, 35131 Padova, Italy}
\email{anastasios.slaftsos@studenti.unipd.it}




\thanks{\textbf{Acknowledgements:} The author would like to thank Henrik Holm, Jan {\v{S}}{\v{t}}ov{\'{\i}}{\v{c}}ek, and Jorge Vit\'oria for helpful discussions on this work. The author was supported by NextGenerationEU, Mission 4 Component 1 Investment 4.1~, CUP C96E23000600001. }
  
\begin{abstract}
Compact generation allows large triangulated categories to be studied through suitably small objects. For homotopy categories arising from exact model structures, a natural question is how this intrinsic notion of smallness relates to finite presentation in the underlying exact category. In this paper, we investigate this relation for hereditary projective model structures generated by finitely presented Gorenstein projective objects. We establish compact generation and study the stronger requirement that every compact object admits a finitely presented cofibrant representative. We call objects with such representatives perfect. Under local coherence and the existence of enough projectives among the finitely presented objects, the finitely presented cofibrant objects form a Frobenius category whose stable category describes the perfect objects and whose idempotent completion recovers all compact objects. The relative $Q$-shaped setting, for suitable small categories $Q$, makes explicit how these finiteness properties depend on the choice of exact structure. The pure exact structure yields a compactly generated homotopy category, called the $Q$-shaped pure derived category. On the other hand, the split exact structure yields the $Q$-shaped homotopy category, which is compactly generated only in the pure semisimple case, analogously to the classical homotopy category of complexes.  Finally, for acyclic $Q$ and for left artinian rings, we show that every compact object admits a strictly perfect representative. In particular, we establish that compact and perfect objects coincide, proving the conjecture of Holm and J{\o}rgensen under these hypotheses.
\end{abstract}

\maketitle

\section{Introduction}
  Compact generation is one of the principal finiteness conditions imposed on a triangulated category with arbitrary coproducts. It first emerged as a common language for phenomena in stable homotopy theory and in unbounded derived categories, and its importance is reflected in the Ravenel--Neeman \cite[Thm. 2.1]{MR1308405} theorem and Brown representability: a set of compact generators controls the whole category, while its thick closure recovers precisely the compact objects \cite{Nee,Krause_2010}. The standard algebraic example is the derived category $\mathbf{D}(A)$ of a ring $A$, compactly generated by the stalk complex $A$, with compact objects the perfect complexes. At the level of model categories, accessibility often guarantees the weaker property of well-generation \cite{Rosicky_Brown_rep_in_hmtp}; compact generation, however, requires generation by genuinely small objects.

 Hovey's correspondence between abelian model structures and compatible complete cotorsion pairs, together with its extension to exact categories, provides a flexible method for constructing algebraic triangulated categories \cite{Hovey02,MR2811572,SaorinStovicek}. On the other hand, the theory of locally finitely presented additive categories revolves around small objects \cite{WCB94}. Positselski's locally coherent exact categories bring these two viewpoints together. Their exact structure is determined by conflations between finitely presented objects and is preserved under directed colimits \cite{Positselski_loc_coh}. This suggests that, when an exact model structure is generated by finitely presented objects, their finiteness should survive passage to the associated homotopy category. Naturally, one can question whether a homotopy category of an exact model structure is compactly generated.

This question already exhibits sharply different behaviour in familiar examples. The pure derived category of a ring, obtained from the pure exact structure, is compactly generated and fits into a recollement relating it to the homotopy category $\mathbf{K}(A)$ and the pure-acyclic complexes \cite{MR3473427,stovicek2014purity}. By contrast, the ordinary homotopy category $\mathbf{K}(A)$ of complexes of arbitrary modules need not even be well generated, see \cite{Nee}. More precisely, {\v{S}}{\v{t}}ov{\'{\i}}{\v{c}}ek proved that the homotopy category of complexes over a finitely accessible additive category is locally well generated, that is, for every set of objects $\mathcal{S}$,  $\operatorname{Loc}(\mathcal{S})$ is a well generated triangulated category, whereas the whole homotopy category is compactly generated precisely in the pure semisimple case \cite{MR2657376}. Thus purity marks a natural boundary between homotopy categories governed by small objects and homotopy categories which may admit no generating set at all.

The same circle of questions arises for $Q$-shaped homological algebra. Holm and J{\o}rgensen introduced the $Q$-shaped derived category $\QSD{Q}{A}$ as the common homotopy category of projective and injective model structures on $\lMod{Q,A}$, see \cite{HJ-JLMS}, and subsequently proved that it is compactly generated and described its compact and perfect objects \cite{HJ-TAMS}. In \cite{SV_Q_shaped_homotopy}, the authors enlarged this construction by replacing the abelian exact structure with objectwise exact structures induced from exact structures in  $\lMod{A}$. The aim of this work is to determine which of the resulting homotopy categories remain controlled by compact objects. This manuscript may be regarded as a direct sequel to \cite{SV_Q_shaped_homotopy}, nevertheless, we recall the essential definitions and results needed to make the exposition self-contained.

The scope of this work is twofold. First we want to find sufficient conditions that guarantee compact generation of the homotopy categories introduced in \cite{SV_Q_shaped_homotopy}, whereas, the second aim is to focus on various examples of such homotopy categories. Our first main result towards the first goal gives an abstract sufficient criterion. Let $\mathcal{E}$ be a weakly idempotent complete, locally finitely presented exact category with enough projectives, and assume that a set $\mathscr{G}$ of finitely presented Gorenstein projective objects induces the hereditary projective model structure via the techniques of \cite[Thm.~3.4]{SV_Q_shaped_homotopy}. For every $G\in\mathscr{G}$, choose the syzygies $\sigma^nG$, $n\in\mathbb{Z}$, in a totally acyclic complex, and put
\[
\mathcal{S}_{\mathscr{G}}
=
\{\sigma^nG\mid G\in\mathscr{G},\ n\in\mathbb{Z}\}.
\]
 The image of $\mathcal{S}_\mathscr{G}$ forms a set of compact objects in the stable category of cofibrant objects, while the cohomological description of the trivial objects shows that they generate the homotopy category. Moreover, in this case, we define perfect objects as the image of cofibrant finitely presented objects in the homotopy category $\operatorname{Ho}(\mathcal{E})$. We show that every perfect object is compact and provide sufficient conditions for the converse.
 
\begin{res*}[Theorem~A (Theorem \ref{thm:compact_generation} and Theorem \ref{thm:compact=perfect})]
Let $\mathcal{E}$ satisfy Setup~\ref{stp:loc_coh}, and let $\mathscr{G}$ be a set of finitely presented Gorenstein projective objects. Then $\mathcal{S}_{\mathscr{G}}$ is a set of compact generators of $\operatorname{Ho}(\mathcal{E})$. In particular,
\[
\operatorname{Ho}(\mathcal{E})
=
\operatorname{Loc}(\mathcal{S}_{\mathscr{G}})
\qquad\text{and}\qquad
\operatorname{Ho}(\mathcal{E})^c
=
\operatorname{Thick}(\mathcal{S}_{\mathscr{G}}).
\]
If moreover, $\mathcal{E}$ is a locally coherent exact category and $\operatorname{fp}(\mathcal{E})$ is Krull-Schmidt, then every perfect object is compact and vice versa, that is,
\[
\operatorname{Perf}(\mathcal{E})=\operatorname{Ho}(\mathcal{E})^c.
\]
\end{res*}

As an application, we show the analogous result for $Q$-shaped triangulated categories in the sense of \cite{HJ-JLMS} and \cite{SV_Q_shaped_homotopy}. The $Q$-shaped version follows naturally as a corollary of Theorem A. Let $\theta$ be the saturation, in the sense of \cite{SV_Q_shaped_homotopy}, of a set $\varrho$ of finitely presented $A$-modules, and endow $\lMod{Q,A}$ with the associated objectwise $\theta$-exact structure. In this setting the Gorenstein projective objects in question are the stalk representations $\Sq{q}(M)$. The special features of $Q$-shaped homology allow us to dispense with all their syzygies.

Assuming that $\theta$ is the saturation of a set $\varrho$ of finitely presented $A$-modules satisfying the Setup~\ref{stp:HJ+theta} we show that the 
\(
\mathcal{S}_{\varrho}
=
\{\Sq{q}(M)\mid q\in Q,\ M\in\varrho\}
\)
is a set of compact generators of $\QST{Q}{A}$. For $\varrho=\{A\}$ and the abelian exact structure it recovers the compact generation of $\QSD{Q}{A}$ from \cite[Thm.~D]{HJ-TAMS}, whereas, for larger choices of $\varrho$ it produces new compactly generated relative $Q$-shaped homotopy categories.

 The perfect objects in the $Q$-shaped derived category were introduced by Holm and J{\o}rgensen in \cite{HJ-TAMS}. They define an object to be perfect if it is isomorphic in the $Q$-shaped derived category to an object with finite support and values in finitely presented $A$-modules. In \emph{loc.cit.} they prove that perfect objects are always compact, and they formulate a conjecture that the converse also holds. Here, we extend their definition to the more general relative setup.

 \begin{res*}[Theorem~B (Theorem \ref{thm:hj_conj})]
    Assume that $Q$ has no cycles. Let also $\rho$ be a set of finitely presented $A$-modules, $\theta=\operatorname{Add}(\varrho)$ be its saturation, and assume that $\lMod{Q,A}$ is endowed with the objectwise $\theta$-exact structure. If $A$ is an artin $\Bbbk$-algebra, then 
    \[
    \QST{Q}{A}^{\text{perf}}=\QST{Q}{A}^c.
    \]
 \end{res*}

 For the special case that $\varrho=\{A\}$, the exact structure is the abelian one, $\QST{Q}{A}$ is the $Q$-shaped derived category $\QSD{Q}{A}$, and Theorem~B proves the Conjecture of Holm and J{\o}rgensen \cite{HJ-TAMS}, for artin algebras $A$.

The main homotopy categories that we study in this work admit both a projective and an injective model with the same weak equivalences. We call a projective model structure
\(
(\mathcal{C}_{\theta},\mathcal{W}_{\theta},\lMode{Q,A}{\theta})
\)
and an injective model structure
\(
(\lMode{Q,A}{\theta},\mathcal{W}_{\theta},\mathcal{F}_{\theta})
\)
a \emph{homotopically balanced pair} if they have the same trivial objects and weak equivalences, and hence the same homotopy category. This includes the projective and injective models for the derived, homotopy and pure derived categories, as well as their $Q$-shaped analogues. For such model structures, we give an explicit characterisation of the cofibrant and fibrant objects in a similar fashion as \cite[Thm.~E]{HJ-TAMS}. In particular, for a saturated class $\theta$ of $A$-modules inducing a homotopically balanced pair of model structures on $\lMode{Q,A}{\theta}$, we have
\begin{align*}
\mathcal{C}_{\theta}
&\subseteq
\{X\in\lMode{Q,A}{\theta}\mid X(q)\in\lPrj{A}^{\theta}
\text{ for every }q\in Q\},\\
\mathcal{F}_{\theta}
&\subseteq
\{Y\in\lMode{Q,A}{\theta}\mid Y(q)\in\lInj{A}^{\theta}
\text{ for every }q\in Q\}.
\end{align*}
If moreover $\lMode{A}{\theta}$ has finite relative global dimension, then both inclusions are equalities. In essence, one recovers the objectwise descriptions of the bifibrant objects in the abelian and split cases and extends them to arbitrary homotopically balanced relative exact structures.

The principal example is obtained by taking $\varrho={}_A\operatorname{mod}$, the set of finitely presented $A$-modules, whose saturation is the class of pure-projective modules. The corresponding objectwise exact structure on $\lMod{Q,A}$ is the objectwise pure exact structure. The projective cotorsion pair generated by the set
\[
\{\Sq{q}(M)\mid q\in Q,\ M\in{}_A\operatorname{mod}\}
\]
gives a projective model
$(\mathcal{C}_{\varrho},\mathcal{W}_{\varrho},\lMode{Q,A}{\varrho})$.
We also construct the complete hereditary cotorsion pair
$(\mathcal{W}_{\varrho},\mathcal{F}_{\varrho})$, and hence an injective model
$(\lMode{Q,A}{\varrho},\mathcal{W}_{\varrho},\mathcal{F}_{\varrho})$.
The two models are homotopically balanced and define the \emph{$Q$-shaped pure derived category}
\[
\frac{\mathcal{C}_{\varrho}}{\lPrj{Q,A}^{\varrho}}
\ \cong\
\QSPD{Q}{A}
\ \cong\
\frac{\mathcal{F}_{\varrho}}{\lInj{Q,A}^{\varrho}}.
\]
The $Q$-shaped version of Theorem~A, shows that this category is compactly generated by the stalks of finitely presented modules. It also fits in a recollement diagram of triangulated categories:
\[
\begin{tikzcd}[column sep=6em]
{\mathbf{K}_Q^{\mathrm{pac}}(A)}
& {\QSH{Q}{A}}
& {\QSPD{Q}{A}.}
\arrow["i_*"{description}, from=1-1, to=1-2]
\arrow["{i^*}"', shift right=3, from=1-2, to=1-1]
\arrow["{i^!}", shift left=3, from=1-2, to=1-1]
\arrow["j^*"{description}, from=1-2, to=1-3]
\arrow["{j_*}", shift left=3, from=1-3, to=1-2]
\arrow["{j_!}"', shift right=3, from=1-3, to=1-2]
\end{tikzcd}
\]
where $\mathbf{K}_Q^{\mathrm{pac}}(A)$ is the full subcategory of $\QSH{Q}{A}$ consisting of the pure-trivial objects. If $A$ has finite pure global dimension, then
\(
\operatorname{Im}(j_!)=\mathbf{K}_Q(\lPPrj{A})\,
\text{ and }\,
\operatorname{Im}(j_*)=\QSH{Q}{{}_A\operatorname{PInj}}.
\) When $Q=Q^{\mathrm{cpx}}$, one recovers the classical pure derived category and the familiar recollement of $\mathbf{K}(A)$ by the pure-acyclic complexes and $\mathbf{D}^{\mathrm{pure}}(A)$, see \cite{stovicek2014purity}.

The preceding results place all objectwise exact structures generated by finitely presented modules on the positive side of the compact generation boundary. The situation changes when the exact structure is generated by objects that are not finitely presented. To formulate the obstruction, for a regular cardinal $\kappa$ we denote by $\lPPrj{A}^{\kappa}$ the additive closure of the $\kappa$-presented $A$-modules. The relevant issue is whether the class $\theta$ is bounded by a single such cardinal, or equivalently whether it is contained in the additive closure of a set of modules.

\begin{res*}[Theorem~C (Theorem \ref{thm:negative_compact_generation})]
Let $\theta$ satisfy Setup~\ref{stp:HJ+theta}. If there is a regular cardinal $\kappa$ such that
\(
\theta\subseteq\lPPrj{A}^{\kappa},
\)
then $\QST{Q}{A}$ is well generated. In the special case $Q=Q^{\mathrm{cpx}}$, the following conditions are equivalent:
\begin{prt}
\item $\mathbf{K}^{\theta}(A)$ is well generated;
\item $\mathbf{K}^{\theta}(A)$ admits a generating set of objects;
\item there exists a regular cardinal $\kappa$ such that
$\theta\subseteq\lPPrj{A}^{\kappa}$.
\end{prt}
\end{res*}

 In particular, for complexes, if the relative projective class $\theta$ cannot be bounded by one regular cardinal, then $\mathbf{K}^{\theta}(A)$ has no generating set; it is therefore neither well generated nor compactly generated. 

The extremal special case is the split exact structure, for which $\theta=\lMod{A}$ and $\QST{Q}{A}=\QSH{Q}{A}$. Here the required boundedness is equivalent to pure semisimplicity. For $Q=Q^{\mathrm{cpx}}$, this specialises to the case found in \cite{MR2657376}. The $Q$-shaped result strengthens the positive conclusion from well-generation to compact generation and, for acyclic $Q$, gives a converse.

\begin{res*}[Theorem~D (Theorem \ref{thm:homotopy_cg})]
Let $A$ be a ring. If $A$ is pure semisimple, then the $Q$-shaped homotopy category $\QSH{Q}{A}$ is compactly generated. If $Q$ is acyclic, the converse also holds; equivalently,
\[
\QSH{Q}{A}\text{ is compactly generated}
\quad\Longleftrightarrow\quad
A\text{ is pure semisimple}.
\]
\end{res*}

Overall in this work, we identify the objectwise pure exact structure as a threshold for compact generation. At and above purity, finitely presented objects yield compact generators, whereas at the split endpoint compact generation can occur only under the exceptional condition that every module is pure-projective. {\color{black} We conclude this work with an application where we consider an exact structure that differs from the ones discussed previously, and investigate whether the associated homotopy category is compactly or well generated.}

\subsection*{Structure of the paper}

Section~\ref{sec:compact-preliminaries} recalls the background on locally finitely presented additive categories, locally coherent exact structures, and compactly and well generated triangulated categories.

In Section~\ref{sec:compact}, we establish the abstract compact generation criterion: finitely presented cofibrant objects become compact in the homotopy category, and the shifts of the prescribed Gorenstein projective objects form a set of compact generators.

Section~\ref{sec:Qompact} applies this criterion to objectwise $\theta$-exact structures on ${}_{Q,A}\operatorname{Mod}$. We construct the projective and injective models of the $Q$-shaped pure derived category and prove it is compactly generated, study homotopically balanced pairs of model structures, and obtain objectwise descriptions of cofibrant and fibrant objects under finite relative global dimension. 

In Section~\ref{sec:negative_compact_gen}, we investigate compact and well generation beyond exact structures generated by finitely presented modules. We study the relation between compact generation of the $Q$-shaped homotopy category and pure semisimplicity, establish a sufficient cardinal boundedness criterion for well generation, and prove the corresponding converse for ordinary complexes. The section concludes with examples arising from the finitely split exact structure.

Lastly, Section~\ref{sec:perfect} develops the theory of perfect objects. Under local coherence and the existence of enough projectives in $\operatorname{fp}(\mathcal E)$, we show that the finitely presented cofibrant objects form a Frobenius category whose stable category identifies with the perfect objects. Its idempotent completion recovers all compact objects, and a Krull-Schmidt hypothesis ensures that this completion adds no new objects. We then specialise to acyclic $Q$, where for left artinian rings, we establish equality between compact and perfect objects, proving the Holm-J{\o}rgensen conjecture in this case.

\section{Preliminaries}\label{sec:compact-preliminaries}

The compact generation results in this work combine two kinds of finiteness. On the exact-categorical side, finite presentation controls the behaviour of objects under directed colimits and ensures that the relevant exact structure is already determined by its finitely presented part. On the triangulated side, compactness controls morphisms into arbitrary coproducts. We recall the two frameworks separately and then indicate the bridge between them. Our conventions follow Crawley--Boevey for locally finitely presented additive categories \cite{WCB94}, Positselski for locally coherent exact categories \cite{Positselski_loc_coh}, and Neeman and Krause for compactly generated triangulated categories \cite{MR1308405,Nee,Krause_2010}.

\subsection{Locally coherent exact categories}
\label{subsec:locally-coherent-categories}

Let $\mathcal{E}$ be an additive category admitting directed colimits. Recall that an object $F$ of $\mathcal{E}$ is \textbf{finitely presented} (or sometimes also \textbf{compact}) if the canonical map
\[
 \varinjlim_{i\in I}\Hom{\mathcal{E}}{F}{X_i}
 \longrightarrow
 \Hom{\mathcal{E}}{F}{\varinjlim_{i\in I}X_i}
\]
is an isomorphism for every directed system $(X_i)_{i\in I}$. We write $\mathsf{fp}(\mathcal{E})$ for the full subcategory of finitely presented objects.

Following Crawley--Boevey \cite{WCB94}, the category $\mathcal{E}$ is called \textbf{locally finitely presented} if $\mathsf{fp}(\mathcal{E})$ is skeletally small and every object of $\mathcal{E}$ is a directed colimit of objects in $\mathsf{fp}(\mathcal{E})$. In the modern terminology of accessible categories this condition is also called \emph{finite accessibility}; in particular, it does not require the existence of all small colimits. This distinction is useful for exact categories such as categories of flat modules.

The restricted Yoneda functor gives a concrete description of these categories. More precisely, Crawley--Boevey's representation theorem \cite[Thm.~1.4]{WCB94} identifies $\mathcal{E}$ with the category of flat additive functors
\[
 \mathcal{E}
 \ \cong\
 \operatorname{Flat}\bigl(\mathsf{fp}(\mathcal{E})^{\mathrm{op}},\mathsf{Ab}\bigr),
 \qquad
 X\longmapsto \Hom{\mathcal{E}}{-}{X}\big|_{\mathsf{fp}(\mathcal{E})}.
\]
Conversely, if $\mathcal{C}$ is a skeletally small additive category in which idempotents split, then the category of flat functors $\operatorname{Flat}(\mathcal{C}^{\mathrm{op}},\mathsf{Ab})$ is locally finitely presented and its finitely presented objects identify with $\mathcal{C}$. Thus a locally finitely presented additive category can be regarded as the category of ind-objects of its finitely presented part.

Suppose now that $\mathcal{A}$ is a locally finitely presented abelian category. The category $\mathcal{A}$ is called \textbf{locally coherent} if $\mathsf{fp}(\mathcal{A})$ is an abelian subcategory. Since finitely presented objects are closed under finite colimits and extensions, this is equivalent to requiring that $\mathsf{fp}(\mathcal{A})$ be closed under kernels in $\mathcal{A}$, or equivalently that the kernel of every epimorphism between finitely presented objects be finitely presented. For example, the category $\lMod{A}$ is locally coherent precisely when the ring $A$ is left coherent; equivalently, the finitely presented left $A$-modules form an abelian category.

For an arbitrary locally finitely presented additive category there is a canonical exact structure which does not require an ambient abelian category. A composable pair
\[
 X'\longrightarrow X\longrightarrow X''
\]
is a \textbf{pure conflation} if applying $\Hom{\mathcal{E}}{F}{-}$ produces a short exact sequence of abelian groups for every $F\in\mathsf{fp}(\mathcal{E})$. Equivalently, the pure conflations are the directed colimits of split short exact sequences \cite[Sec.~3]{WCB94}. The pure exact structure is the smallest locally coherent exact structure in the following sense.

\begin{dfn}[\!{\cite[Sec.~1]{Positselski_loc_coh}}]
\label{dfn:locally-coherent-exact}
Let $\mathcal{E}$ be a locally finitely presented additive category endowed with an exact structure. We call $\mathcal{E}$ a \textbf{locally coherent exact category} if its conflations are precisely the directed colimits of conflations
\[
 F'\rightarrowtail F\twoheadrightarrow F''
\]
whose three terms belong to $\mathsf{fp}(\mathcal{E})$. Here the colimits are taken in the category of three-term complexes in $\mathcal{E}$.
\end{dfn}

Thus local coherence has two components: the exact structure is closed under directed colimits, and every conflation can be recovered from conflations between finitely presented objects. The following properties will be used repeatedly below. 

\subsection{Well-generated triangulated categories}
\label{subsec:kappa-presentability-well-generation}

The finiteness conditions above admit a cardinal-valued refinement.  Recall that an infinite cardinal $\kappa$ is \textbf{regular} if a union of fewer than $\kappa$ sets, each of cardinality less than $\kappa$, again has cardinality less than $\kappa$.  A small category $I$ is \textbf{$\kappa$-filtered} if every diagram in $I$ involving fewer than $\kappa$ objects and morphisms admits a cocone.  For a partially ordered set, this says that every subset of cardinality less than $\kappa$ has an upper bound.

Let $\mathcal{E}$ be a category admitting $\kappa$-filtered colimits.  An object $F$ of $\mathcal{E}$ is \textbf{$\kappa$-presentable} if the functor $\Hom{\mathcal{E}}{F}{-}$ preserves $\kappa$-filtered colimits, explicitly, the canonical map
\[
 \varinjlim_{i\in I}\Hom{\mathcal{E}}{F}{X_i}
 \longrightarrow
 \Hom{\mathcal{E}}{F}{\varinjlim_{i\in I}X_i}
\]
is an isomorphism for every $\kappa$-filtered diagram $(X_i)_{i\in I}$.  The category $\mathcal{E}$ is \textbf{$\kappa$-accessible} if it has $\kappa$-filtered colimits and there is a set of $\kappa$-presentable objects such that every object of $\mathcal{E}$ is a $\kappa$-filtered colimit of objects in that set.  If, in addition, $\mathcal{E}$ is cocomplete, then it is \textbf{locally $\kappa$-presentable} \cite[Chap.~1]{AdamekRosicky}.

For $\kappa=\aleph_0$ one recovers finitely presentable objects and finitely accessible categories.  In a module category, the categorical terminology agrees with the usual one. An $A$-module $M$ is $\kappa$-presentable if and only if there is an exact sequence
\[
 A^{(J)}\longrightarrow A^{(I)}\longrightarrow M\longrightarrow 0
 \qquad\text{with}\qquad |I|,|J|<\kappa.
\]
Such a module is also called \textbf{$\kappa$-presented}. {\color{black} An object $X$ is called \textbf{$\kappa$-generated}, if it admits a set of generators whose cardinality is strictly less than $\kappa$.}  Every locally presentable category is locally $\kappa$-presentable for some regular cardinal $\kappa$, and every set of its objects consists of $\lambda$-presentable objects for some sufficiently large regular cardinal $\lambda$.  This uniform bound on a set of objects will be used in the discussion of $\kappa$-pure-projective modules later in this manuscript.

There is a parallel cardinal filtration on a triangulated category, but the relevant notion is smallness with respect to coproducts rather than presentability with respect to coproducts.  Let $\mathcal{T}$ be a triangulated category with set-indexed coproducts and let $\kappa$ be regular.  An object $G$ of $\mathcal{T}$ is \textbf{$\kappa$-small} if every morphism
\[
 G\longrightarrow\coprod_{i\in I}X_i
\]
factors through $\coprod_{i\in J}X_i$ for some subset $J\subseteq I$ with $|J|<\kappa$.  A full subcategory is \textbf{$\kappa$-localising} if it is thick and closed under coproducts with fewer than $\kappa$ summands; we write $\operatorname{Loc}_{\kappa}(\mathcal{S})$ for the smallest such subcategory containing a class $\mathcal{S}$.

To state well-generation, one also needs a compatibility condition between the generators and arbitrary coproducts.  A set $\mathcal{G}$ \textbf{perfectly generates} $\mathcal{T}$ if $\operatorname{Loc}(\mathcal{G})=\mathcal{T}$ and the following condition holds: for every family of morphisms $(f_i\colon X_i\to Y_i)_{i\in I}$ such that
\[
 \Hom{\mathcal{T}}{G}{X_i}\longrightarrow
 \Hom{\mathcal{T}}{G}{Y_i}
\]
is surjective for every $G\in\mathcal{G}$ and every $i\in I$, the induced map
\[
 \Hom{\mathcal{T}}{G}{\coprod_{i\in I}X_i}
 \longrightarrow
 \Hom{\mathcal{T}}{G}{\coprod_{i\in I}Y_i}
\]
is surjective for every $G\in\mathcal{G}$.

\begin{dfn}[{\cite[Secs.~5.1 and 6.1--6.3]{Krause_2010}}]
\label{dfn:well_generated_triangulated}
The category $\mathcal{T}$ is \textbf{$\kappa$-well generated} if it is perfectly generated by a set of $\kappa$-small objects.  It is \textbf{well generated} if it is $\kappa$-well generated for some regular cardinal $\kappa$.
\end{dfn}

Suppose that $\mathcal{T}$ is $\kappa$-well generated by $\mathcal{G}$.  For every regular cardinal $\lambda\geq\kappa$, put
\[
 \mathcal{T}^{\lambda}
 :=
 \operatorname{Loc}_{\lambda}(\mathcal{G}).
\]
This subcategory is independent of the chosen set of perfect generators, and its objects are called \textbf{$\lambda$-compact}.  They form an essentially small subcategory of $\mathcal{T}$, and
\[
 \mathcal{T}
 =
 \bigcup_{\substack{\lambda\geq\kappa\\
                    \lambda\text{ regular}}}
 \mathcal{T}^{\lambda}.
\]
Thus well-generation expresses that $\mathcal{T}$ is controlled by a set of objects of uniformly bounded ``cardinal size''.  Well generated triangulated categories satisfy Brown representability, and localising subcategories generated by sets, as well as the corresponding Verdier quotients, remain well generated \cite[Secs.~6--7]{Krause_2010}.

It is important not to identify the two cardinal notions without an argument: $\kappa$-presentability concerns $\kappa$-filtered colimits in an underlying category, whereas $\kappa$-smallness concerns factorisation of morphisms through subcoproducts in a triangulated category.  For homotopy categories arising from stable combinatorial model categories, accessibility provides the bridge after possibly enlarging the regular cardinal; Rosick\'y's theorem shows that such homotopy categories are well generated \cite[Thm.~4.9]{Rosicky_Brown_rep_in_hmtp}.  When $\kappa=\aleph_0$, however, the comparison becomes elementary: $\aleph_0$-small objects are precisely compact objects, and the perfectness condition is automatic for a set of compact objects.  Consequently,
\[
 \aleph_0\text{-well generated}
 \quad\Longleftrightarrow\quad
 \text{compactly generated},
 \qquad
 \mathcal{T}^{\aleph_0}=\mathcal{T}^{c}.
\]

Localising generation permits arbitrary coproducts, whereas thick generation only uses finite triangulated constructions and direct summands. Neeman's theorem identifies precisely how the two notions meet on compact objects.

\begin{thm}[Ravenel--Neeman\, {\cite[Thm.~2.1]{MR1308405}}]
\label{thm:ravenel-neeman}
Let $\mathcal{S}$ be a set of compact objects in a triangulated category $\mathcal{T}$ with set-indexed coproducts, and put
\(
 \mathcal{L}=\operatorname{Loc}(\mathcal{S}).
\)
Then $\mathcal{L}$ is compactly generated and
\[
 \mathcal{L}^{c}
 =\mathcal{L}\cap\mathcal{T}^{c}
 =\operatorname{Thick}(\mathcal{S}).
\]
In particular, if $\mathcal{T}=\operatorname{Loc}(\mathcal{S})$, then
\[
 \mathcal{T}^{c}=\operatorname{Thick}(\mathcal{S}).
\]
\end{thm}

Compact generation also places $\mathcal{T}$ within the scope of \textbf{Brown representability}, that is, every contravariant cohomological functor
\[
 H\colon\mathcal{T}^{\mathrm{op}}\longrightarrow\mathsf{Ab}
\]
which takes coproducts to products is representable \cite[Thm.~3.1]{MR1308405}. Among other consequences, an exact coproduct-preserving functor from a compactly generated triangulated category admits a right adjoint. In this work, however, the main consequence we need is the final equality in Theorem~\ref{thm:ravenel-neeman}.

The standard example of a compactly generated triangulated category is the unbounded derived category $\mathbf{D}(A)$ of a ring $A$. The complex $A$ concentrated in degree zero is compact and
\[
 \Hom{\mathbf{D}(A)}{A}{\Sigma^nX}\cong H^n(X),
\]
so $A$ generates $\mathbf{D}(A)$. Hence $\mathbf{D}(A)$ is compactly generated and Theorem~\ref{thm:ravenel-neeman} yields
\[
 \mathbf{D}(A)^c
 =\operatorname{Thick}(A)
 =\mathbf{K}^{b}(\lprj{A}),
\]
the subcategory of perfect complexes.

\subsection{Exact model structures from Gorenstein projective objects}
\label{subsec:SV-models}

We recall the construction from \cite{SV_Q_shaped_homotopy}
which produces the homotopy categories studied below. Let
$\mathcal E$ be a weakly idempotent complete exact category
with enough projectives. For a class $\mathcal X$ of objects,
write
\[
\mathcal X^\perp
 =
 \{Y\in\mathcal E\mid
   \Ext{\mathcal E}{1}{X}{Y}=0
   \text{ for every }X\in\mathcal X\},
\]
and define ${}^\perp\mathcal X$ dually. All Ext groups are
computed with respect to the specified exact structure.

A \textbf{cotorsion pair} is a pair $(\mathcal C,\mathcal W)$
satisfying $\mathcal C={}^\perp\mathcal W$ and
$\mathcal W=\mathcal C^\perp$. It is \textbf{complete} if
every $X\in\mathcal E$ admits conflations
\[
W\rightarrowtail C\twoheadrightarrow X,
\qquad
X\rightarrowtail W'\twoheadrightarrow C',
\]
with $C,C'\in\mathcal C$ and $W,W'\in\mathcal W$.
It is \textbf{hereditary} if
$\Ext{\mathcal E}{i}{C}{W}=0$ for every $i>0$,
$C\in\mathcal C$, and $W\in\mathcal W$.
The cotorsion pair \textbf{generated by} $\mathcal X$ is
$({}^\perp(\mathcal X^\perp),\mathcal X^\perp)$.

An acyclic complex $P$ of projective objects is
\textbf{totally acyclic} if $\Hom{\mathcal E}{P}{P'}$
is acyclic for every projective object $P'$.
An object is \textbf{Gorenstein projective} if it occurs
as a cycle in such a complex. For each object $G$ in a
class $\mathscr G$ of Gorenstein projectives, choose a
totally acyclic complex $P_G$ with
$\operatorname{Z}^0P_G\cong G$, and put
$\sigma^nG=\operatorname{Z}^nP_G$. Consider
\[
\mathcal S_{\mathscr G}
 =\{\sigma^nG\mid G\in\mathscr G,\ n\in\mathbb Z\},
\qquad
\mathcal W_{\mathscr G}=\mathcal S_{\mathscr G}^{\perp},
\qquad
\mathcal C_{\mathscr G}={}^\perp\mathcal W_{\mathscr G}.
\]

If $(\mathcal C_{\mathscr G},\mathcal W_{\mathscr G})$
is complete, then \cite[Thm.~3.4]{SV_Q_shaped_homotopy}
provides a hereditary projective exact model structure
whose Hovey triple is
\[
(\mathcal C_{\mathscr G},\mathcal W_{\mathscr G},\mathcal E).
\]
Thus $\mathcal C_{\mathscr G}$ is the class of cofibrant
objects, $\mathcal W_{\mathscr G}$ is the class of objects
weakly equivalent to zero, and every object is fibrant.
Moreover,
\[
\mathcal C_{\mathscr G}\cap\mathcal W_{\mathscr G}
 =\Prj{\mathcal E},
\qquad
\operatorname{Ho}(\mathcal E)
 \simeq\underline{\mathcal C}_{\mathscr G}.
\]
Here $\mathcal C_{\mathscr G}$, with its inherited exact
structure, is Frobenius: it has enough projectives and
injectives, and these coincide with $\Prj{\mathcal E}$.
Its stable category has the same objects and morphisms
modulo those factoring through $\Prj{\mathcal E}$.
This description allows us to study the homotopy category
using representatives and conflations in $\mathcal E$.

Completeness is automatic when $\mathscr G$ is a set and
$\mathcal E$ is \textbf{efficient}
\cite[Prop.~3.3]{SV_Q_shaped_homotopy}.
We use efficiency in the sense of
\cite[Def.~2.6]{SaorinStovicek}. Since $\mathcal E$ has
enough projectives, this amounts to requiring that
arbitrary transfinite compositions of inflations exist
and remain inflations, and that every object is small
relative to inflations
\cite[Prop.~2.7]{SaorinStovicek}. In particular, efficient exact categories have exact
coproducts \cite[Lem.~1.4]{SaorinStovicek}.

\section{Compactly generated homotopy categories}\label{sec:compact}
In this section we provide a sufficient condition for homotopy categories induced by a set of Gorenstein projective objects as in \cite[Thm.~3.4]{SV_Q_shaped_homotopy}, to be compactly generated triangulated categories.

\begin{stp}\label{stp:loc_coh}
    Let $\mathcal{E}$ be an weakly idempotent complete, locally finitely presented additive category endowed with an exact structure such that:
    \begin{rqm}
        \item it has enough projectives;
        \item admits a projective hereditary model structure as in \cite[Thm.~3.4]{SV_Q_shaped_homotopy}; see the subsection \ref{subsec:SV-models}.
    \end{rqm}
\end{stp}

 The following results are known in the abelian case from \cite[Lem.~4.4~\&~Thm.~4.5]{HJ-TAMS}. We include a proof for the exact case for completeness.

\begin{lem}\label{lem:compacts}
    Let be $\mathcal{E}$ be as in the \stpref{loc_coh} and let $C$ be a finitely presented cofibrant object in $\mathcal{E}$. Then, the image of $C$ in the homotopy category $\operatorname{Ho}(\mathcal{E})$ is a compact object.
\end{lem}

\begin{proof}
    Let $(\mathcal{C}_\mathscr{G},\mathcal{W}_\mathscr{G},\mathcal{E})$ be the projective hereditary Hovey triple in $\mathcal{E}$, given by \cite[Thm.~3.4]{SV_Q_shaped_homotopy}. We show that if $C$ is compact (or, as stated above, finitely presented) in $\mathcal{E}$, then the image of $C$ in $\operatorname{Ho}(\mathcal{E})$ (which we abbreviate as $C$ as well) is compact.

    Let $\{X_i\}_{i\in I}$ be a (set-indexed) family of objects in $\mathcal{C}_\mathscr{G}$. Since $C$ is compact in $\mathcal{E}$, there is a canonical isomorphism
    \begin{equation}\label{eq:compact}
    \vartheta\colon \bigoplus_{i\in I}\Hom{\mathcal{E}}{C}{X_i}\stackrel{\cong}{\longrightarrow}\Hom{\mathcal{E}}{C}{\bigoplus_{i\in I}X_i}\,,
    \end{equation}
    mapping a family $(a_i\colon C\to X_i)_{i\in I}$ of morphisms to the sum $\Sigma_{i\in I}\iota_ia_i$, where $\iota_i$ is the canonical inclusion of the $i$-th summand to the direct sum. 
    
    Recall that, in an exact model category, the class of cofibrant objects $\mathcal{C}$ is closed under coproducts, whereas the class of fibrant objects $\mathcal{F}$ is closed under products. In our case, both classes are closed under coproducts, since the class of fibrant objects is $\mathcal{E}$. In particular, the intersection $\mathcal{C}_\mathscr{G}\cap\mathcal{W}_\mathscr{G}$ is also closed under coproducts since it is equal to the class of projective objects $\Prj{\mathcal{E}}$. It follows that, the functor $\mathcal{E}\longrightarrow\operatorname{Ho}(\mathcal{E})$ preserves coproducts, since $\operatorname{Ho}(\mathcal{E})\,\cong\,\underline{\mathcal{C}}_\mathscr{G}$, and thus, there is a well defined map
    \[
    \underline{\vartheta}\colon\bigoplus_{i\in I}\underline{\operatorname{Hom}}_\mathcal{{E}}(C,X_i)\longrightarrow\underline{\operatorname{Hom}}_\mathcal{E}(C,\bigoplus_{i\in I}X_i)
    \]
    in $\underline{\mathcal{C}}_\mathscr{G}$. We claim that $\underline{\vartheta}$ is a bijection. For this purpose it suffices to show that the bijection in \eqref{compact} restricts to a bijection
    \[
    \vartheta|_\mathcal{I}\colon \bigoplus_{i\in I}\mathcal{I}(C,X_i)\longrightarrow\mathcal{I}(C,\bigoplus_{i\in I}X_i)\,,
    \]
    where $\mathcal{I}(-,-)$ denotes the ideal of morphisms in $\mathcal{E}$ that factor through a projective-injective object. 
First we argue that $\vartheta|_\mathcal{I}$ maps the subgroup $\bigoplus_{i \in I}\mathcal{I}(C,X_i)$ to the subgroup $\mathcal{I}(C,\bigoplus_{i \in I}X_i)$. Let $\alpha = (\alpha_i)_{i \in I}$ be an element in $\bigoplus_{i \in I}\mathcal{I}(C,X_i)$, that is, the set $I_0 = \{i \in I \,|\ \alpha_i \neq 0\}$ is finite and each $\alpha_i \colon C \to X_i$ admits a factorization
\begin{equation*}
 \xymatrix@!=1pc{
 C \ar@{.>}[dr]_-{\varphi_i} \ar[rr]^-{\alpha_i} & & X_i
 \\
 & W_i \ar@{.>}[ur]_-{\psi_i} &
 }
\end{equation*}
through some object $W_i \in \Prj{\mathcal{E}}$. Note that for $i \notin I_0$ we can take $W_i=0$ as $\alpha_i=0$. Let $\varepsilon_{\!j} \colon W_{\!j} \to \bigoplus_{i \in I}W_i$ be the canonical injection. The commutative diagram above yields commutativity of the diagram
\begin{equation*}
 \xymatrix@!=1pc{
 C \ar[dr]_(0.4){\sum_{i \in I}\varepsilon_i\mspace{1mu}\varphi_i} \ar[rr]^-{\sum_{i \in I}\iota_i\mspace{1mu}\alpha_i} & & \bigoplus_{i \in I}X_i
 \\
 & \bigoplus_{i \in I}W_i
 \ar[ur]_(0.55){\bigoplus_{i \in I}\psi_i} &
 }
\end{equation*}
so the morphism $\vartheta|_\mathcal{I}(\alpha) = \sum_{i \in I}\iota_i\,\alpha_i$ factors through the object $W=\bigoplus_{i \in I}W_i \in \Prj{\mathcal{E}}$ (note that $W$ is equal to the finite direct sum $\bigoplus_{i \in I_0}W_i$). Thus $\vartheta|_\mathcal{I}(\alpha)$ is in $\mathcal{I}(C,\bigoplus_{i \in I}X_i)$.

It remains to argue that $\vartheta|_\mathcal{I}^{-1}$, that is, the inverse of the homomorphism from \eqref{compact}, maps $\mathcal{I}(C,\bigoplus_{i \in I}X_i)$ to $\bigoplus_{i \in I}\mathcal{I}(C,X_i)$. Let \smash{$\pi_{\!j} \colon \bigoplus_{i \in I}X_i \to X_{\!j}$} be the canonical projection given by \smash{$\pi_{\!j}\,\iota_{\!j} = \mathrm{id}_{X_{\!j}}$} and \smash{$\pi_{\!j}\,\iota_i = 0$} if $i \neq j$, and note that $\vartheta|_\mathcal{I}^{-1}$ maps a morphism $\beta \colon C \to \bigoplus_{i \in I}X_i$ to the family $(\pi_i\,\beta)_{i \in I}$. (In general, this family only belongs to the product $\prod_{i \in I}\Hom{\mathcal{E}}{C}{X_i}$, but as $\vartheta$ is assumed to be bijective, it actually belongs to the coproduct $\bigoplus_{i \in I}\Hom{\mathcal{E}}{C}{X_i}$.) Now, let $\beta$ be an element in $\mathcal{I}(C,\bigoplus_{i \in I}X_i)$, that is, $\beta \colon C \to \bigoplus_{i \in I}X_i$ is a morphism in $\mathcal{E}$ that factors through some object $W$ in $\Prj{\mathcal{E}}$. Evidently each $\pi_i\,\beta$ also factors through $W$, so $\vartheta|_\mathcal{I}^{-1}(\beta) = (\pi_i\,\beta)_{i \in I}$ belongs to $\bigoplus_{i \in I}\mathcal{I}(C,X_i)$, as desired.   
\end{proof}

\begin{thm}\label{thm:compact_generation}
     Let be $\mathcal{E}$ be as in the \stpref{loc_coh}. If $\mathscr{G}$ is a set of finitely presented Gorenstein projective objects, then the homotopy category $\operatorname{Ho}(\mathcal{E})$ is compactly generated, that is,
     \[
     \operatorname{Ho}(\mathcal{E})\,=\,\operatorname{Loc}(\mathcal{S}_\mathscr{G})\quad\text{and}\quad\operatorname{Ho}(\mathcal{E})^c\,=\,\operatorname{Thick}(\mathcal{S}_\mathscr{G}).
     \]
\end{thm}

\begin{proof}
    By \cite[Dfn.~1.7]{MR1308405} to prove that $\operatorname{Ho}(\mathcal{E})$ is compactly generated we have to show that there exist a set $\mathcal{S}$ of compact objects such that, for every object $S$ in $\mathcal{S}$ and every object $X$ in $\operatorname{Ho}(\mathcal{E})$ the following holds:
    \[
    \Hom{\operatorname{Ho}(\mathcal{E})}{S}{X}=0 \implies X=0.
    \]
    
    Our candidate set of compacts is $\mathcal{S}_\mathscr{G}=\{\sigma^nG\,|\, G\in\mathscr{G},\,n\in\mathbb{Z}\}$. First notice that $\mathcal{S}_\mathscr{G}\subset\mathcal{C}_\mathscr{G}$ by construction since $\mathcal{S}_\mathscr{G}\subset {}^\perp(\mathcal{S}_\mathscr{G}^\perp)=\mathcal{C}_\mathscr{G}$. {\color{black}Notice that, \emph{a priori}, the syzygies of the objects $G$ in $\mathscr{G}$, should not be finitely presented. However, for a given $G$ in $\mathscr{G}$, there exist a totally acyclic projective resolution
    \small
 \[\begin{tikzcd}
	{P^\bullet:} & \cdots & {P^{n-1}} && {P^n} && {P^{n+1}} & \cdots \\
	&&& {Z^nP^\bullet=\sigma^nG} && {Z^{n+1}P^\bullet=\sigma^{n+1}G}
	\arrow[from=1-2, to=1-3]
	\arrow[from=1-3, to=1-5]
	\arrow[from=1-3, to=2-4]
	\arrow[from=1-5, to=1-7]
	\arrow[from=1-5, to=2-6]
	\arrow[from=1-7, to=1-8]
	\arrow[from=2-4, to=1-5]
	\arrow[from=2-6, to=1-7]
\end{tikzcd}\,,\]
\normalsize
and thus, for every $n\in\mathbb{Z}$, there exist conflations
\[
0\longrightarrow Z^nP^\bullet\longrightarrow P^n\longrightarrow Z^{n+1}P^\bullet\longrightarrow0,
\]
whose images in the homotopy category $\operatorname{Ho}(\mathcal{E})$ gives $Z^{n+1}P^\bullet\,\cong\,\Sigma Z^n P^\bullet\,\cong\, \Sigma^{n+1}G$. For any $G$ in $\mathscr{G}$ and any $X$ in $\operatorname{Ho}(\mathcal{E})$, one can compute easily that,
\[
\Hom{\operatorname{Ho}(\mathcal{E})}{\sigma^nG}{X}\,\cong\,\Hom{\operatorname{Ho}(\mathcal{E})}{Z^nP^\bullet}{X}\,\cong\, \Hom{\operatorname{Ho}(\mathcal{E})}{G}{\Sigma^{-(n+1)}X}\,.
\]
Notice that since $G$ is compact by Lemma \ref{lem:compacts}, the images of the syzygies $\sigma^nG$ are compact as well in the homotopy category.
} It remains to show that they also generate $\operatorname{Ho}(\mathcal{E})$ as a triangulated category.

    Let $G$ be an object in $\mathcal{S}_\mathscr{G}$ and $Y$ an object in $\mathcal{E}$. One has,
    \[
    \Hom{\operatorname{Ho}(\mathcal{E})}{G}{Y}\,\cong\, \Hom{\operatorname{Ho}(\mathcal{E})}{G}{\Sigma(\Sigma^{-1}Y)}\,\cong\, \Ext{\mathcal{E}}{1}{G}{\textbf{c}(\Sigma^{-1}Y)}\,,
    \]
    where $\textbf{c}(\Sigma^{-1}Y)$ is the cofibrant (in fact bifibrant) replacement of the object $\Sigma^{-1}Y$.  Assuming that $\Hom{\operatorname{Ho}(\mathcal{E})}{G}{Y}$ vanishes for every $G$ in $\mathcal{S}_\mathscr{G}$, we get that $\Ext{\mathcal{E}}{1}{G}{\textbf{c}(\Sigma^{-1}Y)}=0$, and thus $\textbf{c}(\Sigma^{-1}Y)$ is belongs to the class $\mathcal{W}_\mathscr{G}$. Consequently, in $\operatorname{Ho}(\mathcal{E})$ the object $\textbf{c}(\Sigma^{-1}Y)$ is isomorphic to the zero object and so does $\Sigma^{-1}Y$ which is weakly equivalent to its cofibrant replacement. This in turn implies that $Y=0$ in $\operatorname{Ho}(\mathcal{E})$ and $\operatorname{Ho}(\mathcal{E})=\operatorname{Loc}(\mathcal{S}_\mathscr{G})$. The last implication follows from Ravenel-Neeman's Theorem \ref{thm:ravenel-neeman}.
\end{proof}

Notice that exact structures on a given abelian category $\mathcal{E}$ form a lattice with the partial order on the class of exact sequences, given by inclusion of classes of admissible short exact sequences. The maximal element in this exact structure is the abelian exact structure, whereas the minimal element is the split exact structure. Positselski in \cite[Sec.~3]{Positselski_loc_coh} proves that locally coherent exact structures form a sublattice of the lattice of all exact structure whose minimal element is the pure exact structure and the maximal element is the maximal locally coherent exact structure on $\mathcal{E}$. For example, if $\mathcal{E}=\operatorname{Mod}{R}$ is the category of (right) modules over a ring $R$, the maximal locally coherent exact structure coincides with the abelian exact structure if and only if, $R$ is a (right) coherent ring. The result of Theorem \ref{thm:compact_generation} suggest that this sublattice of locally coherent exact structures control in a certain degree the compact generation of the homotopy categories that one can form based on such exact structures. We make this idea a bit more concrete by specialising in the following sections in derived categories and their $Q$-shaped analogues.

\section{Compact generation in the
  \texorpdfstring{$Q$}{Q}-shaped world}
\label{sec:Qompact}

We apply the preceding construction to categories of
$Q$-shaped representations. Throughout the $Q$-shaped setting, we adopt the Setup of \cite[Stp.~2.5]{HJ-JLMS}. Thus $\Bbbk$ is a commutative noetherian hereditary ring, $A$ is a $\Bbbk$-algebra, and $Q$ is a small $\Bbbk$-linear category satisfying the following conditions. It is \emph{Hom-finite}, meaning that each $Q(p,q)$ is finitely generated projective over $\Bbbk$, and
\emph{locally bounded}, meaning that
\[
\mathsf{N}_-(q)=\{p\in Q\mid Q(p,q)\ne0\},
\qquad
\mathsf{N}_+(q)=\{p\in Q\mid Q(q,p)\ne0\}
\]
are finite for every $q$.
There is a \emph{Serre functor}, namely a
$\Bbbk$-linear autoequivalence $\mathbb S$ with natural
isomorphisms
\[
Q(p,q)\cong\Hom{\Bbbk}{Q(q,\mathbb S p)}{\Bbbk}.
\]
We also assume the \emph{strong retraction property}:
there is an ideal $\mathfrak r$ satisfying
\[
Q(q,q)=\Bbbk\operatorname{id}_q\oplus\mathfrak r(q,q),
\qquad
\mathfrak r(p,q)=Q(p,q)\quad(p\ne q).
\]
This ideal is called the \emph{pseudo-radical}, and is
assumed nilpotent. When we additionally require $Q$ to
be \emph{acyclic}, we mean that there is no sequence
\[
q_0\xrightarrow{a_1}q_1\longrightarrow\cdots
\xrightarrow{a_n}q_n=q_0,\qquad n\geq1,
\]
with every $a_i$ a nonzero morphism in $\mathfrak r$,
regardless of whether their composite vanishes
\cite[Def.~5.13]{HJ-TAMS}. The category of \textbf{$Q$-shaped $A$-modules} is defined as
\[
\lMod{Q,A}
 =\operatorname{Fun}_{\Bbbk}(Q,\lMod{A}).
\]
Its morphisms are natural transformations, and its
limits, colimits, and abelian exact sequences are
computed objectwise. The \textbf{support} of $X$ is defined by
$\supp{X}=\{q\in Q\mid X(q)\ne0\}$.
For each $q\in Q$, the evaluation functor belongs to an adjoint triple
$(F_q,\, E_q,\,G_q)$, where
\[
F_q(M)=Q(q,-)\otimes_{\Bbbk}M,
\qquad
E_q(X)=X(q),
\qquad
G_q(M)=\Hom{\Bbbk}{Q(-,q)}{M}.
\]
Serre duality gives $G_q\cong F_{\mathbb S^{-1}q}$.
The \textbf{stalk functor} $S_q(M)$ has value $M$ at
$q$, zero elsewhere, and zero $\mathfrak r$-action.
It belongs to another adjoint triple
$(C_q,\, S_q,\, K_q)$, with
\[
C_q(X)
 =X(q)\Big/\sum_{\substack{p\in Q\\a\in\mathfrak r(p,q)}}
             \operatorname{Im}X(a),
\qquad
K_q(X)
 =\bigcap_{\substack{p\in Q\\a\in\mathfrak r(q,p)}}
             \operatorname{Ker}X(a).
\]
We refer to \cite{HJ-JLMS,HJ-TAMS} or \cite[\S4.1]{SV_Q_shaped_homotopy} for these
constructions and their basic properties.

Let $\theta$ be a class of $A$-modules. A short exact
sequence $\xi$ is \textbf{$\theta$-exact} if
$\Hom{A}{T}{\xi}$ is exact for every $T\in\theta$.
These sequences define an exact structure, denoted by
$\lMode{A}{\theta}$. Its \textbf{objectwise
$\theta$-exact structure} on $\lMod{Q,A}$ consists of
the sequences whose evaluation at every $q$ is
$\theta$-exact; we denote it by $\lMode{Q,A}{\theta}$.
Projectivity and injectivity relative to these
structures will be called often relative projectivity and
relative injectivity. The \textbf{saturation} of $\theta$ is defined in \cite{SV_Q_shaped_homotopy} to be
\[
\operatorname{Sat}(\theta)
 :=\Prjp{\lMode{A}{\theta}}.
\]
Equivalently, it consists of all modules $P$ such that
$\Hom{A}{P}{\xi}$ is exact for every $\theta$-exact
sequence $\xi$.
We call $\theta$ \textbf{saturated} if
$\theta=\operatorname{Sat}(\theta)$.
Although $\theta$ need not contain all relative
projectives, replacing it by its saturation does not
change the exact structure:
\[
\lMode{A}{\theta}
 =\lMode{A}{\operatorname{Sat}(\theta)}.
\]
Thus saturation allows us to identify the chosen class
with the full class of relative projectives
\cite[Def.~4.15]{SV_Q_shaped_homotopy}.

For a class $\mathcal X$, write
$\operatorname{Add}(\mathcal X)$ for the direct summands
of arbitrary coproducts of objects of $\mathcal X$,
and $\operatorname{add}(\mathcal X)$ for the analogous
closure using finite coproducts.
For any set $\varrho$ of modules,
\[
\operatorname{Sat}(\varrho)
 =\operatorname{Add}(\{A\}\cup\varrho).
\]
In particular, if $A\in\varrho$, its saturation is
$\operatorname{Add}(\varrho)$; see
\cite[Lem.~4.16]{SV_Q_shaped_homotopy}.

A morphism $f:T\to M$, with $T\in\theta$, is a
\textbf{$\theta$-precover} if every morphism $T'\to M$
with $T'\in\theta$ factors through $f$.
The class $\theta$ is \textbf{precovering} if every
module admits such a morphism.
If $\theta$ is saturated, it contains $A$, so every
$\theta$-precover is surjective and yields a
$\theta$-exact sequence
\[
0\longrightarrow\operatorname{Ker}f
 \longrightarrow T\stackrel{f}{\longrightarrow}M\longrightarrow0.
\]
Consequently, a saturated precovering class gives
enough projectives in $\lMode{A}{\theta}$.

The functors $F_q,E_q,G_q$ are exact for the relative
structures, and $F_q,E_q$ preserve relative
projectives \cite[Lem.~4.3]{SV_Q_shaped_homotopy}.
If $\theta$ is saturated and precovering, then
$\lMode{Q,A}{\theta}$ also has enough projectives, with
\[
\Prjp{\lMode{Q,A}{\theta}}
 =
 \operatorname{Add}
 \{F_q(T)\mid q\in Q,\ T\in\theta\}.
\]
If $\theta$ is the saturation of a set of finitely
presented modules, both exact categories are efficient
\cite[Prop.~4.17]{SV_Q_shaped_homotopy}. Below we present the main examples of such exact structures.

\begin{exa}
\label{exa:Prj-PPrj-Mod}
  We mention a few canonical examples of saturated classes of $A$-modules $\theta$ for which the exact categories $\lMode{A}{\theta}$ and $\lMode{Q,A}{\theta}$ have enough projectives and are efficient (in particular, they have exact coproducts).
  \begin{prt}
  \item If $\theta = {}_A\operatorname{Prj}$ is the class of projective $A$-modules, then the conflations in $\lMode{A}{\theta}$ are precisely the usual exact sequences in the abelian structure. As $\theta$ is the saturation of the singleton set $\{A\}$ both $\lMode{A}{\theta}$ and $\lMode{Q,A}{\theta}$ are efficient. (This is well-known from \cite[Ex.~2.8(1)]{SaorinStovicek} as, in this case, $\lMode{A}{\theta}$ and $\lMode{Q,A}{\theta}$ are just the abelian categories $\lMod{A}$ and $\lMod{Q,A}$, which are Grothendieck). It is proven in \cite{SV_Q_shaped_homotopy} that the set $\mathscr{G}_\theta:=\mathscr{G}(\theta)=\{\Sq{q}(T)\,|\, T\in\theta,\,q\in Q\}$ generates a complete cotorsion pair in $\lMode{Q,A}{\theta}$ and by extension it corresponds to a Hovey triple, whose homotopy category is the $Q$-shaped derived category of the ring $A$.
  
  \item If $\theta = {}_A\operatorname{PPrj}$ is the class of pure-projective $A$-modules, then the conflations in $\lMode{A}{\theta}$ are precisely the pure exact sequences \cite[Def.~6.5]{mta}. As $\theta$ is the saturation of the class of finitely presented $A$-modules, both $\lMode{A}{\theta}$ and $\lMode{Q,A}{\theta}$ are efficient with enough projectives. 

  \item If $\theta = \lMod{A}$ is the class of all $A$-modules, then the conflations in $\lMode{A}{\theta}$ are precisely the split exact sequences. The class $\theta$ is trivially precovering and, therefore, $\lMode{A}{\theta}$ and $\lMode{Q,A}{\theta}$ have enough projectives. It is also easy to check that the split exact structure in a Grothendieck abelian category (in fact, in any accessible additive category with set-indexed coproducts, see \cite[Ex.~2.8(2)]{SaorinStovicek}), like the category of $A$-modules, is efficient. Thus, both $\lMode{A}{\theta}$ and $\lMode{Q,A}{\theta}$ are efficient. Moreover, $\lMode{A}{\theta}$ and $\lMode{Q,A}{\theta}$ are both Frobenius. It is proven in \cite{SV_Q_shaped_homotopy} that the class $\mathscr{G}_\theta=\{\Sq{q}(T)\,|\, T\in\theta,\,q\in Q\}$ generates a complete cotorsion pair (that coincides with the trivial injective cotorsion pair) in $\lMode{Q,A}{\theta}$. The associated homotopy category was defined in \emph{loc. cit.} to be the $Q$-shaped homotopy category.
  \end{prt}
\end{exa}

We are interested in the compact generation of the homotopy categories that arise from $\theta$-exact structures on $\lMod{A}$. We now specify the hypotheses under which these exact structures give the homotopy categories considered below.

\begin{stp}
\label{stp:HJ+theta}
Let $\Bbbk$ be a hereditary, commutative noetherian ring, and  $Q$ and $A$ be as in \cite[Stp.~2.5]{HJ-JLMS}. Let also $\theta$ to be a class of $A$-modules satisfying the following properties:
\begin{rqm}    
\item $\theta$ is saturated and precovering;
\item the exact category $\lMode{A}{\theta}$ is efficient;
\item the class $\mathscr{G}_\theta:=\mathscr{G}(\theta)=\{\Sq{q}(T)\,|\, T\in\theta,\,q\in Q\}$ generates a complete cotorsion pair in $\lMode{Q,A}{\theta}$.
\end{rqm}
\end{stp}

Under these assumptions, the stalks in
$\mathscr G_\theta$ are Gorenstein projective.
Moreover, in this case one does not need to consider the additional syzygies. Hence the construction recalled
in Subsection~\ref{subsec:SV-models} takes the form
\[
\mathcal W_\theta=\mathscr G_\theta^\perp,
\qquad
\mathcal C_\theta={}^\perp\mathcal W_\theta,
\]
where the orthogonals are computed using the objectwise
$\theta$-exact structure.
The resulting hereditary projective model structure
has Hovey triple
\(
(\mathcal C_\theta,\mathcal W_\theta,
 \lMode{Q,A}{\theta}),
\)
and we write
\[
\QST{Q}{A}
 :=\operatorname{Ho}(\lMode{Q,A}{\theta})
 \simeq\underline{\mathcal C}_\theta;
\]
for the resulting homotopy category, see \cite[Cor.~4.12, Lem.~5.1, Thm.~5.2]{SV_Q_shaped_homotopy}.
In particular, $\mathcal C_\theta$ is Frobenius with
projective-injectives
$\Prjp{\lMode{Q,A}{\theta}}$.
When $\theta$ is the saturation of a set of finitely
presented modules, the conditions of
Setup~\ref{stp:HJ+theta} hold by
\cite[Prop.~4.17(d), Lem.~6.7]{SV_Q_shaped_homotopy}.

For the cohomology notation used below, put
$S\langle q\rangle=S_q(\Bbbk)$ in $\lMod{Q}$ and write
\[
\sHH{i}{\stalkco{q}}{n}(Y)
 :=
 \Ext{Q}{i}{\sigma^nS\langle q\rangle}{Y},
\qquad i>0,\quad n\in\mathbb Z.
\]
Here the cycles are taken in chosen totally acyclic
projective complexes in $\lMod{Q}$, and Ext is computed
in its abelian exact structure.
For $T\in\lMod{A}$ and $X\in\lMod{Q,A}$, the expression
$\Hom{A}{T}{X}$ denotes the $Q$-module obtained by
applying $\Hom{A}{T}{-}$ objectwise.

\begin{prp}\label{prp:comp_gen_q_homotopy}
    Let $\theta$ be the saturation of a set $\varrho$  of finitely presented $A$-modules containing the ring $A$ and satisfying the Setup \ref{stp:HJ+theta}. The following hold.
    \begin{prt}
        \item If $M\in\varrho$ then $S_q(M)$ is a compact object in $\lMode{Q,A}{\theta}$;
        \item The set $\{S_q(M)\in\QST{Q}{A}\,|\, q\in Q,\,M\in\varrho\}$ is a set of compact generators of~ $\QST{Q}{A}$.
    \end{prt}
\end{prp}

\begin{proof}
    \proofoftag{a} Assume that $M$ lies in $\varrho$ and fix a $q\in Q$. There is a natural isomorphism 
    \[
    \Hom{Q,A}{S_q(M)}{-}\,\cong\,\Hom{A}{M}{K_q(-)}.
    \]
    The functor $K_q$ preserves coproducts by \cite[Lem.~4.6]{HJ-TAMS} (which is independent of the exact structure) and $\Hom{A}{M}{-}$ preserves coproducts since $M\in\varrho\subseteq {}_A\operatorname{mod}$.

    \proofoftag{b} The cotorsion pair $(\mathcal{C}_\theta,\mathcal{W}_\theta)$ is generated by the set of compact objects $\{S_q(M)\in\lMode{Q,A}{\theta}\,|\, q\in Q,\,M\in\varrho\}$ in $\lMode{Q,A}{\theta}$. By Theorem \ref{thm:compact_generation} the homotopy category $\QST{Q}{A}$ is compactly generated by the set of compact objects $\{S_q(M)\in\QST{Q}{A}\,|\, q\in Q,\,M\in\varrho\}$.
\end{proof}

\begin{exa}
    Let $\theta=\lPrj{A}$ be the class of all projective $A$-modules, meaning that $\lMode{Q,A}{\theta}$ is simply the abelian category $\lMod{Q,A}$. In \cite[Ex.~5.4]{SV_Q_shaped_homotopy}, it is shown that picking the class of relative projective objects $\mathscr{T}$ to be the singleton $\{A\}$ then the set of Gorenstein projective objects that generates the abelian model structure is given by $\mathscr{G}=\{S_q(A)\,|\,q\in Q\}$. In this case the homotopy category identifies with the $Q$-shaped derived category $\QSD{Q}{A}$ of Holm and J{\o}rgensen. Notice that \prpref{comp_gen_q_homotopy} implies that the $Q$-shaped derived category is compactly generated since $\{A\}$ is a set (in fact a singleton) of finitely presented $A$-modules. This of course is an already known result, see \cite[Thm.~D]{HJ-TAMS}.
\end{exa}

\subsection{The \texorpdfstring{$Q$}{Q}-shaped pure derived category of a ring}\label{subsec:Q-shaped-pure}

In Proposition \ref{prp:comp_gen_q_homotopy} we present a sufficient condition for a $Q$-shaped $\theta$-homotopy category to be compactly generated. In this subsection we are going to construct another example of such homotopy category by endowing $\lMod{Q,A}$ with the objectwise pure exact structure. \,

Consider the set $\varrho$ of all finitely presented $A$-modules, that is $\varrho={}_A\operatorname{mod}$ and let $\theta$ be the saturation of $\varrho$. By Example \ref{exa:Prj-PPrj-Mod} one has that $\theta$ is the class of all pure-projective $A$-modules. In particular, $\lMode{A}{\theta}$ identifies with the pure exact structure on $\lMod{A}$. In this case we call the objectwise $\theta$-exact structure on $\lMode{Q,A}{\theta}$, the \textbf{objectwise pure exact structure}. To avoid confusion with other $\theta$-exact structures, when referring to the (objectwise) pure exact structure we will decorate our categories from now on with $\varrho$ instead of $\theta$, e.g. we will write $\lMode{Q,A}{\varrho}$.\,

In view of \cite[Thm.~5.2]{SV_Q_shaped_homotopy} taking as $\mathscr{T}$ the set $\varrho$, one has that $\mathscr{G}_\varrho=\{S_q(M)\,|\,M\in\varrho,\,q\in Q\}$ generates a complete hereditary cotorsion pair and by extension a hereditary projective Hovey triple $(\mathcal{C}_\varrho,\mathcal{W}_\varrho,\lMode{Q,A}{\varrho})$. The subcategory of all cofibrant objects $\mathcal{C}_\varrho$ is a Frobenius exact category whose projective-injective objects are given by $\operatorname{Add}(\{F_q(T)\,|\, T\in{}_A\operatorname{PPrj},\,q\in Q\})$. We call the stable category $\underline{\mathcal{C}}_\varrho$ the \textbf{$Q$-shaped pure derived category} of $A$ and we denote it by $\QSPD{Q}{A}$. \,

The construction of the projective model of the $Q$-shaped pure derived category that we introduced in the paragraph above, follows directly from an application of Theorem \cite[Thm.~5.2]{SV_Q_shaped_homotopy}. In theory, one could construct also the injective model of the $Q$-shaped pure derived category using the dual theorem in \emph{loc. cit.}, however, that theorem does not guarantee the completeness of the associated cotorsion pair. In what follows we construct manually an injective model for the $Q$-shaped pure derived category. \,

In \emph{loc. cit.} the authors give an explicit description of the pure-trivial objects $\mathcal{W}_\varrho$. In particular, one has
\[
\mathcal{W}_\varrho\,=\,\mathscr{G}_\varrho^\perp\,=\,\{X\in\lMode{Q,A}{\varrho}\,|\, \Ext{\exact{\varrho}}{1}{S_q(F)}{X}=0, \text{ for all }q\in Q,\text{ and all }F\in\varrho\},
\]
where $\Ext{\exact{\varrho}}{1}{-}{-}$ denotes the Ext functor with respect to the objectwise pure exact structure. Denote by $\mathcal{F_\varrho}:=\mathcal{W}_\varrho^\perp$ the subcategory of the (candidate) fibrant objects. \,

Let $\lambda$ be a regular cardinal. In a locally finitely presented Grothendieck category $\mathcal{M}$ (for example, $\mathcal{M}=\lMod{Q,A}$), a class of objects $\mathcal{K}$ is said to be \textbf{$\lambda$-Kaplansky} if for every $K\in\mathcal{K}$ and every $\lambda$-generated subobject (in the sense of \cite[Dfn.~1.67]{AdamekRosicky}) $X\subseteq K$, there exists a $\lambda$-presented subobject $Y$ of $F$ such that $X\subseteq Y\subseteq K$ and $Y,\, K/Y\in \mathcal{K}$. One calls $\mathcal{K}$ a \textbf{Kaplansky class}, if it is $\lambda$-Kaplansky for some regular cardinal $\lambda$. In this generality the definition is due to Gillespie \cite[Dfn.~4.9]{MR2342555} (we refer also to \cite[Dfn.~2.6]{Sto2013a}). Historically, Kaplansky classes were introduced by Enochs and L\'opez-Ramos in \cite{EEnJLR02} for the special case where $\mathcal{M}$ is the module category of a ring. Kaplansky classes play an important role in deconstructibility and completeness of cotorsion pairs, in exact categories of Grothendieck type \cite{Stovicek2013}. It is worth mentioning that a class $\mathcal{K}$ in $\mathcal{M}$ is Kaplansky, if it is closed under pure subobjects and pure quotients, see \cite[Proof of Thm.~A.3]{HJ-JLMS}. For module categories this was initially shown by Holm and J\o rgensen in \cite{HJ08}.

\begin{dfn}\label{dfn:cat_pure}
    A short exact sequence $\xi:0\to X\to Y\to Z\to0$ in $\lMod{Q,A}$ is said to be a \textbf{categorical pure exact sequence}, if $\Hom{Q,A}{K}{\xi}$ is exact, for all finitely presented objects $K$ in $\lMod{Q,A}$. 
\end{dfn}

Equivalently, one can define categorical pure exact sequences as filtered colimits of split exact sequences, as usual in Grothendieck categories. When referring to categorical purity we will add ``categorical" in every notion that we have already defined so far (e.g. categorical pure-projective, etc.). We note that categorical purity is stronger than the objectwise purity we mentioned above. Indeed, the evaluation functor $\operatorname{ev}_q(-)\colon\lMod{Q,A}\to\lMod{A}$ preserves split exact sequences and colimits, thus categorical purity implies objectwise purity.

\begin{prp}\label{prp:injective_pure}
    Let $\varrho=\lsmod{A}$ denote the set of all finitely presented $A$-modules. Consider the set $\mathscr{G}_\varrho=\{S_q(M)\,|\,M\in\varrho,\,q\in Q\}$, the class
    \[
\mathcal{W}_\varrho\,=\,\mathscr{G}_\varrho^\perp\,=\,\{X\in\lMode{Q,A}{\varrho}\,|\, \Ext{\exact{\varrho}}{1}{S_q(F)}{X}=0, \text{ for all }q\in Q,\text{ and all }F\in\varrho\},
\]
and let $\mathcal{F}_\varrho=\mathcal{W}_\varrho^\perp$. The following hold.
\begin{prt}
    \item $\mathcal{W}_\varrho$ is a thick subcategory of $\lMode{Q,A}{\varrho}$ and it is closed under filtered colimits;
    \item $\mathcal{W}_\varrho$ is closed under categorical pure subobjects and categorical pure quotients, that is, $\mathcal{W}_\varrho$ is a Kaplansky class;
    \item $\mathcal{W}_\varrho\cap\mathcal{F}_\varrho=\lInj{Q,A}^\varrho$;
    \item The cotorsion pair $(\mathcal{W}_\varrho,\mathcal{F}_\varrho)$ is complete and hereditary.
\end{prt}
\end{prp}

\begin{proof}
    \proofoftag{a} Note that $\mathcal{W}_\varrho$ is thick by \cite[Prop.~3.3]{SV_Q_shaped_homotopy}. Recall also from \emph{loc. cit.} that stalks of relative projective objects are Gorenstein projective. For the closure under filtered colimits, by \cite[Lem.~5.1~\&~Thm.~5.2]{SV_Q_shaped_homotopy} one has for any finitely presented $A$-module $F$ an isomorphism
    \[
    \Ext{\exact{\varrho}}{1}{S_q(F)}{-}\,\cong\, \mathbb{H}^{1}_{S\langle q\rangle,0}(\Hom{A}{F}{-})
   \,,
    \]
   where the latter commutes with filtered colimits in $\lMod{Q}$ by \cite[Rmk.~4.7]{HJ-TAMS}, and by the fact that $F$ is a finitely presented $A$-module. In particular, one has that the stalks $S_q(F)$, $q\in Q$, with $F$ in $\lsmod{A}$, are $\operatorname{FP}_\infty$ objects in $\lMode{Q,A}{\varrho}$.\,
    
    \proofoftag{b} To prove that $\mathcal{W}_\varrho$ is a Kaplansky class, we show it is closed under categorical pure subobjects and categorical pure quotients. In particular, given a categorical pure exact sequence $0\to X'\to X\to X''\to 0$, with $X$ in $\mathcal{W}_\varrho$, we show that $X',\,X''$ lie in $\mathcal{W}_\varrho$, as well. Since, $\lMod{Q,A}$ is locally finitely presented, the given categorical pure exact sequence can be written as a filtered colimit of split exact sequences
    \begin{equation}\label{eq:cat_pure_colim}
    \left\{0\longrightarrow X_i'\longrightarrow X_i\longrightarrow X_i''\longrightarrow0\right\}_{i\in I}.
    \end{equation}
    Note that $\Ext{\exact{\varrho}}{1}{S_q(F)}{-}$ is an additive functor and, by the same argument in the proof of item (a) above, it commutes with filtered colimits. More concretely, since \eqref{cat_pure_colim} splits, we have $$\Ext{\exact{\varrho}}{1}{S_q(F)}{X_i}\cong\Ext{\exact{\varrho}}{1}{S_q(F)}{X'_i}\oplus\Ext{\exact{\varrho}}{1}{S_q(F)}{X''_i},\, i\in I.$$
    Taking, the colimit here, we have 
    \begin{align*}
    \Ext{\exact{\varrho}}{1}{S_q(F)}{X}\,\cong\,\underrightarrow{\lim}\left(\Ext{\exact{\varrho}}{1}{S_q(F)}{X_i}\right) & \,\cong \,\underrightarrow{\lim}\left(\Ext{\exact{\varrho}}{1}{S_q(F)}{X'_i}\right)\,\oplus\,\underrightarrow{\lim}\left(\Ext{\exact{\varrho}}{1}{S_q(F)}{X''_i}\right)\\ &\,\cong\, \Ext{\exact{\varrho}}{1}{S_q(F)}{X'}\,\oplus \,\Ext{\exact{\varrho}}{1}{S_q(F)}{X''}
    \end{align*}
    However, $\Ext{\exact{\varrho}}{1}{S_q(F)}{X}=0$, since $X\in \mathcal{W}_\varrho$, and thus the same holds also for its summands. We conclude that $X',\,X''$ belong to $\mathcal{W}_\varrho$.

    \proofoftag{c} The inclusion $\lInj{Q,A}^\varrho\subseteq\mathcal{W}_\varrho\cap\mathcal{F}_\varrho$ is straightforward to check. Indeed, given a relative injective object $J\in \lInj{Q,A}^\varrho$, both $\Ext{\exact{\varrho}}{1}{S_q(F)}{J}$ and $\Ext{\exact{\varrho}}{1}{W}{J}$ vanish, implying that $J$ is a trivial fibrant object.  For the converse, let $I$ lie in $\mathcal{W}_\varrho\cap\mathcal{F}_\varrho$. Since, $\lMode{Q,A}{\varrho}$ has enough injectives (and projectives), let 
    \[
    \xi\colon 0\longrightarrow I\longrightarrow J\longrightarrow C\longrightarrow0\,,
    \]
    be a conflation with $J\in \lInj{Q,A}^\varrho$ and denote by $C$ the cokernel $\Coker{(I\to J)}$. Since $J$ is a relative injective, it lies $\mathcal{W}_\varrho$. Moreover, by assumption $I\in \mathcal{W}_\varrho$, and thus, so does $C$, by thickness. However, $\xi$ is an element of $\Ext{\exact{\varrho}}{1}{C}{I}$ and the latter vanishes, since $I\in \mathcal{F}_\varrho$. Hence, $\xi$ splits and $I$ is a summand of $J$, and thus it is a relative injective as well.

    \proofoftag{d} Items (a) and (b) show that $\mathcal{W}_\varrho$ is a Kaplansky class that is closed under filtered colimits. By \cite{Sto2013a}, it is a deconstructible class and hence, $(\mathcal{W}_\varrho,\mathcal{W}_\varrho^\perp)$ is a complete cotorsion pair by \cite{Stovicek2013}. The hereditariness of the cotorsion pair follows by \cite[Thm.~5.6]{SV_Q_shaped_homotopy}.
\end{proof}

The corresponding model structure induced by the complete cotorsion pair $(\mathcal{W}_\varrho,\mathcal{F}_\varrho)$ gives the injective model of the $Q$-shaped pure derived category. Indeed, this injective model structure shares the same trivial objects and weak equivalences with the projective model structure we introduced above. Moreover,  the subcategory $\mathcal{F}_\varrho$ of the (bi-)fibrant objects is a Frobenius exact category whose projective/injective objects are precisely the relative injective objects $\lInj{Q,A}^\varrho$. In particular, for the homotopy category one has
\[
\frac{\mathcal{C}_\varrho}{\lPrj{Q,A}^\varrho}\stackrel{\cong}{\longrightarrow}\operatorname{Ho}(\lMod{Q,A}^\varrho)=\QSPD{Q}{A}\stackrel{\cong}{\longleftarrow}\frac{\mathcal{F}_\varrho}{\lInj{Q,A}^\varrho}\,.
\]

Let $\theta$ be a class of objects such that $\lPrj{A}\subseteq\theta$, and assume it satisfies the Setup \ref{stp:HJ+theta}, that is, one has a projective model structure on $\lMode{Q,A}{\theta}$ and moreover, assume that one can also define an injective model structure with the same trivial objects and weak equivalences, and thus the same homotopy category $\mathbf{K}_Q^\theta(A)$. We will call for brevity such model structures a \textbf{homotopically balanced pair of model structures}. Examples, of homotopically balanced pairs of model structures include (but are not limited to): the projective and injective models for the derived, homotopy and pure derived category and their $Q$-shaped analogues. 

\begin{dfn}\label{dfn:relative_gldim}
    Let $(\mathcal{A},\exact{})$ be a weakly idempotent complete exact category with enough projectives. We define the (left) \textbf{relative global dimension} of $(\mathcal{A},\exact{})$ by
    \[
    \operatorname{gl.dim}_{\exact{}}\mathcal{A}=\sup\{\operatorname{pd}_{\exact{}}M\,|\, M\in\mathcal{A}\}
    \]
    where $\operatorname{pd}_{\exact{}}M$ denotes the relative projective dimension of the object $M$, that is the length of a relative projective resolution of $M$.
\end{dfn}

\begin{exa}
    We give a brief list of examples where the notion of relative global dimension is used with a slightly different but more familiar terminology.
    \begin{rqm}
        \item Let $\lMod{A}$ be endowed with the abelian exact structure. The relative projectives in this case are the categorical projectives of the abelian category, and thus, the relative global dimension is the ordinary global dimension.
        \item  Assume now that $\lMod{A}$ is endowed with the split exact structure. In this case the relative global dimension is zero, since every object is projective.
        \item  In the case that $\lMod{A}$ is endowed with the pure exact structure, the notion of (left) relative global dimension coincides with that of (left) pure global dimension. These dimensions are explored in greater detail in the beginning of Section \ref{sec:negative_compact_gen}.
    \end{rqm}
\end{exa}

\begin{lem}\label{lem:proj/inj_char}
     Let $\theta$ be a saturated class of $\lMod{A}$ and assume that it induces a homotopically balanced pair of model structures in $\lMod{Q,A}$, denoted by $(\mathcal{C}_\theta,\mathcal{W}_\theta,\lMode{Q,A}{\theta})$ and $(\lMode{Q,A}{\theta},\mathcal{W}_\theta,\mathcal{F}_\theta)$, for the projective and injective model, respectively. Consider an object $Z$ in $\mathcal{W}_\theta$. The following hold.
     \begin{rqm}
         \item The object $Z$ is projective in $\lMode{Q,A}{\theta}$ if and only if, $C_p(Z)$ is projective in $\lMode{A}{\theta}$, for all $p\in Q$.
         \item The object $Z$ is injective in $\lMode{Q,A}{\theta}$ if and only if, $K_p(Z)$ is injective in $\lMode{A}{\theta}$, for all $p\in Q$.
     \end{rqm}
\end{lem}

\begin{proof}
    In \cite[Prop.~4.6]{SV_Q_shaped_homotopy} is is shown that relative projectives are precisely summands of coproducts of objects of the form $F_q(P)$, with $P\in \lPrj{A}^\theta$, and dually, relative injective objects are summands of products of objects of the form $G_q(I)$, with $I\in\lInj{A}^\theta$. This proves the ``\emph{only if}'' assertions. \,

    For the converse ``\emph{if}'' assertion we will work with the injectives and one can prove the dual arguments similarly. For this purpose, assume that $Z\in\mathcal{W}_\theta$ and moreover, $K_p(Z)$ are relative injective $A$-modules for all $p\in Q$. Consider the canonical conflation
    \[
    \mathfrak{g}_Z\colon 0\longrightarrow Z\stackrel{\eta^Z}{\longrightarrow}\mathbb{G}(Z)\longrightarrow\mathbb{C}(Z)\longrightarrow0
    \]
    from \cite[Prop.~5.7]{HJ-TAMS}, where it is also shown that $\mathfrak{g}_Z$ is always objectwise split (and hence also a $\theta$-conflation). Apply the functor $\Hom{\lMode{Q,A}{\theta}}{S_q(P)}{-}$ to the conflation $\mathfrak{g}_Z$, for an arbitrary $P\in\lPrj{A}^\theta$. Since $Z$ is a trivial object, $\Ext{\lMode{Q,A}{\theta}}{1}{S_q(P)}{Z}=0$, for all $q\in Q$ and $P\in\lPrj{A}^\theta$, one gets that $\mathfrak{g}_Z$ remains a conflation after $\Hom{\lMode{Q,A}{\theta}}{S_q(P)}{-}$. Using the adjunction $(S_q,\,K_q)$, {\color{black} that is, $\Hom{\lMode{Q,A}{\theta}}{S_q(P)}{-}\,\cong\,\Hom{\lMode{A}{\theta}}{P}{K_q(-)}$,} one also has that applying $K_q(-)$ to $\mathfrak{g}_Z$  yields a $\theta$-conflation, that is,
    \[
    0\longrightarrow K_q(Z)\longrightarrow K_q(\mathbb{G}(Z))\longrightarrow K_q(\mathbb{C}(Z))\longrightarrow0
    \]
    is exact, {\color{black}since $\Hom{\lMode{A}{\theta}}{P}{K_q(-)}$ is an exact functor by the adjunction, and $P$ is a relative projective, thus $K_q(-)$ is exact as well}. Moreover, by \cite[Prop.~7.28(b)]{HJ-JLMS}, one has $K_q(\mathbb{G}(Z)(q))\cong Z(q)$, for any $q\in Q$. In particular, since $K_q(Z)$ is relative injective $A$-module the inflation $K_q(Z)\to Z(q)$ splits. Choose a retraction $\pi_q\colon Z(q)\to K_q(Z)$ and define the composition
    \[
    \phi_q\colon Z\stackrel{\eta^Z_q}{\longrightarrow} G_q(E_q(Z))=G_q(Z(q))\stackrel{G_q(\pi_q)}{\longrightarrow} G_q(K_q(Z))
    \]
    where $\eta^Z_q$ is the unit of the adjunction $(E_q,\,G_q)$. Using \cite[Prop.~7.28(b)]{HJ-JLMS}, one calculates that $K_q(\phi_q)\cong \operatorname{id}_{K_q(Z)}$. {\color{black}More concretely, one has
    \[
    K_q(\phi_q)=K_q(G_q(\pi_q)\circ\eta_q^Z)=K_qG_q(\pi_q)\circ K_q(\eta^Z_q)=\pi_q\circ K_q(\eta^Z_q)=\operatorname{id}_{K_q(Z)}.
    \]
    }By the universal property of the product, the family $\{\phi_q\}_{q\in Q}$ induces a unique map $\phi\colon Z\to\prod_{q\in Q}G_q(K_q(Z))$. Since $K_q$ commutes with products (as a right adjoint), one also has that $K_q(\phi)\cong \operatorname{id}_{K_q(Z)}$. We claim that $\phi$ is an isomorphism. Since $K_q$ is left exact one has that $K_q(\Ker{\phi})=\Ker{K_q(\phi)}=0$, for all $q\in Q$. Consequently, by \cite[Prop.~7.19]{HJ-JLMS}, $\Ker{\phi}=0$ and thus $\phi$ is a monomorphism in the ambient abelian category. On the other hand, applying $K_q$ to the short exact sequence (again in the ambient abelian category) induced by the monomorphism $\phi$, we get a long exact sequence
    \[
    0\longrightarrow K_q(Z)\longrightarrow K_q(\prod_{q\in Q}G_qK_q(Z))\stackrel{0}{\longrightarrow} K_q(\Coker{\phi})\longrightarrow\mathbb{R}^1K_q(Z)\longrightarrow\cdots
    \]
    where the third arrow represents the zero morphism since $K_q(\phi)$ is an isomorphism. Recall, however, by \cite[Lem.~4.3]{HJ-JLMS} that the first derived functor of $K_q$, satisfies
    \[
    \mathbb{R}^1K_q(Z)\cong\Ext{Q,A}{1}{S_q(A)}{Z},
    \]
    where the Ext group is computed in the abelian category. We claim that the latter is zero. Indeed, we note that one has a natural isomorphism $$\Ext{Q,A}{1}{S_q(A)}{Z}\cong\Ext{\lMode{Q,A}{\theta}}{1}{S_q(A)}{Z}.$$ Since $A$ is a projective $A$-module, every short exact sequence $0\to Z\to E\to S_q(A)\to 0$ is an objectwise split conflation, and consequently, also an objectwise $\theta$-conflation. We conclude that the latter Ext-group vanishes, since $Z$ lies in $\mathcal{W}_\theta$ and $A\in\theta$, therefore, showing that $K_q(\Coker{\phi})=0$, for all $q\in Q$ and thus, $\phi$ is an epimorphism as wanted, meaning that $Z$ is isomorphic to an injective object in $\lMode{Q,A}{\theta}$.
\end{proof}

Lemma \ref{lem:proj/inj_char} recovers the already known description of the projective and injective objects in the case when $\theta=\lPrj{A}$ (see \cite[Thm.~7.29]{HJ-JLMS}) and (trivially) when $\theta=\lMod{A}$.  The following Theorem, gives a concrete description of the cofibrant and fibrant objects in homotopically balanced pairs of model structures. However, notice that one can get the description of the cofibrant (resp. fibrant) objects by assuming only a projective (resp. injective) model structure induced by $\theta$. We ask for $\theta$ to induce a homotopically balanced pair of model structures to present a unified version of the Theorem.

\begin{thm}\label{thm:bifibrant_description}
    Let $\theta$ be a saturated class of $\lMod{A}$ and assume that it induces a homotopically balanced pair of model structures in $\lMod{Q,A}$, denoted by $(\mathcal{C}_\theta,\mathcal{W}_\theta,\lMode{Q,A}{\theta})$ and $(\lMode{Q,A}{\theta},\mathcal{W}_\theta,\mathcal{F}_\theta)$, for the projective and injective model, respectively. There are inclusions
    \begin{align*}
        \mathcal{C}_\theta & \subseteq \{X\in\lMode{Q,A}{\theta}\,|\,X(q)\in\lPrj{A}^\theta,\text{ for all } q\in Q\}\,\text{ and }\\
        \mathcal{F}_\theta & \subseteq \{Y\in\lMode{Q,A}{\theta}\,|\,Y(q)\in\lInj{A}^\theta,\text{ for all } q\in Q\},
    \end{align*}
    that become equalities if $\lMod{A}^\theta$ has finite relative global dimension with respect to $\theta$.
\end{thm}

\begin{proof}
    We prove the result for the fibrant objects. The cofibrant case follows by the dual arguments. First we establish the inclusion
    \[
    \mathcal{F}_\theta \subseteq \{Y\in\lMode{Q,A}{\theta}\,|\,Y(q)\in\lInj{A}^\theta,\text{ for all } q\in Q\}=:\mathscr{I}.
    \]
    To show that a functor $X\in\mathcal{F}_\theta$ takes values in relative injective $A$-modules, {\color{black}it suffices to show that the functor $\Hom{A}{-}{X(q)}$ is exact for all $q\in Q$. By the adjunction $(F_q,E_q)$, it is enough to show that $\Hom{\lMode{Q,A}{\theta}}{F_q(-)}{X}$ is exact. By \cite[Lem.~4.3]{SV_Q_shaped_homotopy}, the functors $F_q,\,E_q$ and $G_q$ are exact, thus one needs to show that for a given conflation $\xi\colon0\to L\to M\to N\to 0$, the sequence $\Hom{\lMode{Q,A}{\theta}}{F_q(\xi)}{X}$ is exact as well. Equivalently, it suffices to show that \begin{equation}\label{eq:ext-vanishing-F}
    \Ext{\lMode{Q,A}{\theta}}{1}{F_q(N)}{X}=0,
    \end{equation} or even stronger (as we argue below) that $F_q(M)$ lies in $\mathcal{W}_\theta$ for all $q\in Q$ and $M\in\lMod{A}$.\,
    
    We prove at the same time that $F_q(M)$ and $G_q(M)$ lie in $\mathcal{W}_\theta$ for all $q\in Q$ and $M\in \lMod{A}$. For this purpose, consider a relative projective resolution $P^\bullet\longrightarrow S_p(T)$ in $\lMode{Q,A}{\theta}$, for some $p\in Q$ and $T\in\theta$. Recall that the functors $E_q$ and $G_q$ are exact and in particular, by \cite[Prop.~4.6]{SV_Q_shaped_homotopy}, $E_q$ maps relative projective objects to relative projective $A$-modules. Thus, applying $E_q$ to the projective resolution of $S_p(T)$ gives a projective resolution of $E_q(S_p(T))$ in $\lMode{A}{\theta}$. The $(E_q,G_q)$ adjunction gives an isomorphism of complexes
    \[
    \Hom{\lMode{Q,A}{\theta}}{P^\bullet}{G_q(M)}\,\cong\,\Hom{\lMode{A}{\theta}}{E_q(P^\bullet)}{M}.
    \]
    Taking cohomology induces an isomorphism
    \[
    \Ext{\lMode{Q,A}{\theta}}{1}{S_p(T)}{G_q(M)}\,\cong\,\Ext{\lMode{A}{\theta}}{1}{E_q(S_p(T))}{M},
    \]
    where $E_q(S_p(T))$ takes the value $T$ in $p=q$ and $0$ elsewhere. Finally, since $T$ is a relative projective $A$-module, the latter Ext group vanishes, and thus, $G_q(M)$ lies in $\mathcal{W}_\theta$ as desired. By \cite[Lem.~3.4]{HJ-TAMS} for every $q\in Q$, there exist a natural isomorphism $F_q\cong G_{\mathbb{S}q}$, and thus $F_q(M)$ lies in $\mathcal{W}_\theta$ as well. Thus, we show that the Ext group in \eqref{ext-vanishing-F} indeed vanishes, and thus $\Hom{A}{-}{X(q)}$ is an exact functor, proving that $X$ takes values in relative injective $A$-modules.\,}

    Assume now that $A$ has finite relative global dimension with respect to $\theta$, denoted by $\operatorname{gl.dim}_\theta A\leq n$. Let $I$ lie in $\mathscr{I}$. By the completeness of the cotorsion pair $(\mathcal{W}_\theta,\mathcal{F}_\theta)$, there is a conflation   \begin{equation}\label{eq:inj_cot_pair_approx}
    0\longrightarrow I\longrightarrow F\longrightarrow W\longrightarrow0,
    \end{equation} 
    with $W\in \mathcal{W}_\theta$ and $F\in\mathcal{F}_\theta$. Evaluating at an object $q\in Q$, one has a $\theta$-conflation
    \[
    0\to I(q)\to F(q)\to W(q)\to0
    \] 
    in $\lMod{A}$, that splits since $I(q)$ is relative injective in $\lMod{A}$. In fact, $F(q)$ is also relative injective, since $F$ is a fibrant object, and we proved above that fibrant objects take values in relative injective $A$-modules. We conclude that $W(q)$ is also relative injective as a direct summand of a relative injective. Hence, $W$ lies in $\mathscr{I}$. In particular, by \cite[Thm.~4.13]{SV_Q_shaped_homotopy}, $I, W$ and $F$ are all relative Gorenstein injective objects, as they take values on relative (Gorenstein) injective $A$-modules, for every $q\in Q$. We claim that $W$ is in fact relative injective.\,

    Since $W$ is Gorenstein injective, there is by definition a totally acyclic injective co-resolution of $W$
    \begin{equation}\label{eq:inj_cores}
    0\longrightarrow W^{-n}\longrightarrow J^{-n-1}\longrightarrow\cdots\longrightarrow J^{-1}\longrightarrow J^0\longrightarrow W\longrightarrow0,
    \end{equation}
where $J^{-i}$ are relative injective objects for every $i$. As we argued above, injective objects in $\lMode{Q,A}{\theta}$ are trivial and, by Proposition \ref{prp:injective_pure}, $\mathcal{W}_\theta$ is thick, thus $W^{-n}$ lies in $\mathcal{W}_\theta$ as well. For a fixed $q\in Q$, apply to \eqref{inj_cores} the functor $K_q(-)$. Note that $K_q$ is a left exact functor, but a priori not exact. In fact, in \cite[Sec.~7]{HJ-JLMS}, the $i$-th right derived functor of $K_q$ is the $i$-th cohomology computed in the abelian structure. Nevertheless, we showed that $W^{-n}$ is in $\mathcal{W}_\theta$, meaning that it has vanishing relative cohomology, computed in the exact category $\lMode{Q,A}{\theta}$. An application of \cite[Thm.~5.2]{SV_Q_shaped_homotopy} allows one to compute this relative cohomology in the abelian category $\lMod{Q}$, independently of the exact structure. Moreover, since $\theta$ contains $\lPrj{A}$ by assumption, vanishing of the relative cohomology implies vanishing of the abelian cohomology. In particular, $\mathbb{R}^1K_q(W^{-n})=0$.  Therefore, the sequence
\[
0\longrightarrow K_q(W^{-n})\longrightarrow K_q(J^{-n-1})\longrightarrow\cdots \longrightarrow K_q(J^{-1})\longrightarrow K_q(J^0)\longrightarrow K_q(W)\longrightarrow0
\]
is exact in $\lMode{A}{\theta}$. By Lemma \ref{lem:proj/inj_char}(2), $K_q(J^{-i})$ are injective for all $i$, and $K_q(W^{-n})$ has relative injective dimension at most $n$, by our assumption on the finiteness of the relative global dimension of $A$. It follows that $K_q(W)$ is relative injective as well, for every $q\in Q$. Thus, $W$ lies in $\mathcal{W}_\theta$ and $K_q(W)$ is injective in $\lMode{A}{\theta}$ for every $q\in Q$, hence by another application of Lemma \ref{lem:proj/inj_char}(2), we conclude that $W$ is in fact a relative injective object. Finally, in the conflation \eqref{inj_cot_pair_approx}, we have that $W$ is injective and $I$ is Gorenstein injective, and thus $\Ext{\lMode{Q,A}{\theta}}{1}{W}{I}=0$, implying that \eqref{inj_cot_pair_approx} splits. It follows that $I$ is a direct summand of $F$ and consequently, $I\in\mathcal{F}_\theta$. This concludes the proof.
\end{proof}

In the special case that $\theta$ is the class of projective $A$-modules, and one works with the abelian structure, Theorem \ref{thm:bifibrant_description} recovers a result of Holm and J\o rgensen \cite[Thm.~E]{HJ-TAMS}. When working with the objectwise pure exact structures, we recover as a Corollary an explicit description of the cofibrant and fibrant objects.

\begin{cor}\label{cor:pure_bifib_objects}
    Let $\lMod{A}$ be endowed with the pure exact structure and denote by $\varrho=\lsmod{A}$. There are inclusions
    \begin{align*}
        \mathcal{C}_\varrho & \subseteq \{X\in\lMod{Q,A}\,|\,X(q)\in\lPPrj{A},\text{ for all } q\in Q\}\,\text{ and }\\
        \mathcal{F}_\varrho & \subseteq \{Y\in\lMod{Q,A}\,|\,Y(q)\in{}_{A}\operatorname{PInj},\text{ for all } q\in Q\},
    \end{align*}
    that become equalities if $A$ has finite pure global dimension.
\end{cor}

In the special case that $Q=Q^{\text{cpx}}$ and $\lMod{Q,A}=\mathbf{Ch}(A)$ is the category of (co)chain complexes of $A$-modules, the $Q$-shaped pure derived category coincides with the (classical) pure derived category $\mathbf{D}^{\text{pure}}(A)$ of $A$, obtained by the so-called \emph{induced pure exact structure} as explained in \cite{stovicek2014purity}. This homotopy category, is proven in \emph{loc. cit.} to enjoy certain nice homological properties, e.g. it is a compactly generated triangulated category and it fits into a recollement diagram
\[
\begin{tikzcd}[column sep=6em]
	{\mathbf{K}^{\text{pac}}(A)} & {\mathbf{K}(A)} & {\mathbf{D}^{\text{pure}}(A).}
	\arrow["i_*"{description}, from=1-1, to=1-2]
	\arrow["{i^*}"', shift right=3, from=1-2, to=1-1]
	\arrow["{i^!}", shift left=3, from=1-2, to=1-1]
	\arrow["j^*"{description}, from=1-2, to=1-3]
	\arrow["{j_*}", shift left=3, from=1-3, to=1-2]
	\arrow["{j_!}"', shift right=3, from=1-3, to=1-2]
\end{tikzcd}
\]
where $\mathbf{K}^{\text{pac}}(A)$ denotes the full subcategory of $\mathbf{K}(A)$ consisting of pure-acyclic complexes of $A$-modules, that is, the full subcategory of $\mathbf{K}(A)$ consisting of trivial objects with respect to the model structure induced by the induced pure exact structure on complexes.

\begin{prp}\label{prp:pure-is-cg}
    The $Q$-shaped pure derived category $\QSPD{Q}{A}$ of a ring $A$ is compactly generated. Moreover, the inclusion $\theta_1=\lPPrj{A}\subseteq\lMod{A}=\theta_2$ induces a recollement of triangulated categories
    \[
\begin{tikzcd}[column sep=6em]
	{\mathbf{K}_Q^{\text{pac}}(A)} & {\mathbf{K}_Q(A)} & {\mathbf{D}_Q^{\text{pure}}(A).}
	\arrow["i_*"{description}, from=1-1, to=1-2]
	\arrow["{i^*}"', shift right=3, from=1-2, to=1-1]
	\arrow["{i^!}", shift left=3, from=1-2, to=1-1]
	\arrow["j^*"{description}, from=1-2, to=1-3]
	\arrow["{j_*}", shift left=3, from=1-3, to=1-2]
	\arrow["{j_!}"', shift right=3, from=1-3, to=1-2]
\end{tikzcd}
\]
where $\mathbf{K}_Q^{\text{pac}}(A)$ is the full subcategory of $\QSH{Q}{A}$ consisting of all trivial objects with respect to the model structure $(\mathcal{C}_{\theta_1},\mathcal{W}_{\theta_1},\lMode{Q,A}{\theta_1})$. Moreover, if $A$ has finite pure global dimension, the image of $j_!$ is $\mathbf{K}_Q(\lPPrj{A})$ and the image of $j_*$ is $\QSH{Q}{{}_A\operatorname{PInj}}$.
\end{prp}
\begin{proof}
    Compact generation follows immediately by Proposition \ref{prp:comp_gen_q_homotopy}. For the recollement diagram, we want to apply the results of \cite[Prop.~6.2~\&~6.3]{SV_Q_shaped_homotopy} for the projective and injective model structures \small \[(\lMode{Q,A}{\theta_2},\lPrj{Q,A}^{\theta_2},\lMode{Q,A}{\theta_2}) - (\mathcal{C}_{\theta_1},\mathcal{W}_{\theta_1},\lMode{Q,A}{\theta_1})\quad\text{and}\quad (\lMode{Q,A}{\theta_2},\lPrj{Q,A}^{\theta_2},\lMode{Q,A}{\theta_2}) - (\lMode{Q,A}{\theta_1},\mathcal{W}_{\theta_1},\mathcal{F}_{\theta_1}),\] \normalsize respectively. For this purpose we just need to show that $\Prj({\lMode{Q,A}{\theta_2}})=\Inj({\lMode{Q,A}{\theta_2}})\subseteq \mathcal{W}_{\theta_1}$. Since $\theta_1$ and $\theta_2$ are selected in a way to satisfy Setup \ref{stp:HJ+theta}, we can make use of \cite[Thm.~5.2]{SV_Q_shaped_homotopy} with $\mathscr{T}_\ell=\theta_\ell$, $\ell=1,2$, to identify trivial objects. More concretely, an object $X$ in $\lMode{Q,A}{\theta_\ell}$ is trivial if and only if one has
 \[
 \sHH{i}{\stalkco{q}}{n}(\Hom{A}{T}{X})=0,\, \text{ for all } T \in \mathscr{T}_\ell, \, q \in Q,\, n \in \mathbb{Z},\, \text{ and }\, i>0,
 \]
 where $\sHH{i}{\stalkco{q}}{n}$ is computed in $\lMod{Q}$, independently of the exact structure on $\lMode{Q,A}{\theta_\ell}$. Since  $\mathscr{T}_1\subseteq\mathscr{T}_2$, it follows that, if $X$ is trivial in \smash{$\lMode{Q,A}{\theta_2}$}, then it is also trivial in \smash{$\lMode{Q,A}{\theta_1}$}, as wanted.\,

 Finally, assuming that $A$ has finite pure global dimension, by Corollary \ref{cor:pure_bifib_objects} and, \cite[Rmk.~6.4~\&~Cor.~6.5]{SV_Q_shaped_homotopy}, one gets that the image of $j_!$ is $\mathbf{K}_Q(\lPPrj{A})$ and the image of $j_*$ is $\QSH{Q}{{}_A\operatorname{PInj}}$.
\end{proof}

Conceptually, the following result strive to establish the pure exact structure as a sort of limit for compact generation.

\begin{cor}
    Let $\varrho={}_A\operatorname{mod}$ be the class of finitely presented $A$-modules and consider a class $\varrho'\subsetneq\varrho$  whose saturations satisfy the properties of Setup \ref{stp:HJ+theta} and moreover, its saturation $\theta'$ is strictly included in the class $\theta={}_A\operatorname{PPrj}$. Then, $\operatorname{Ho}(\lMode{Q,A}{\varrho'})$ is a compactly generated triangulated category.
\end{cor}

\begin{proof}
    The proof follows by Proposition \ref{prp:comp_gen_q_homotopy}. We simply stress that we ask for the saturation $\theta'$ of $\varrho$ to be strictly included to $\theta={}_A\operatorname{PPrj}$ simply to distinguish the exact structure from the objectwise pure one, and avoid that $\operatorname{Ho}(\lMode{Q,A}{\varrho'})$ coincides with the $Q$-shaped pure derived category, which is a case that we studied earlier.
\end{proof}

\section{When is the \texorpdfstring{$Q$}{Q}-shaped homotopy category compactly generated?}\label{sec:negative_compact_gen}

For any exact structure in the lattice of exact structures of $\lMod{Q,A}$ that includes the objectwise pure exact structure and satisfies the Setup \ref{stp:HJ+theta}, the associated homotopy category is compactly generated. In this section we investigate what happens when one consider objectwise $\theta$-exact structures that are strictly included in the objectwise pure exact structure. We begin studying  the extremal example of the objectwise split exact structure.

\begin{dfn}
    An additive category $\mathcal{A}$ with coproducts is called \textbf{pure semisimple} if it has an additive generator, that is there exists an object $X$ in $\mathcal{A}$ such that $\mathcal{A}=\operatorname{Add}(X)$.
\end{dfn}

Let $\mathcal{G}$ be a locally finitely presented Grothendieck category. By \cite[Lem.~2.2~\&~Cor.~2.5]{Simson_pureglobaldim} $\mathcal{G}$ has enough pure-projective objects. In particular, every object $G$ of $\mathcal{G}$ admits a pure-projective resolution 
\[
\mathbf{P}_G=\cdots\longrightarrow P^{-n}\stackrel{d^{-n}}{\longrightarrow}P^{-n+1}\longrightarrow\cdots\longrightarrow P^{-1}\stackrel{d^{-1}}{\longrightarrow} P^0\longrightarrow G\longrightarrow 0,
\]
where each $P^i$ is pure-projective and $\Ker{d^{i}}$ is a pure subobject of $P^i$, for all $i$. We define the \textbf{pure extension functor} $\operatorname{PExt}^n_\mathcal{G}(-,-)$ by 
\[
\operatorname{PExt}^n_\mathcal{G}(G,H)\,=\, H^n(\Hom{\mathcal{G}}{\mathbf{P}_G}{H})
\]
The (left) \textbf{pure-projective dimension} (resp. (left) pure-injective dimension) of an object $G$ in $\mathcal{G}$ is defined as
\[
\operatorname{p.pd}(G)=\sup\{n\ge0\,|\, \operatorname{PExt}^n_\mathcal{G}(G,-)\neq0\}\quad \left(\text{resp. }\, \operatorname{p.id}(G)=\sup\{n\ge0\,|\, \operatorname{PExt}^n_\mathcal{G}(-,G)\neq0\}\right)
\]
The (left) \textbf{pure global dimension} of $\mathcal{G}$ is given by
\[
\operatorname{p.gl.dim}(\mathcal{G})\,=\,\sup\{\operatorname{p.pd}M\,|\, M\in \mathcal{G}\}\,=\,\sup\{\operatorname{p.id}M\,|\, M\in \mathcal{G}\}.\,
\]
Recall that we already introduced these dimensions under the notion of relative global dimension in Definition \ref{dfn:relative_gldim}. For a locally finitely presented Grothendieck category $\mathcal{G}$ (for instance $\lMod{Q,A}$) an equivalent definition by Simson \cite{Simson_puresemisimple}, states that  $\mathcal{G}$ is pure semisimple if the pure global dimension of $\mathcal{G}$ is zero, i.e. $\operatorname{p.gl.dim}(\mathcal{G})=0$, or equivalently, if every object of $\mathcal{G}$ is a direct summand of a coproduct of finitely presented objects.\,

A typical example of a pure semisimple category is the category of modules over a pure semisimple ring. Recall that a ring $R$ is called pure semisimple if every pure epimorphisms splits. It is shown in \cite{MR2657376} that this definition is in fact equivalent to the one  stated above. Moreover, in \emph{loc. cit.} it is shown that the homotopy category of complexes of $R$-modules is compactly generated if and only if the ring $R$ is pure semisimple. The following result is a $Q$-shaped analogue to \cite[Prop.~2.6]{MR2657376}

\begin{thm}\label{thm:homotopy_cg}
Let $A$ be a ring. Consider the following statements:
\begin{prt}
    \item The ring $A$ is pure semisimple;
    \item The $Q$-shaped homotopy category $\QSH{Q}{A}$ is compactly generated.
    \end{prt}
    One has that (a) implies (b), and moreover, if $Q$ is acyclic, then the statements are equivalent.
\end{thm}

\begin{proof}

       We begin by showing that (a) $\Rightarrow$ (b). Since $A$ is a pure semisimple ring, every module is pure-projective.
    In particular, this implies that every objectwise pure exact sequence in $\lMod{Q,A}$ is an objectwise split exact sequence and the objectwise pure exact structure coincides with the objectwise split exact structure. It follows that, \cite[Thm.~5.2]{SV_Q_shaped_homotopy} produces the same homotopy category for both the objectwise pure and objectwise split exact model structure, that is, the $Q$-shaped homotopy category coincides with the $Q$-shaped pure derived category, and hence, it is compactly generated by Proposition \ref{prp:pure-is-cg}.\,

    For (b) $\Rightarrow$ (a), assume $Q$ is acyclic, and that there exists a set $\mathcal{S}$ of compact objects in $\QSH{Q}{A}$ such that $\QSH{Q}{A}=\operatorname{Loc}(\mathcal{S})$. We prove that there exists an object $X\in\lMod{A}$ such that $\lMod{A}=\operatorname{Add}(X)$, that is $\lMod{A}$ is a pure semisimple category. Notice that for the class $\theta=\lMod{A}$ the category $\lMode{Q,A}{\theta}$ endowed with the objectwise split exact structure is an efficient Frobenius exact category and has set-indexed coproducts (as we argued in Example \ref{exa:Prj-PPrj-Mod}). It follows that the stable category which is equivalent to $\QSH{Q}{A}$ admits set-indexed coproducts (see \cite[proof~of~Cor.~3.12]{SaorinStovicek}) and consequently it has splitting idempotents by \cite[Prop.~1.6.8]{Nee}. This allows us to replace the set $\mathcal{S}$ with a single generator $\tilde{Y}=\coprod_{Z\in\mathcal{S}}Z$ in $\QSH{Q}{A}$. {\color{black}Notice that, $\operatorname{Loc}(\{\tilde{Y}\})\,=\,\operatorname{Loc}(\mathcal{S})\,=\,\QSH{Q}{A}$. In particular, $\operatorname{Loc}(\{\tilde{Y}\})\subseteq\operatorname{Loc}(\mathcal{S})$ since the latter is closed under coproducts, whereas the inverse inclusion is given by the fact that $\operatorname{Loc}(\{\tilde{Y}\})$ is closed under summands.} We lift $\tilde{Y}$ to an object $Y$ in $\lMod{Q,A}$ by taking its cofibrant replacement with respect to the $\theta$-exact model structure. That is, $Y$ is weakly equivalent (and thus isomorphic in $\QSH{Q}{A}$) to $\tilde{Y}$. For convenience we will use bracket notation for the iso-closure of objects in the stable category $\QSH{Q}{A}$, for example, here one has that $\tilde{Y}\cong[Y]$ in $\QSH{Q}{A}.$ \,

    Consider now the object $$X=\coprod_{q\in Q}E_q(Y)$$
    in $\lMod{A}$. This is our candidate additive generator of $\lMod{A}$, that is, we strive to prove the equality $\lMod{A}=\operatorname{Add}_A(X)$. For brevity denote by $\mathcal{B}=\operatorname{Add}_A(X)$ and define the subcategory $\qb{Q}\subset\lMod{Q,A}$ by
    \[
    \qb{Q}\,:=\,\{ M\in\lMod{Q,A}\,|\, M(q)\in \mathcal{B},\text{ for every } q\in Q\},
    \]
that is, all the $\Bbbk$-linear functors $Q\to\lMod{A}$ that take values in $\mathcal{B}$. Consider also the subcategory
\[
\mathcal{L}_\mathcal{B}\,:=\,\{T\in\QSH{Q}{A}\,|\, T\cong[M]\text{ in } \QSH{Q}{A}\text{ for some } M\in\qb{Q}\}\,\subset\,\QSH{Q}{A}.
\]

We prove that $\mathcal{L}_\mathcal{B}$ is a localising subcategory of $\QSH{Q}{A}$. First notice that $\mathcal{L}_\mathcal{B}$ is closed under coproducts. Indeed, if $\{M_i\}_{i\in I}$ is a set-indexed family of objects in $\qb{Q}$, then the evaluation functor $E_q$ preserves coproducts (as a left adjoint to the functor $G_q(-)$, see \cite[Cor.~3.9]{HJ-JLMS}), that is $(\coprod_{i\in I}M_i)(p)\cong\coprod_{i\in I}(M_i(p))$ in $\lMod{A}$ and thus $(\coprod_{i\in I}M_i)(p)$ lies in $\mathcal{B}$ since $\mathcal{B}$ is closed under coproducts. We conclude that $\coprod_{i\in I}M_i$ lies in $\qb{Q}$.\,

For the closure under negative suspensions, let $M$ be an object in $\qb{Q}$ and consider the conflation in the objectwise split exact structure
\[
0\longrightarrow \Omega M\longrightarrow \mathbb{F}(M)\longrightarrow M\longrightarrow0.
\]
Evaluating at an object $p\in Q$, yields a split short exact sequence
\begin{equation}\label{eq:split_cosyzygy}
0\longrightarrow (\Omega M)(p)\longrightarrow \mathbb{F}(M)(p)\longrightarrow M(p)\longrightarrow0
\end{equation}
in $\lMod{A}$ and hence $\mathbb{F}(M)(p)\,\cong\,M(p)\oplus(\Omega M)(p)$. Notice that $\mathbb{F}(M)(p)=\coprod_{q\in Q}Q(q,p)\otimes M(q)$. Since $Q$ is Hom-finite, $Q(q,p)$ is a finitely generated projective $\Bbbk$-module, thus there exists a finitely generated projective $\Bbbk$-module $P'$ and a non-negative integer $d\ge0$ such that $Q(q,p)\oplus P'\cong \Bbbk^d$. Tensoring (over $\Bbbk$) with the $A$-module $M(q)$ one has
\[
(Q(q,p)\otimes M(q))\oplus (P'\otimes M(q))\,\cong\, \Bbbk^d\otimes M(q)\,\cong\, M(q)^d.
\]
Thus, $Q(q,p)\otimes M(q)$ lies in $\mathcal{B}$ and so does $\mathbb{F}(M)(p)$ and since \eqref{split_cosyzygy} splits, $(\Omega M)(p)$ is a summand of an object in $\mathcal{B}$, and hence it lies in there as well. We conclude that for an $M$ in $\qb{Q}$ both $\mathbb{F}(M)$ and $\Omega M$ belong in $\qb{Q}$. \,

For positive suspensions, let $M$ lie in ${}_Q\mathcal{B}$ and choose a conflation in the objectwise split exact structure
\[
0\longrightarrow M\longrightarrow\mathbb{G}(M)\longrightarrow \Omega^{-1}M\longrightarrow0.
\]
Evaluating at an object $p\in Q$, yields a split exact sequence
\[
0\longrightarrow M(p)\longrightarrow\mathbb{G}(M)(p)\longrightarrow (\Omega^{-1}M)(p)\longrightarrow0
\]
in $\lMod{A}$. Notice that by \cite[Lem.~3.4]{HJ-TAMS} and the local boundedness condition on $Q$, one has $$\mathbb{G}(M)(p)=\left(\prod_{q\in Q}G_q(M)\right)(p)\cong \left(\coprod_{q\in Q}F_{\mathbb{S}q}(M)\right)(p),$$hence the previous arguments can be applied also in this case to show that $\Omega^{-1} M$ lies in ${}_Q\mathcal{B}$.\,

To prove that $\mathcal{L}_\mathcal{B}$ is closed under cones, consider a triangle 
\[
[M]\stackrel{[f]}{\longrightarrow}[N]\longrightarrow C_{[f]}\longrightarrow\Sigma[M],
\]
 in $\QSH{Q}{A}$ with $M,N$ in $\qb{Q}$. We represent $[f]$ by a morphism $f\colon M\to N$ in $\lMod{Q,A}$ (note that there is no need for bifibrant replacement here, since every object is bifibrant), thus the above triangle gives a conflation (see \cite[Thm.~I.2.6]{Happel})
 \[
 0\longrightarrow M\longrightarrow N\oplus \mathbb{F}(M)\longrightarrow C\longrightarrow0
 \]
 where $C$ is the pushout of the map $f$ along the map $M\to \mathbb{F}(M)$ and $C_{[f]}\cong [C]$ in $\QSH{Q}{A}$. Evaluating the conflation above at an object $p\in Q$, one has $N(p)\oplus\mathbb{F}(M)(p)\cong M(p)\oplus C(p)$, where $N(p), \, M(p)$ and $ \mathbb{F}(M)(p)$ (by the arguments presented in the closure under suspensions), all lie in $\mathcal{B}$, thus so does $C(p)$, and consequently, $C$ belongs in $\qb{Q}$.\,
 
 We conclude that $\mathcal{L}_\mathcal{B}$ is a localising subcategory of $\QSH{Q}{A}$. Notice that $Y$ lies in $\qb{Q}$, since $Y(p)$ is in $\mathcal{B}$ for any $p\in Q$ by construction thus
 \[
 \QSH{Q}{A}=\operatorname{Loc}(\{Y\})\subseteq\mathcal{L}_\mathcal{B}\subseteq\QSH{Q}{A},
 \]
 so $\QSH{Q}{A}\,=\,\mathcal{L}_\mathcal{B}$, meaning that every object in $\QSH{Q}{A}$ is isomorphic to one represented by an object in $\qb{Q}$. We will use this density argument to show that $\lMod{A}$ is additively generated by the subcategory $\mathcal{B}=\operatorname{Add}_A(X)$, and thus $A$ is a pure semisimple ring.\,

 Let $M$ be an $A$-module and $q$ an object in $Q$. Consider the stalk functor $S_q(M)$ in $\lMod{Q,A}$. Since $\QSH{Q}{A}\,=\,\mathcal{L}_\mathcal{B}$, there exists an object $R$ in $\qb{Q}$ such that $[S_q(M)]\,\cong\,[R]$ in $\QSH{Q}{A}$. We represent this isomorphism by a morphism $a\colon S_q(M)\to R$ in $\lMod{Q,A}$. Notice that, there are morphisms $a\colon S_q(M)\to R$ and $b\colon R\to S_q(M)$ such that, $[b][a]\cong \operatorname{id}_{S_q(M)}$ in $\QSH{Q}{A}$. Equivalently, $ba-\operatorname{id}_{S_q(M)}$ factors through a projective/injective object in the Frobenius exact category of the bifibrant objects (which in our case is the whole $\lMod{Q,A}$).  In \cite[Prop.~4.17]{SV_Q_shaped_homotopy} it is shown that the projective/injective objects are summands of coproducts of objects of the form $F_p(N)$ for $p$ in $ Q$ and $N$ in $\lMod{A}$. Assume that $ba-\operatorname{id}_{S_q(M)}=:h$ factors through the projective/injective object $P$ that is a direct summand of an object $$D=\bigoplus_{p\in Q}F_p(N)=\bigoplus_{p\in \mathsf{N}_{-}(q)}F_p(N),$$ that is, there is a retraction $P\stackrel{j}{\longrightarrow}D\stackrel{r}{\longrightarrow}P$, with $rj=\operatorname{id}_P$. Notice that, we index the direct sum over the objects in $\mathsf{N}_{-}(q)$, since outside of this finite set, the values of $F_p(-)$ are zero. Diagrammatically, we have
 \begin{equation}\label{eq:stalk_id_decomp}
 \begin{tikzcd}
	{S_q(M)} && {S_q(M)} \\
	& P \\
	& D
	\arrow["{h=ba-\operatorname{id}_{S_q(M)}}", from=1-1, to=1-3]
	\arrow["x"', from=1-1, to=2-2]
	\arrow["u"{description}, bend right=20pt, from=1-1, to=3-2]
	\arrow["y"', from=2-2, to=1-3]
	\arrow["j"', shift right=2, from=2-2, to=3-2]
	\arrow["v"{description}, bend right=20pt, from=3-2, to=1-3]
	\arrow["r"', shift right=2, from=3-2, to=2-2]
\end{tikzcd}\end{equation}
 Evaluating at an object $\ell$ of $Q$ we have $D(\ell)=\bigoplus_{p\in\mathsf{N}_{-}(\ell)}F_p(N)(\ell)$. Let us denote by 
 \[
 \iota_p\colon F_p(N)\longrightarrow D \quad\text{and}\quad \pi_p\colon D\longrightarrow F_p(N)
 \]
 the canonical inclusions and projections. At the object $\ell\in Q$, we have $\operatorname{id}_{D(\ell)}=\sum_{p\in \mathsf{N_{-}(\ell)}}(\iota_p)_\ell (\pi_p)_\ell$. Since $h=vu$ as seen in the diagram \eqref{stalk_id_decomp}, evaluating at the object $\ell\in Q$, one has that the $\ell$-th summand of $h$ is given by a finite sum:
 \[
 h_\ell=v_\ell u_\ell=v_\ell\operatorname{id}_{D(\ell)}u_\ell=v_\ell\left(\sum_{p\in\mathsf{N}_{-}(\ell)}(\iota_p)_\ell(\pi_p)_\ell\right)u_\ell=\sum_{p\in\mathsf{N}_{-}(\ell)}(v\iota_p)_\ell(\pi u_p)_\ell\,.
 \]
 

 We claim that each component $h_\ell,\, \ell\in Q$ of $h$ vanishes. For $\ell\neq q$, it is clear that $S_\ell(M)=0$, thus our claim holds. \,
 
  We now prove that $h=ba-\operatorname{id}_{S_q(M)}=0$, even for $\ell=q$. For this purpose we use the combinatorial structure of $Q$. Indeed by the Serre duality (denote by $\mathbb{S}$ the Serre functor of $Q$), one has
 \[
 Q(q,\mathbb{S}q)\,\cong\,\Hom{\Bbbk}{Q(q,q)}{\Bbbk}\,\cong\, \Hom{\Bbbk}{\Bbbk}{\Bbbk}\,\cong\,\Bbbk,
 \]
 where the first isomorphism follows from the Serre duality on $Q$ while the second one follows from the retraction property of $Q$ (see \cite[Stp.~2.5(4)]{HJ-JLMS}) and the acyclicity of $Q$, meaning that $\mathfrak{r}_q=0$. Moreover, since $Q$ is acyclic, one has that $q\neq \mathbb{S}q$. Consider a map $\phi\colon S_q(M)\to F_q(N)$ and evaluate at the object $q$ in $Q$ (otherwise it is the zero morphism), that is there is a map $\phi_q\colon M\to F_q(N)(q)$, where $F_q(N)(q)=Q(q,q)\otimes N=N$. Let $\alpha\in Q(q,\mathbb{S}q)$ be the generator corresponding to the unit $\mathbf{1}$ of $\Bbbk$. Then, there is a natural transformation 
 \[\begin{tikzcd}
	q & {S_q(M)(q)=M} & {F_q(N)(q)=N} \\
	{\mathbb{S}q} & {S_q(M)(\mathbb{S}q)=0} & {F_q(N)(\mathbb{S}q)\cong N}
	\arrow["\alpha"', from=1-1, to=2-1]
	\arrow["{\phi_q}", from=1-2, to=1-3]
	\arrow["{S_q(M)(\alpha)}"', from=1-2, to=2-2]
	\arrow["{F_q(N)(\alpha)}", from=1-3, to=2-3]
	\arrow["{\phi_{\mathbb{S}q}}"', from=2-2, to=2-3]
\end{tikzcd}\]
 that is, $F_q(N)(\alpha)\circ\phi_q\,=\,\phi_{\mathbb{S}q}\circ S_q(M)(\alpha)$. However, since $\mathbb{S}q\neq q$ we have $S_q(M)(\mathbb{S}q)=0$ and so $F_q(N)(\alpha)\circ\phi_q=0$. Under the identification
 \[
 F_q(N)(\mathbb{S}q)\,=\,Q(q,\mathbb{S}q)\otimes N \,\cong\, \Bbbk\otimes N\,\cong\, N
 \]
 the morphism $F_q(N)(\alpha)$ is the identity on $N$ which forces $\phi_q=0$ and since $\supp{S_q(M)}=\{q\}$, one has $\phi=0$ at every $q\in Q$. We conclude that $\Hom{Q,A}{S_q(M)}{F_q(N)}=0$ for every $M,N$ in $\lMod{A}$. Thus, $ba-\operatorname{id}_{S_q(M)}=0$ even in the model category. Evaluating the functor $S_q(M)$ at the vertex $q$ in $Q$, we have that $M$ is a direct summand of $R(q)$ and thus $M$ lies in $\mathcal{B}$. Since $M$ was an arbitrary $A$-module, we conclude that $\lMod{A}=\mathcal{B}=\operatorname{Add}_A(X)$. This concludes the proof.
\end{proof}

\subsection{Homotopy categories of complexes}
Intuitively, the phenomenon described in Theorems \ref{thm:homotopy_cg} and Proposition \ref{prp:comp_gen_q_homotopy}, implies a limit of compact generation that is controlled by purity. In the special case that $Q=Q^{\text{cpx}}$ and $\lMod{Q,A}$ is the category of complexes of $A$-modules, we follow this intuition and find ``cardinal bounds'' that guarantee the well generation of such homotopy categories. However, as we show later on in subsection \ref{subsec:applications-generation}, there are exact structures that are controlled by classes of modules that are strictly larger than the finitely presented ones, who induce exact model structures whose homotopy category is still compactly generated.  

\begin{dfn}\label{dfn:kappa_presented}
    Let $\kappa$ be a regular cardinal. We denote by $\lMode{A}{\kappa}$ the class of all \textbf{$\kappa$-presented} $A$-modules. We denote by $\lPPrj{A}^\kappa$ the additive closure of all $\kappa$-presented $A$-modules and we call them \textbf{$\kappa$-pure-projective modules}. 
\end{dfn}

\begin{rmk}\label{rmk:set_of_representatives}
Notice that for $\kappa=\aleph_0$ in Definition \ref{dfn:kappa_presented} one recovers the notions of finitely presented and pure-projective modules. Moreover, in the ``Notation'' environment below \cite[Rmk.~1.19]{AdamekRosicky} in loc. cit., all $\kappa$-presentable objects in a locally $\kappa$-presentable category admit a \emph{set} of representatives. Thus, $\lPPrj{A}^\kappa=\operatorname{Add}(\lMod{A}^\kappa)$ is additively generated by a set of $\kappa$-presented modules.    
\end{rmk}

\begin{lem}\label{lem:add_set_bounded}
    Let $\theta\subseteq\lMod{A}$ be a class of modules. The following are equivalent.
    \begin{prt}
        \item There exists a regular cardinal $\kappa$, such that $\theta\subseteq\lPPrj{A}^\kappa$;
        \item There exists a set $\mathcal{S}\subseteq\lMod{A}$ of $A$-modules, such that $\theta\subseteq\operatorname{Add}(\mathcal{S)}$.
    \end{prt}
\end{lem}

\begin{proof}
    \proofoftag{a)$\Rightarrow$ (b} Let $\theta\subseteq\lPPrj{A}^\kappa$. Since $\lPPrj{A}^\kappa=\operatorname{Add}(\lMod{A}^\kappa)$ by Remark \ref{rmk:set_of_representatives}, one can always find a set of representatives such that $\lPPrj{A}^\kappa$ is the additive closure of a set.\,
    \proofoftag{b)$\Rightarrow$ (a} Conversely, suppose that $\theta\subseteq\operatorname{Add}(\mathcal{S})$ for some set $\mathcal{S}$. Since $\lMod{A}$ is locally presentable and $\mathcal{S}$ is a set, there exists a sufficiently large cardinal regular cardinal $\kappa$, such that every $S\in\mathcal{S}$ is $\kappa$-presented. Thus, $\mathcal{S}\subseteq\lMod{A}^\kappa$. Taking additive closures one has
    \[
    \theta\subseteq\operatorname{Add}(\mathcal{S})\subseteq\operatorname{Add}(\lMod{A}^\kappa)\,=\,\lPPrj{A}^\kappa.
    \]
    This completes the proof.
\end{proof}

\begin{thm}\label{thm:negative_compact_generation}
Let $\theta$ be a class of $A$-modules satisfying
\stpref{HJ+theta}. If there exists a set $\mathcal{S}$ of $A$-modules, such that $\theta\subseteq\operatorname{Add}(\mathcal{S})$,
then $\QST{Q}{A}$ is a well generated triangulated category.
In the special case that $Q=Q^{\text{cpx}}$, the following conditions are equivalent.
\begin{prt}
\item $\mathbf{K}^{\theta}(A)$ is well generated;
\item $\mathbf{K}^{\theta}(A)$ admits a generating set of objects;
\item there exists a regular cardinal $\kappa$ such that $\theta\subseteq\lPPrj{A}^{\kappa}.$
\end{prt}
\end{thm}

\begin{proof}
By \lemref{add_set_bounded}, the existence of a set
$\mathcal{S}\subseteq\lMod{A}$ such that
\(
\theta\subseteq\operatorname{Add}(\mathcal{S}),
\) implies the existence of a regular cardinal $\kappa$ such that $\theta\subseteq\lPPrj{A}^\kappa$.
Choose another regular cardinal $\lambda$ such that every object of
$\mathcal{S}$ is $\lambda$-presented. After replacing $\lambda$ by a
larger regular cardinal if necessary, the $\lambda$-presented version
of Walker's theorem \cite[Thm.~26.1]{Anderson-Fuller}, that can be found in \cite[Lem.~2.2]{Cortes-Izurdiaga-Poveda} implies that every object $T$ in $\theta$ admits
a decomposition
\[
T\cong\bigoplus_{i\in I}T_i,
\]
where every $T_i$ is a $\lambda$-presented $A$-module. This follows from the fact that every object in $\theta$ is a direct summand of a coproduct of modules in $\mathcal{S}$, which are $\lambda$-presented. Since $\theta$
is chosen to be saturated, it is closed under direct summands, and therefore every
$T_i$ belongs to $\theta$.\,

Let $\mathcal{S}_{\theta}$ be a set of representatives of the
$\lambda$-presented $A$-modules which belong to $\theta$. The
preceding decomposition, together with the closure of $\theta$ under
coproducts and direct summands, gives
\[
\theta=\operatorname{Add}(\mathcal{S}_{\theta}).
\]

Let for brevity $\mathcal{E}=\lMode{Q,A}{\theta}$ denote the exact category $\lMod{Q,A}$ endowed with the objectwise $\theta$-exact structure, and consider the set
\[
\mathscr{G}_0
:=
\{\,\Sq{q}(T)\mid T\in\mathcal{S}_{\theta},\ q\in Q\,\}.
\]
We claim that the cotorsion pair
$(\mathcal{C}_{\theta},\mathcal{W}_{\theta})$ is generated by
$\mathscr{G}_0$. Our goal is to show that this cotorsion pair induces a combinatorial model structure.\,

Since $\lMode{A}{\theta}$ is efficient, it has exact coproducts, see \cite[Lem.~1.4]{SaorinStovicek}.
Consequently, by \cite[Prop.~4.17]{SV_Q_shaped_homotopy}, the exact category
$\mathcal{E}$ also has exact coproducts and enough projective objects. It follows
that, for every family $\{X_i\}_{i\in I}$ of objects in $\mathcal{E}$
and every object $Y$ in $\mathcal{E}$, there is a natural isomorphism
\begin{equation}\label{eq:ext-iso}
\Ext{\mathcal{E}}{1}{\bigoplus_{i\in I}X_i}{Y}
\cong
\prod_{i\in I}\Ext{\mathcal{E}}{1}{X_i}{Y}.
\end{equation}
Indeed, one can choose projective presentations of the objects $X_i$
and take their coproduct. Exactness of coproducts implies that the
resulting sequence is again a projective presentation.

The functors $\Sq{q}$ preserve coproducts and direct summands. Since
$\theta=\operatorname{Add}(\mathcal{S}_{\theta})$, one has
\[
\mathscr{G}_{\theta}
=
\{\,\Sq{q}(T)\mid T\in\theta,\ q\in Q\,\}
\subseteq
\operatorname{Add}(\mathscr{G}_0).
\]
The reverse inclusion
$\mathscr{G}_0\subseteq\mathscr{G}_{\theta}$ is immediate. The Ext-isomorphism \eqref{ext-iso} therefore gives $\mathscr{G}_{\theta}^{\perp}=\mathscr{G}_0^{\perp}$.
By \stpref{HJ+theta}, the cotorsion pair
$(\mathcal{C}_{\theta},\mathcal{W}_{\theta})$ is generated by
$\mathscr{G}_{\theta}$. Hence
\[
\mathcal{W}_{\theta}
=
\mathscr{G}_{\theta}^{\perp}
=
\mathscr{G}_0^{\perp},
\]
so the same cotorsion pair is generated by the set $\mathscr{G}_0$.\,

We next show that $\mathcal{E}$ has a set of projective generators. Consider the set
\[
\mathscr{P}_0
:=
\{\,\Fq{q}(T)\mid T\in\mathcal{S}_{\theta},\ q\in Q\,\}.
\]
The functors $\Fq{q}$ preserve coproducts and direct summands.
Therefore, by \cite[Prop.~4.17]{SV_Q_shaped_homotopy},
\[
\Prj{\mathcal{E}}
=
\operatorname{Add}
\{\,\Fq{q}(T)\mid T\in\theta,\ q\in Q\,\}
=
\operatorname{Add}(\mathscr{P}_0).
\]

For every $G\in\mathscr{G}_0$, choose a conflation
\begin{equation}\label{eq:G_0-conflation}
0\longrightarrow K_G
\xrightarrow{i_G}
P_G
\xrightarrow{\pi_G} G
\longrightarrow 0
\end{equation}
with $P_G$ projective. Notice that $\pi_G$ is, in particular, a projective precover of $G$: every morphism from a projective object to $G$ lifts through $\pi_G$. Define two sets of morphisms
\[
J
:=
\{\,0\longrightarrow P\mid P\in\mathscr{P}_0\,\}
\quad\text{and}\quad
I
:=
J\cup
\{\,i_G\colon K_G\longrightarrow P_G
\mid G\in\mathscr{G}_0\,\}.
\]

Denote by $J^{\square}$ the class of all morphisms that have the right lifting property with respect to the morphisms in $J$. We claim that $J^\square$ is precisely the class of
deflations in $\mathcal E$. Indeed, a morphism
$p\colon X\longrightarrow Y$ has the right lifting property with respect to $0\longrightarrow P$ if and only if
\[
\Hom{\mathcal E}{P}{X}
   \longrightarrow
\Hom{\mathcal E}{P}{Y}
\]
is surjective. Hence every deflation belongs to
$J^\square$, since every object of $\mathcal P_0$ is
projective.

Conversely, let $p\colon X\longrightarrow Y$ belong to
$J^\square$. Since
$\Prj{\mathcal E}=\operatorname{Add}(\mathcal P_0)$, the preceding surjectivity also holds for every projective object $P$ of $\mathcal E$. Choose a deflation $e\colon P'\longrightarrow Y$ with $P'$ projective. There is therefore
a morphism $h\colon P'\longrightarrow X$ such that $ph=e$. Since the underlying additive category $\lMod{Q,A}$ is abelian, the morphism $p$ admits a kernel. As $ph=e$ is a deflation, the obscure axiom \cite[Prop.~2.16]{Buhler} implies that $p$ is a deflation. Thus
\[
J^\square=\{\text{deflations in }\mathcal E\}.
\]

Let now $p\colon X\twoheadrightarrow Y$ be a deflation, put
$L=\Ker{p}$, and denote the kernel inflation by
$j\colon L\longrightarrow X$. Fix $G\in\mathscr{G}_0$ and consider
the chosen conflation from equation \eqref{G_0-conflation},
where $P_G$ is projective. 

Applying $\Hom{\mathcal E}{-}{L}$ to this conflation gives the exact
sequence
\[
\Hom{\mathcal E}{P_G}{L}
\xrightarrow{i_G^{*}}
\Hom{\mathcal E}{K_G}{L}
\xrightarrow{\delta}
\Ext{\mathcal E}{1}{G}{L}
\longrightarrow
\Ext{\mathcal E}{1}{P_G}{L}
=0,
\]
see, for example, \cite[Sec.~12]{Buhler}. Here
$i_G^{*}(d)=di_G$. Consequently,
\begin{equation}\label{eq:ext_property}
\Ext{\mathcal E}{1}{G}{L}=0
\quad\Longleftrightarrow\quad
\text{every morphism }s\colon K_G\longrightarrow L
\text{ extends via $i_G$ to a morphism }d\colon P_G\longrightarrow L.
\end{equation}

However, the extension property in \eqref{ext_property} is equivalent to $p$ having the right lifting property with respect to $i_G$, or in symbols: 
$p\in\{i_G\}^\square$. Moreover, since $J^\square$ is the class of deflations and
$I=J\cup\{\,i_G\mid G\in\mathscr{G}_0\,\}$, it follows that
\[
I^\square
=
\left\{\,p\colon X\twoheadrightarrow Y
\ \middle|\
\Ext{\mathcal E}{1}{G}{\Ker{p}}=0
\text{ for every }G\in\mathscr{G}_0
\,\right\}
=
\left\{\,p\colon X\twoheadrightarrow Y
\ \middle|\
\Ker{p}\in\mathcal{W}_{\theta}
\,\right\}.
\]

The morphisms in $J^\square$ are therefore 
the fibrations of the projective model structure, while the morphisms in $I^\square$ are its trivial
fibrations. Hence the model structure is cofibrantly generated by the sets $I$ and $J$.

The exact structure does not change the underlying additive category of $\mathcal{E}$, which is the locally presentable category $\lMod{Q,A}$. Thus, the model structure is combinatorial. It is also stable, since it is a hereditary exact model structure. Rosicky's well-generation Theorem \cite{Rosicky_Brown_rep_in_hmtp} for stable combinatorial model categories now implies that
\[
\QST{Q}{A}
=
\operatorname{Ho}(\lMode{Q,A}{\theta})
\]
is a well generated triangulated category. This proves the first assertion and, when $Q=Q^{\text{cpx}}$, the implication
(c)$\Rightarrow$(a).

Assume now that $Q=Q^{\text{cpx}}$. The implication
(a)$\Rightarrow$(b) is immediate. To prove (b)$\Rightarrow$(c), we
argue by contraposition. Suppose that $\theta\nsubseteq\lPPrj{A}^{\kappa}$ for every regular cardinal $\kappa$. We will show that $\mathbf{K}^{\theta}(A)$ admits no generating set of objects. For this purpose, let
$\mathcal{G}\subset\mathbf{K}^\theta(A)$ be an arbitrary set of
objects. We will construct a non zero object $X$ in
$\mathbf{K}^\theta(A)$ such that
\[
\Hom{\mathbf{K}^\theta(A)}{G}{X[n]}=0
\]
for every $G\in\mathcal{G}$ and $n\in\mathbb{Z}$. For each $G'\in\mathcal{G}$, choose a cofibrant representative
$(G')^{\mathrm{cof}}$ with respect to the exact model structure
induced by $\theta$, and consider also the set of $A$-modules
\[
\mathcal{S}_{\mathcal{G}}
:=
\{\,((G')^{\mathrm{cof}})^n\in\lMod{A}
\mid G'\in\mathcal{G},\ n\in\mathbb{Z}\,\},
\]
where by $(-)^n$ we denote the degree $n$ term of a complex. Since $\lPPrj{A}$ is additively generated by the set $\lsmod{A}$ of
all finitely presented $A$-modules, one has that the class
\[
\mathcal{A}
:=
\operatorname{Add}
(\lsmod{A}\,\cup\,\mathcal{S}_{\mathcal{G}})
\]
is also additively generated by a set. \,

Recall that, $\theta$
is not Add-set-bounded, that is
$\theta\nsubseteq\operatorname{Add}(\mathcal{S})$ for any set
$\mathcal{S}$. In particular, $\theta\nsubseteq\mathcal{A}$, i.e.\
there exists a module $M$ in $\theta$ such that
$M\notin\mathcal{A}$. Let
\[
P_{\text{pure}}
=
\bigoplus_{\substack{f\in\Hom{A}{F}{M}\\ F\in\lsmod{A}}}
F
\longrightarrow M
\]
be a pure-projective precover of $M$. We extend this precover
by forcing also all maps from $\mathcal{S}_{\mathcal{G}}$ to lift.
Define the module
\[
P
=
P_{\text{pure}}
\oplus
\bigoplus_{\substack{f\in\Hom{A}{S}{M}\\
S\in\mathcal{S}_{\mathcal{G}}}}
S
\]
and define a map $\pi\colon P\longrightarrow M$ such that, on the
first summand $P_{\text{pure}}$, it identifies with the map
$P_{\text{pure}}\to M$, and on the second summand it is identified
with the map $f\colon S\to M$.

One has that $\pi$ is a pure-epimorphism, since
$P_{\text{pure}}\to M$ is a pure-epimorphism. Indeed, let
$F\in\lsmod{A}$ and consider a map $g\colon F\to M$. Notice that $g$
corresponds to a summand $F_g$ of $P_{\text{pure}}$, and consider the
natural inclusion $i_g\colon F_g\to P_{\text{pure}}$. By definition
of $\pi$, one has $\pi\circ i_g=g$. Thus, every map $F\to M$ lifts
through $\pi$, hence for every $F\in\lsmod{A}$ the map
\[
\Hom{A}{F}{\pi}\colon
\Hom{A}{F}{P}
\longrightarrow
\Hom{A}{F}{M}
\]
is surjective and $\pi$ is a pure-epimorphism. Using similar arguments
one can show that every map $S\to M$ with
$S\in\mathcal{S}_{\mathcal{G}}$ also lifts through $\pi$. We complete
$\pi$ into a pure exact sequence of $A$-modules
\[
\xi\colon
0\longrightarrow K
\stackrel{i}{\longrightarrow}P
\stackrel{\pi}{\longrightarrow}M
\longrightarrow0
\]
with $K=\Ker{\pi}$. We claim that $\xi$ is not $\theta$-exact.
Suppose for a contradiction that it is. Since $M\in\theta$ but
$M\notin\mathcal{A}$, applying the functor $\Hom{A}{M}{-}$ to $\xi$
would give an exact sequence
\[
0\longrightarrow
\Hom{A}{M}{K}
\longrightarrow
\Hom{A}{M}{P}
\longrightarrow
\Hom{A}{M}{M}
\longrightarrow0.
\]
In particular, the epimorphism
\[
\Hom{A}{M}{P}\longrightarrow\Hom{A}{M}{M}\longrightarrow0
\]
implies that the identity $\operatorname{id}_M$ lifts through $\pi$,
meaning that $M$ is a direct summand of $P$. However, by construction,
$P$ lies in $\mathcal{A}$ and thus so must $M$. This is a
contradiction, so $\xi$ is not $\theta$-exact. \,

Consider now the three-term complex
\[
\Delta(\xi)
=
(\cdots\longrightarrow0\longrightarrow K\longrightarrow P
\longrightarrow M\longrightarrow0\longrightarrow\cdots)
\]
in the category $\mathbf{Ch}(A)$ of complexes of $A$-modules, where
$M$ sits in degree zero. We denote this complex by $X=\Delta(\xi)$.
Since $\xi$ is pure exact, $X$ is a pure-acyclic complex. On the other
hand, since $\xi$ is not $\theta$-exact, $X$ is not a
$\theta$-trivial object. Thus, $X$ is non zero in
$\mathbf{K}^{\theta}(A)$. \,

It remains to show that
\[
\Hom{\mathbf{K}^{\theta}(A)}{G}{X[n]}=0
\]
for every $G\in\mathcal{G}$ and $n\in\mathbb{Z}$. Let
$G'\in\mathcal{G}$ and write $G=(G')^{\mathrm{cof}}$ for the chosen cofibrant representative of $G'$. Since $G$ is cofibrant and every object is fibrant, a morphism
$G\to X[n]$ in $\mathbf{K}^{\theta}(A)$ can be represented by a chain
map
\[
f\colon G\longrightarrow\Delta(\xi)[n].
\]
Suppose first that $n=0$. Then $f$ has in degree zero an $A$-linear map $f^0\colon G^0\longrightarrow M.$ Since $G^0\in\mathcal{S}_{\mathcal{G}}$, the map $f^0$ lifts through
$\pi\colon P\to M$. Choose a lift $\widetilde{f}^{\,0}\colon G^0\to P$ such that
$\pi\widetilde{f}^{\,0}=f^0$. Equivalently, by forming the pullback
of $\xi$ along $f^0$,
\begin{equation}\label{eq:pullback}
\begin{tikzcd}
0 & K & {P_f} & {G^0} & 0 \\
0 & K & P & M & 0
\arrow[from=1-1, to=1-2]
\arrow[from=1-2, to=1-3]
\arrow[equals, from=1-2, to=2-2]
\arrow[dashed, from=1-3, to=1-4]
\arrow[dashed, from=1-3, to=2-3]
\arrow["\lrcorner"{anchor=center, pos=0.125}, draw=none,
       from=1-3, to=2-4]
\arrow[from=1-4, to=1-5]
\arrow["{f^0}", from=1-4, to=2-4]
\arrow[from=2-1, to=2-2]
\arrow[from=2-2, to=2-3]
\arrow[from=2-3, to=2-4]
\arrow[from=2-4, to=2-5]
\end{tikzcd}
\end{equation}
one gets that the upper short exact sequence splits in $\lMod{A}$.
We denote this split exact sequence by $\xi_f$ and its three-term
complex by $\Delta(\xi_f)$. Since $\xi_f$ splits, the complex
$\Delta(\xi_f)$ is contractible. \,

We claim that the map
$f\colon G\longrightarrow\Delta(\xi)$ factors through
$\Delta(\xi_f)$. The only non zero terms of $f$ are in degrees
$-2,-1$ and $0$, or diagrammatically,
\[
\begin{tikzcd}
\cdots & {G^{-3}} & {G^{-2}} & {G^{-1}} & {G^0} & {G^1} & \cdots \\
\cdots & 0 & K & P & M & 0 & \cdots
\arrow[from=1-1, to=1-2]
\arrow[from=1-2, to=1-3]
\arrow[from=1-2, to=2-2]
\arrow["{d_G^{-2}}", from=1-3, to=1-4]
\arrow["{f^{-2}}"', from=1-3, to=2-3]
\arrow["{d_G^{-1}}", from=1-4, to=1-5]
\arrow["{f^{-1}}"', from=1-4, to=2-4]
\arrow[from=1-5, to=1-6]
\arrow["{f^0}"', from=1-5, to=2-5]
\arrow[from=1-6, to=1-7]
\arrow[from=1-6, to=2-6]
\arrow[from=2-1, to=2-2]
\arrow[from=2-2, to=2-3]
\arrow["i"', from=2-3, to=2-4]
\arrow["\pi"', from=2-4, to=2-5]
\arrow[from=2-5, to=2-6]
\arrow[from=2-6, to=2-7]
\end{tikzcd}
\]
By the commutativity of the chain map, one has
\begin{equation}\label{eq:commutativity}
i\circ f^{-2}
=
f^{-1}\circ d_G^{-2}
\quad\text{and}\quad
\pi\circ f^{-1}
=
f^0\circ d_G^{-1}.
\end{equation}

The pullback $P_f=P\times_MG^0$ is taken in the category of
$A$-modules, and it can be described as
\[
P_f
=
\{\,(p,x)\in P\times G^0\mid\pi(p)=f^0(x)\,\}.
\]
The maps in $\xi_f$ are given by
\[
\alpha\colon K\longrightarrow P_f,
\qquad
k\longmapsto(i(k),0),
\]
and
\[
\beta\colon P_f\longrightarrow G^0,
\qquad
(p,x)\longmapsto x.
\]
Define a chain map
\[
\widetilde{f}\colon G\longrightarrow\Delta(\xi_f)
\]
whose non zero terms are
\[
\widetilde{f}^{\,0}=\operatorname{id}_{G^0},
\qquad
\widetilde{f}^{\,-2}=f^{-2},
\]
and
\[
\widetilde{f}^{\,-1}\colon G^{-1}\longrightarrow P_f,
\qquad
x\longmapsto
\bigl(f^{-1}(x),d_G^{-1}(x)\bigr).
\]
This map is well-defined by the second equality in
\eqref{commutativity}. In degree $-1$, one has
\[
\beta\widetilde{f}^{\,-1}(x)
=
\beta\bigl(f^{-1}(x),d_G^{-1}(x)\bigr)
=
d_G^{-1}(x)
=
\widetilde{f}^{\,0}d_G^{-1}(x).
\]
In degree $-2$, one has
\[
\widetilde{f}^{\,-1}d_G^{-2}(y)
=
\bigl(
f^{-1}d_G^{-2}(y),
d_G^{-1}d_G^{-2}(y)
\bigr)
=
\bigl(i f^{-2}(y),0\bigr)
=
\alpha\widetilde{f}^{\,-2}(y).
\]
Thus, $\widetilde{f}$ is a chain map. Finally, the composite
\[
G
\stackrel{\widetilde{f}}{\longrightarrow}
\Delta(\xi_f)
\longrightarrow
\Delta(\xi)
\]
has terms
\[
(\ldots,0,f^{-2},f^{-1},f^0,0,\ldots),
\]
and therefore coincides with $f$. Since $\Delta(\xi_f)$ is
contractible, it is $\theta$-trivial. Thus, $f$ represents the zero
morphism in $\mathbf{K}^{\theta}(A)$, and hence
\(
\Hom{\mathbf{K}^{\theta}(A)}{G}{X}=0.
\)
Applying the same argument after shifting $G$ gives
\(
\Hom{\mathbf{K}^{\theta}(A)}{G}{X[n]}=0
\)
for every $n\in\mathbb{Z}$, since all the components of every shift
of $G$ belong to $\mathcal{S}_{\mathcal{G}}$. We have therefore constructed a non zero object
$X\in\mathbf{K}^{\theta}(A)$ such that
\[
\Hom{\mathbf{K}^{\theta}(A)}{G}{X[n]}=0
\]
for every $G\in\mathcal{G}$ and every $n\in\mathbb{Z}$. Hence
$\mathcal{G}$ is not a generating set. Since $\mathcal{G}$ was
arbitrary, $\mathbf{K}^{\theta}(A)$ admits no generating set of objects.\,

This proves the contrapositive of (b)$\Rightarrow$(c). Together with (c)$\Rightarrow$(a) and (a)$\Rightarrow$(b), this completes the proof.
\end{proof}

\begin{rmk}
In our proof we use Rosicky's theorem for our stable combinatorial model structure. Notice that Rosicky's theorem is formulated for \emph{strongly} $\lambda$-combinatorial model structures, for some regular cardinal $\lambda$. We do not claim that it is strongly $\kappa$-combinatorial for the regular cardinal $\kappa$ fixed above. By \cite[Rmk.~3.4(1)]{Rosicky_Brown_rep_in_hmtp}, every combinatorial model category is strongly $\mu$-combinatorial for some sufficiently
large regular cardinal $\mu$. Indeed, the functorial
cofibration--trivial-fibration factorisation is accessible, and the Uniformisation Theorem \cite[Thm.~2.19]{AdamekRosicky} (originally due to Makkai and Par\'e \cite{Makkai_Pare_accessible_cats}) allows one to choose $\mu$ so that this factorisation preserves $\mu$-presentable objects. Enlarging $\mu$ if necessary, we may also assume that $\lMod{Q,A}$ is locally $\mu$-presentable and that the domains and codomains of the morphisms in $I$ and $J$ are $\mu$-presentable. Thus, the model structure is
strongly $\mu$-combinatorial. It thus follows from \cite[Thm.~4.9]{Rosicky_Brown_rep_in_hmtp} that
$\operatorname{Ho}(\mathcal E)=\QST{Q}{A}$ is well
$\mu$-generated, and hence it is well generated in our sense.
\end{rmk}

\subsection{Applications}
\label{subsec:applications-generation}

As a first application, in the special case that $\theta=\lMod{A}$, one recovers a known result of Saor{\'\i}n and {\v{S}}{\v{t}}ov{\'\i}{\v{c}}ek  \cite[Cor.~4.10]{SaorinStovicek}, where the same condition as in Theorem \ref{thm:negative_compact_generation} is used to give a criterion for the homotopy category of complexes to be well generated.

We conclude this section by providing an example of an exact structure that differs from the ones we have seen so far, and such that, the class of relative projective modules will satisfy
\[
 \lPPrj{A}\subsetneq\theta\subsetneq\lMod{A}.
\]
This exact structure is given by finitely split exact sequences, whose associated homotopy category is compactly generated in the case of ordinary complexes and well generated in the general $Q$-shaped setting.  

The relative model structure on complexes used below is an instance of those studied by Christensen and Hovey in \cite{MR1912401}. In particular, their Theorem~5.1 in \emph{loc. cit.} characterises when such a model structure is cofibrantly generated and when its homotopy category admits a set of weak generators. Our purpose here is to verify the hypotheses of Setup \ref{stp:HJ+theta} for the exact structures in question and to express the resulting generation properties through the criterion of Theorem \ref{thm:negative_compact_generation}.

An epimorphism $p\colon M\longrightarrow N$ of $A$-modules is called \textbf{finitely split} if every morphism $F\longrightarrow N$ with $F$ finitely generated lifts through $p$. This notion was introduced by Azumaya in \cite{Azumaya_finite_splitness} and is also studied in \cite{Mao_finite_splitness_phantom_maps}. We consider the exact structure whose conflations are the short exact sequences
\[
 0\longrightarrow L\longrightarrow M\stackrel{p}{\longrightarrow}N\longrightarrow0
\]
with $p$ finitely split. Notice that the construction of this exact structure is equivalent to the functor $\Hom{A}{F}{-}$ to be exact for every finitely generated $A$-module $F$. These conflations are pure exact, since finitely presented modules are finitely generated. Over a non-noetherian ring, however, the two exact structures need not coincide.

The motivation for this choice is that finitely generated modules may have infinitely many relations. Testing with such modules retains information which can be lost when one tests only with finitely presented modules. The associated relative homotopy category provides a derived setting for this stronger lifting condition.  For what follows, assume that $\Bbbk$, $Q$ and $A$ satisfy the ring and category assumptions of Setup \ref{stp:HJ+theta}.

\begin{lem}\label{lem:applications-fg-setup}
Let $\mathcal S$ be a set of finitely generated $A$-modules containing $A$, and put $\theta=\operatorname{Add}_A(\mathcal S)$. Then $\theta$ is saturated and precovering, and $\lMode{A}{\theta}$ is efficient.  Consequently $\mathbf K_Q^{\theta}(A)$ is well generated. For $Q=Q^{\mathrm{cpx}}$, it is compactly generated by the stalks $S_q(F)$ with $F\in\mathcal S$ and $q\in\mathbb Z$.
\end{lem}

\begin{proof}
For an $A$-module $M$, consider the evaluation morphism
\[
 p_M\colon
 \bigoplus_{F\in\mathcal S}F^{(\Hom{A}{F}{M})}
 \longrightarrow M.
\]
Since $A\in\mathcal S$, this is an epimorphism. Every morphism from an object of $\mathcal S$ to $M$ lifts through $p_M$, and the same holds for coproducts and direct summands of such objects. Thus $p_M$ is a $\theta$-precover and a deflation in $\lMode{A}{\theta}$. The objects of $\theta$ are projective in this exact category. Conversely, if $P$ is projective, then $p_P$ splits, showing that $P\in\theta$. This proves saturation and the existence of enough projectives.

We verify efficiency in the sense of \cite[Def.~2.6]{SaorinStovicek}. Let $(M_\alpha)_{\alpha\leq\lambda}$ be a continuous ordinal-indexed chain whose successor maps are inflations. Its colimit exist in $\lMod{A}$, and all its maps are monomorphisms. We show by transfinite induction that $M_0\longrightarrow M_\alpha$ is an inflation. The successor step follows from composition of inflations. At a limit ordinal $\beta$, one has
\[
 M_\beta/M_0=\bigcup_{\alpha<\beta}M_\alpha/M_0.
\]
For $F\in\mathcal S$, every morphism $F\longrightarrow M_\beta/M_0$ has finitely generated image, and hence factors through $M_\alpha/M_0$ for some $\alpha<\beta$. By the induction hypothesis it lifts to $M_\alpha$, and therefore to $M_\beta$. It follows that $M_0\longrightarrow M_\beta$ is an inflation. The same argument applies to every tail of the chain, that is, for every $\gamma\le\beta$ the map $M_\gamma\to M_\beta$ is an inflation. So arbitrary transfinite compositions of inflations are inflations.

Every $A$-module is small relative to monomorphisms, and hence relative to the inflations under consideration. Since the exact category has enough projectives, \cite[Prop.~2.7]{SaorinStovicek} now proves that $\lMode{A}{\theta}$ is efficient. In particular, it has exact coproducts.

By \cite[Prop.~4.7]{SV_Q_shaped_homotopy}, the exact category $\lMode{Q,A}{\theta}$ is also efficient and has enough projectives. The class $\theta$ is the saturation of the set $\mathcal S$, so by \cite[Lem.~6.7]{SV_Q_shaped_homotopy} one has that the stalk class $\mathscr G_\theta$ generates a complete cotorsion pair. This verifies all the conditions of Setup \ref{stp:HJ+theta}. Well generation follows from Lemma \ref{lem:add_set_bounded} and Theorem \ref{thm:negative_compact_generation}.

For the last assertion, work with cochain complexes and let $S_q(F)$ denote the stalk in degree $q$. These stalks are cofibrant, every object is fibrant, and morphisms from a cofibrant object are computed up to chain homotopy; see \cite[Lem.~2.7 and Cor.~2.13]{MR1912401}. Consequently, for every complex $X$, there is an isomorphism
\[
 \Hom{\mathbf K^\theta(A)}{S_q(F)}{X}
 \cong H^q\Hom{A}{F}{X}.
\]
The functor $\Hom{A}{F}{-}$ commutes with coproducts because $F$ is finitely generated. Coproducts in the homotopy category are represented by degreewise coproducts: coproducts preserve the relative weak equivalences, and coproducts of cofibrant replacements are cofibrant replacements. Thus $S_q(F)$ is compact. Vanishing of the displayed groups for all $F\in\mathcal S$ and $q\in\mathbb Z$ says exactly that $X$ is trivial in the relative model structure. Hence these stalks compactly generate $\mathbf K^\theta(A)$.
\end{proof}

\begin{exa}\label{exa:applications-fg}
Let $\Bbbk=\overline{\mathbb F}_2$, let $G=\bigoplus_{i\geq1}C_2$, and put $A=\Bbbk G$. If $g_i$ is the generator of the $i$-th summand and $x_i=g_i-1$, then
\[
 A\cong\Bbbk[x_1,x_2,\ldots]/(x_1^2,x_2^2,\ldots).
\]
Let $\mathcal S_{\mathrm{fg}}$ be a set of representatives of the finitely generated $A$-modules, containing $A$, and set
\[
 \theta_{\mathrm{fg}}=\operatorname{Add}_A(\mathcal S_{\mathrm{fg}}).
\]
Such a set exists because every finitely generated module is a quotient of $A^n$ for some $n\geq0$. The $\theta_{\mathrm{fg}}$-exact structure is precisely the finitely split exact structure. By Lemma \ref{lem:applications-fg-setup}, it satisfies Setup \ref{stp:HJ+theta}; its $Q$-shaped homotopy category is well generated, and its homotopy category of ordinary complexes is compactly generated.

We claim that
\[
 \lPPrj{A}\subsetneq\theta_{\mathrm{fg}}\subsetneq\lMod{A}.
\]
The first inclusion follows because every finitely presented module is finitely generated. It is strict: the trivial $A$-module $\Bbbk=A/(x_1,x_2,\ldots)$ is cyclic but not finitely presented. Indeed, if $I=(x_1,x_2,\ldots)$, then $I/I^2$ has the infinite basis given by the classes of the $x_i$, so $I$ is not finitely generated. A finitely generated pure-projective module is a direct summand of a finite coproduct of finitely presented modules (see \cite{Pure-proj-modules} or the more general construction in Section \ref{subsec:Q-shaped-pure}), and is therefore finitely presented. Thus $\Bbbk\in\theta_{\mathrm{fg}}\setminus\lPPrj{A}$.

For the second inclusion, if $\lMod{A}=\operatorname{Add}_A(\mathcal S_{\mathrm{fg}})$, then $A$ would be pure semisimple by \cite[Prop.~2.2]{MR2657376}. Every $A$-module would then be pure-projective, contradicting the preceding paragraph. This proves the claim. Notice also that $A$ is countable. Every finitely generated $A$-module is therefore countably presented, giving the explicit bound 
\(
 \theta_{\mathrm{fg}}\subseteq\lPPrj{A}^{\aleph_1}.
\)
\end{exa}

\section{Perfect objects}\label{sec:perfect}

In compactly generated triangulated categories the compact objects are described as the thick closure of the compact generators. In applications, however, one wants representatives of such objects that are finitely presented in the underlying additive category, and easier to work with. A classical example is that of the derived category $\mathbf{D}(A)$ of a ring $A$. The compact objects are the thick closure of the ring $A$, however, in this case they can be represented also as bounded complex of finitely generated projective $A$-modules, up to homotopy. These objects are also known as \emph{perfect complexes}. In particular, there is an equivalence, $\mathbf{D}^c(A)\,\cong\,\mathbf{K}^b(\lprj{A})=:\operatorname{Perf}(A)$. An extension of this result can be found in \cite{HJ-TAMS}, where the authors define a $Q$-shaped analogue of perfect complexes and show that every perfect object is compact. The converse implication remains still an open question. Holm and J{\o}rgensen, in \emph{loc. cit.} formulated this as the following conjecture. 

\begin{res*}[Conjecture {\cite{HJ-TAMS}}]
    There is an equality of triangulated subcategories of the $Q$-shaped derived category $\QSD{Q}{A}$ of a ring $A$,
    \[    \mathbf{D}^{\operatorname{perf}}_Q(A)\,=\,\mathbf{D}^c_Q(A).
    \]
\end{res*}

In this section, we define perfect objects in abstract homotopy categories obtained by \cite[Thm.~3.4]{SV_Q_shaped_homotopy} and show that such objects are compact, while also providing a sufficient condition for the converse to hold. As an application, we generalise the results of \cite{HJ-TAMS} for compactly generated $Q$-shaped $\theta$-homotopy categories $\QST{Q}{A}$, and prove the conjecture of Holm and J{\o}rgensen for artin algebras $A$ over a commutative hereditary and noetherian ground ring $\Bbbk$.

\textbf{Notation.} For what follows, $\mathcal{E}$ will denote efficient exact category with enough projectives whose underlying additive category is locally finitely presented. For a fixed \emph{set} $\mathscr{G}$ of finitely presented Gorenstein projective objects, denote by $(\mathcal{C}_\mathscr{G},\mathcal{W}_\mathscr{G})$ generated by the set $\mathcal{S}_\mathscr{G}$ as in \cite[Thm.~3.4]{SV_Q_shaped_homotopy}. We also denote by $q\colon \mathcal{E}\to\operatorname{Ho}(\mathcal{E})$ the quotient functor to the homotopy category. Recall by Theorem \ref{thm:compact_generation}, that the homotopy category $\operatorname{Ho}(\mathcal{E})$ is compactly generated by the set $q(\mathcal{S}_\mathscr{G})$. Finally, we denote by $\mathscr{F}:=\mathcal{C}_\mathscr{G}\cap\operatorname{fp}(\mathcal{E})$, the full subcategory of the finitely presented cofibrant objects, and by $\operatorname{prj}(\mathcal{E}):=\Prj{\mathcal{E}}\cap\operatorname{fp}(\mathcal{E})$, the set of all finitely presented projective objects in the exact category $\mathcal{E}$.

An object $X$ in $\operatorname{Ho}(\mathcal{E})$ is called \textbf{perfect} if it is isomorphic in $\operatorname{Ho}(\mathcal{E})$ to the image of an object in $\mathscr F$. In symbols, if $X\,\cong\, qM$ with $M$ in $\mathscr{F}$. We denote the isomorphism closure in $\operatorname{Ho}(\mathcal{E})$ of all perfect objects by $\operatorname{Perf}(\mathcal{E})$. Note that the subcategory of perfect objects coincides with the image of $\mathscr{F}$ in the homotopy category.

\begin{lem}\label{lem:perf=compacts}
    One always has $\operatorname{Perf}(\mathcal{E})\subseteq\operatorname{Ho}(\mathcal{E})^c$ and $\operatorname{Thick}(\operatorname{Perf}(\mathcal{E}))=\operatorname{Ho}(\mathcal{E})^c$.
\end{lem}

\begin{proof}
    Let $M$ be an object in $\mathscr{F}$. Since $M$ is finitely presented, its image $qM$ in the homotopy category $\operatorname{Ho}(\mathcal{E})$ is compact by Lemma \ref{lem:compacts}. This proves the desired inclusion. Moreover, since $\mathcal{S}_\mathscr{G}\subseteq\mathscr F$ by definition, it follows that $\operatorname{Thick}(q(\mathcal{S}_\mathscr{G}))\subseteq\operatorname{Thick}(q\mathscr F)$. In particular, one has by Theorem \ref{thm:compact_generation},
    $$\operatorname{Thick}(\operatorname{Perf}(\mathcal{E}))=\operatorname{Ho}(\mathcal{E})^c=\operatorname{Thick}(q(\mathcal{S}_\mathscr{G})). $$
    This proves the equality. 
\end{proof}

Recall that a category $\mathcal{C}$ is called idempotent complete if every idempotent splits.
For any category $\mathcal{C}$ there exists an idempotent complete category $\operatorname{Kar}(\mathcal{C})$ called the \textbf{idempotent completion} of $\mathcal{C}$, and a fully faithful functor $\mathcal{C}\to\operatorname{Kar}(\mathcal{C})$, which is universal; see, for example, \cite[Sec.~6]{Buhler}. We denote the idempotent completion of a category $\mathcal{C}$ by $\operatorname{Kar}(\mathcal{C})$ instead of $\mathcal{C}^{\wedge}$ (as in \emph{loc. cit.}),  since idempotent complete categories are also commonly referred to as \textbf{Karoubian categories} and idempotent completions as Karoubian envelopes.

\begin{thm}\label{thm:compact=perfect}
    Assume $\mathcal{E}$ to be a locally coherent exact category satisfying the Setup \ref{stp:loc_coh}, and $\operatorname{fp}(\mathcal{E})$ to have enough projectives. Then, $\mathscr{F}$ is an idempotent complete Frobenius category whose projective-injective objects are $\operatorname{prj}(\mathcal{E})$. Its stable category $\underline{\mathscr{F}}$ embeds fully faithfully and exactly in the homotopy category $\operatorname{Ho}(\mathcal{E})$, with image $\operatorname{Perf}(\mathcal{E})$ and 
    \[
    \operatorname{Kar}(\operatorname{Perf}(\mathcal{E}))\,=\, \operatorname{Ho}(\mathcal{E})^c.
    \]
    Consequently, $\operatorname{Perf}(\mathcal{E})=\operatorname{Ho}(\mathcal{E})^c$, if and only if, $\underline{\mathscr{F}}$ is idempotent complete. In particular, if $\operatorname{fp}(\mathcal{E})$ is a Krull-Schmidt category, then 
    \[
    \operatorname{Perf}(\mathcal{E})=\operatorname{Ho}(\mathcal{E})^c.
    \]
\end{thm}

\begin{proof}
    Notice that since $\operatorname{fp}(\mathcal{E})$ is assumed to have enough projectives, the projectives in $\mathcal{E}$ are the additive closure of the projectives in $\operatorname{fp}(\mathcal{E})$, that is, $\Prj{\mathcal{E}}=\operatorname{Add}(\operatorname{prj}(\mathcal{E}))$, see \cite[Prop.~2.13]{schlegel2025exact}.

    Our aim is to show that $\mathscr{F}$ is a Frobenius exact category with projective-injective objects $\operatorname{prj}(\mathcal{E})$. The exact structure of $\mathscr F$ is the inherited one from $\mathcal{E}$ since $\mathscr{F}$ is an extension closed subcategory. Indeed, the subcategory $\mathcal{C}_\mathscr G$ is extension closed as the left class of a cotorsion pair and $\operatorname{fp}(\mathcal{E})$ is extension closed since $\mathcal{E}$ is locally coherent, see \cite[Prop.~1.4]{Positselski_loc_coh} for $\kappa=\aleph_0$. It is also idempotent complete since both $\mathcal{C}_\mathscr{G}$ is closed under retracts and $\operatorname{fp}(\mathcal{E})$ is idempotent complete by \cite[Prop.~1.16~\&~Obs.~2.4]{AdamekRosicky} (see also \cite[p.4]{Positselski_loc_coh} for the same result in notation/terminology more consistent with ours). 

    To prove that $\mathscr{F}$ has enough projectives, note that since $\operatorname{fp}(\mathcal{E})$ has enough projectives, for any $M$ in $\mathscr{F}$ there exists a conflation in $\operatorname{fp}(\mathcal{E})$
    \[
    0\longrightarrow K\longrightarrow P\longrightarrow M\longrightarrow 0.
    \]
    All the objects in the conflation above are cofibrant. Indeed, $M$ is cofibrant since it lies in $\mathscr{F}$, $P$ is cofibrant since every finitely presented projective object is projective also in $\mathcal{E}$ and the projectives are trivial cofibrant objects in our model structure, and finally, the kernel $K$ is cofibrant since $\mathcal{C}_\mathscr{G}$ is resolving, and thus closed under kernels of deflations.

    We now show that $\mathscr{F}$ has enough injectives as well. We want to prove that an arbitrary object $M$ in $\mathscr{F}$ fits into a conflation
    \begin{equation}\label{eq:enough_F-inj}
    0\longrightarrow M\longrightarrow I \longrightarrow M'\longrightarrow 0,
    \end{equation}
    with $I$ in $\operatorname{prj}(\mathcal{E})\subset\mathscr{F}$, and cokernel $M'$ in $\mathscr{F}$.
    Let $\xi\colon0\to M\stackrel{i}{\to} I\to C\to 0$ be a conflation in the Frobenius exact category $\mathcal{C}_\mathscr{G}$, with $I\in \Prj{\mathcal{E}}$. Since $\Prj{\mathcal{E}}=\operatorname{Add}(\operatorname{prj}(\mathcal{E}))$, there exists a complement $I'$ such that 
    \[
    I\oplus I'=\bigoplus_{j\in J}P_j=:P,\quad\text{with } \,P_j\in\operatorname{prj}(\mathcal{E}).
    \]
    Consider now the conflation (the direct sum of $\xi$ with the conflation $0\to0\to I'\to I'\to0$)
    \[
    0\longrightarrow M\stackrel{g}{\longrightarrow} P\longrightarrow C\oplus I'\longrightarrow0,
    \]
    with $g=(i,0)$. Its cokernel $C\oplus I'$ is cofibrant, thus $g$ is a cofibration. The (infinite) coproduct $P$ is a filtered colimit of its finite subcoprodcuts, see \cite[Exa.~4.6]{adamek2010algebraic}. Let us denote the filtered system of all finite subcoproducts by $\Lambda$. Since $M$ is finitely presented, there is an isomorphism
    \[
    \Hom{\mathcal{E}}{M}{P}\,\cong\, \varinjlim_{S\in \Lambda}\Hom{\mathcal{E}}{M}{P_S}.
    \]
    In particular, every morphism $g\colon M\to P$ can be factored through a finite coproduct (indexed by a finite set $J_0\in \Lambda$) as follows:
    \[\begin{tikzcd}
	M & {P':=\bigoplus_{j\in J_0}P'_j} & P \\
	& {C':=\operatorname{Cok}(\alpha)}
	\arrow["\alpha", from=1-1, to=1-2]
	\arrow["\beta", shift left=1, from=1-2, to=1-3]
	\arrow[two heads, from=1-2, to=2-2]
	\arrow["\gamma", shift left=1, from=1-3, to=1-2]
\end{tikzcd},\]
where $\gamma$ is the canonical projection from the infinite coproduct $P$ to the finite coproduct $P'$. Here one has, $g=\beta\alpha$, $\gamma\beta=\operatorname{id}_{P'}$, and $\gamma g=\alpha$. Denote also by $C'$ the cokernel of $\alpha\colon M\to P'$. In the arrow category $\operatorname{Mor}(\mathcal{E})$, whose objects are morphisms in $\mathcal{E}$ and the morphisms are commutative squares, the morphism $\alpha\xrightarrow{(\operatorname{id}_M,\,\beta)}g\xrightarrow{(\operatorname{id}_M,\,\gamma)}\alpha$, has a composite $\operatorname{id}_\alpha$. Notice that the squares in the morphisms above are commutative because, $\beta\alpha=g=g\operatorname{id}_M$ and $\gamma g=\alpha=\alpha\operatorname{id}_M$. It follows that $\alpha$ is a retract of a cofibration, thus a cofibration, which in our exact model structure means that it is an inflation with a cofibrant cokernel. Hence, $C'$ lies in $\mathcal{C}_\mathscr{G}$. Since $P'$ is a finite coproduct of finitely presented projective objects it lies in $\operatorname{prj}(\mathcal{E})\subset\mathscr{F}$. It remains to show that $C'$ is finitely presented, then we have successfully constructed a conflation like in \eqref{enough_F-inj}, with $I=P'$ and $M'=C'$.

Apply the functor $\Hom{\mathcal{E}}{-}{X}$, with $X$ in $\mathcal{E}$, to the (a priori $\mathcal{C}_\mathscr{G}$-)conflation $0\to M\to P'\to C'\to0$. This yields an exact sequence of abelian groups
\[
0\longrightarrow\Hom{\mathcal{E}}{C'}{X}\longrightarrow\Hom{\mathcal{E}}{P'}{X}\longrightarrow\Hom{\mathcal{E}}{M}{X}.
\]
Filtered colimits of abelian groups are exact and the last two Hom functors preserve filtered colimits since $P'$ and $M$ are finitely presented. Diagrammatically, for a filtered diagram $\{X_a\}_{a\in A}$, such that $\varinjlim X_a=X$, we have
\[\begin{tikzcd}
	0 & {\Hom{\mathcal{E}}{C'}{\varinjlim X_a}} & {\Hom{\mathcal{E}}{P'}{\varinjlim X_a}} & {\Hom{\mathcal{E}}{M}{\varinjlim X_a}} \\
	0 & {\varinjlim \Hom{\mathcal{E}}{C'}{X_a}} & {\varinjlim \Hom{\mathcal{E}}{P'}{X_a}} & {\varinjlim \Hom{\mathcal{E}}{M}{X_a}}
	\arrow[from=1-1, to=1-2]
	\arrow[from=1-2, to=1-3]
	\arrow["\cong", dashed, from=1-2, to=2-2]
	\arrow[from=1-3, to=1-4]
	\arrow["\cong"{description}, from=1-3, to=2-3]
	\arrow["\cong"{description}, from=1-4, to=2-4]
	\arrow[from=2-1, to=2-2]
	\arrow[from=2-2, to=2-3]
	\arrow[from=2-3, to=2-4]
\end{tikzcd}\]
We conclude that $\varinjlim\Hom{\mathcal{E}}{C'}{X_a}\,\cong\,\Hom{\mathcal{E}}{C'}{\varinjlim X_a}$. Thus $C'$ is finitely presented, and we constructed an $\mathscr{F}$-conflation as in \eqref{enough_F-inj}. This proves that $\mathscr{F}$ has enough injectives.

Note that in the construction of the projective and injective approximation sequences in $\mathscr F$, we used as projectives and injectives objects from the class $\operatorname{prj}(\mathcal{E})$. Thus,  the projective and injective objects of $\mathscr{F}$ coincide and $\mathscr{F}$ is a Frobenius exact category.

The inclusion functor $\mathscr{F}\longrightarrow\mathcal{C}_\mathscr{G}$ induces an exact stable functor (since a map factoring through a cofibrant finitely presented projective object factors also through a cofibrant projective object) $\underline{\mathscr F}\longrightarrow\underline{\mathcal{C}}_\mathscr G$. The image of $\mathscr{F}$ through this functor is triangulated, coincides with $\operatorname{Perf}(\mathcal{E})$ (by definition) and by Lemma \ref{lem:perf=compacts}, $\operatorname{Thick}(\underline{\mathscr{F}})=\operatorname{Ho}(\mathcal{E})^c$. The homotopy category has coproducts since $\mathcal{E}$ is efficient (see \cite[proof~of~Cor.~3.12]{SaorinStovicek}) and consequently it has split idempotents by \cite[Prop.~1.6.8]{Nee}, thus the idempotent completion $\operatorname{Kar}(\operatorname{Perf}(\mathcal{E}))$ of $\operatorname{Perf}(\mathcal{E})$ coincides with $\operatorname{Thick}(\operatorname{Perf}(\mathcal{E}))$ (see \cite[Prop.~1.3]{MR1027750}), and is a triangulated category, with the unique triangulated structure described in \cite[Thm.~1.5]{Balmer_idempotent_complition}. We conclude that $\operatorname{Kar}(\operatorname{Perf}(\mathcal{E}))=\operatorname{Ho}(\mathcal{E})^c$.

Finally, assume that $\operatorname{fp}(\mathcal{E})$ is a Krull-Schmidt category. Then, since $\mathscr{F}$ is closed under summands, every object $F$ in $\mathscr{F}$ inherits a finite decomposition
\[
F=\bigoplus_{i=1}^n U_i,\quad\text{with }\, U_i\in\mathscr{F},
\]
and every $U_i$ has a local endomorphism ring, let us denote it by $R_i$. Passing to the stable category, $\underline{\mathscr{F}}$, the surviving non-zero summands of $F$, also have local stable endomorphism rings. Indeed, denoting by $\mathfrak{I}$ the ideal of the morphisms factoring through a projective-injective object in $\mathscr{F}$, one has the stable endomorphism rings:
\[
\underline{R_i}=\operatorname{End}_{\mathscr{F}}(U_i)/\mathfrak{I}(U_i,U_i).
\]
If the stable endomorphism ring ring $\underline{R_i}$ is zero, then $U_i$ becomes zero in the stable category. Alternatively, $\underline{R_i}$ is a non-zero quotient of a local ring and hence it is a local ring as well; see \cite[Exc.~19.1(2)]{Lam1}. Thus, $\underline{\mathscr{F}}$ is a Krull-Schmidt category and, by \cite[Cor.~4.4]{Krause_KS} it has split idempotents. We conclude that in this case, $\operatorname{Perf}(\mathcal{E})=\operatorname{Ho}(\mathcal{E})^c$.
\end{proof}

In \cite{HJ-TAMS}, Holm and J{\o}rgensen define an object $X$ in $\lMod{Q,A}$ to be \textbf{strictly perfect}, if it has finite support and takes values in finitely presented\footnote{In fact, they ask for such functors to take values in finitely \emph{generated} projective modules, however, since we mainly work with finitely presented objects, and for any ring the finitely generated projective modules are precisely the finitely presented projective modules, we will write that our functors take values in finitely presented projective modules.} projective $A$-modules. The class of all strictly perfect objects is denoted by ${}_{Q,A}\mathcal{K}$. Note, that every strictly perfect object is a finitely presented object of the locally finitely presented category $\lMod{Q,A}$, see \cite[Prop.~5.9]{HJ-TAMS}; in fact, these objects are of type $\operatorname{FP}_\infty$. An object $X$ in $\QSD{Q}{A}$ is called \textbf{perfect} if it has a strictly perfect cofibrant replacement in the projective model structure on $\lMod{Q,A}$. Denote by
\[
\mathbf{D}^\text{perf}_Q(A)=
     \left\{ X \in \QSD{Q}{A} \;\left|\mspace{-5mu} 
    \begin{array}{l}
    \text{There is an isomorphism $X\cong K$ in $\QSD{Q}{A}$} \\
    \text{for some object $K\in{}^\perp\mathbb{E}\cap{}_{Q,A}\mathcal{K}$}
    \end{array}
    \mspace{-8mu}
    \right.
    \right\} ,
\]
the category of perfect objects in $\QSD{Q}{A}$, where ${}^\perp\mathbb{E}$ denotes the subcategory of cofibrant objects with respect to the projective model structure on $\lMod{Q,A}$.

For what follows, assume that $\varrho$ is a set of finitely presented $A$-modules, and denote its saturation by $\theta$, that is, $\theta=\operatorname{Add}(\varrho)$. We equip $\lMod{Q,A}$ with the objectwise $\theta$-exact structure. Notice that $\theta$ satisfies the Setup \ref{stp:HJ+theta} and the resulting homotopy category is compactly generated by Proposition \ref{prp:comp_gen_q_homotopy}. Consistent with the terminology of Holm and J{\o}rgensen, we define an object $X$ in $\lMode{Q,A}{\theta}$ to be \textbf{$\theta$-strictly perfect} if it has finite support and it takes values in $\operatorname{add}(\varrho)=\operatorname{Add}(\theta)\cap\lsmod{A}$. In symbols,
\[
{}_{Q,A}\mathcal{K}^\theta:=\{X\in\lMode{Q,A}{\theta}\,|\, \supp{X}<\infty,\text{ and } X(q)\in\operatorname{add}(\varrho),\text{ for every }q\in Q\}.
\]
Note that it is important for the functors to take values in $\operatorname{add}(\varrho)$ instead of $\theta=\operatorname{Add}(\varrho)$. For example, in the case of ordinary complexes, if one allows infinite direct sums, the object $\bigoplus_{i\in I}A$ is no longer compact. We show that $\theta$-strictly perfect objects are precisely the objects in $\mathscr{F}=\mathcal{C}_\theta\cap\operatorname{fp}(\lMod{Q,A})$. 

\begin{lem}
    Assume that $Q$ has no cycles. The following hold.
    \begin{rqm} 
    \item The finitely presented objects of $\lMod{Q,A}$ are given by
    \[
\operatorname{fp}(\lMod{Q,A})=\{X\in\lMod{Q,A}\,|\,\supp{X}<\infty\text{ and } X(q)\in\lsmod{A},\text{ for every } q\in Q\};
\]
    \item One has $\mathscr F={}_{Q,A}\mathcal{K}^\theta$.
    \end{rqm}
\end{lem}

\begin{proof}
    \proofoftag{1} Notice that Hom-finiteness and local boundedness give a characterisation of the finitely presented objects. Indeed, by \cite[Lem.~3.11~\&~Prop.~3.12]{HJ-JLMS}, the objects $F_p(A)$, $q\in Q$ are finitely presented projective generators of $\lMod{Q,A}$. Thus, every finitely presented object admits a presentation by finite sums of those objects. This proves the forward inclusion. For the converse, for an object $X$ with finite support and values in $\lsmod{A}$, we use the canonical objectwise split exact sequence from \cite[Prop.~5.7]{HJ-TAMS}
\begin{equation}\label{eq:splice1}
0\longrightarrow \mathbb{K}(X)\longrightarrow\mathbb{F}(X)\longrightarrow X\longrightarrow0.
\end{equation}
The kernel $\mathbb{K}(X)$ has finite support by \cite[Lem.~5.8(b)]{HJ-TAMS} and takes values in finitely presented $A$-modules, since the short exact sequence \eqref{splice1} splits objectwise. Moreover, by \cite[Lem.~5.8(c)]{HJ-TAMS} $\mathbb{F}(X)$ is finitely presented. Consider now the canonical objectwise conflation
\begin{equation}\label{eq:splice2}
    0\longrightarrow \mathbb{K}(\mathbb{K}(X))\longrightarrow\mathbb{F}(\mathbb{K}(X))\longrightarrow \mathbb{K}(X)\longrightarrow0,
\end{equation}
with $\mathbb{F}(\mathbb{K}(X))$ finitely presented by the same arguments. Splicing the exact sequences \eqref{splice1} and \eqref{splice2}, shows that $X$ is a cokernel of a map $\mathbb{F}(\mathbb{K}(X))\to\mathbb{F}(X)$ of finitely presented objects, and thus it is finitely presented.
Notice, that the description of finitely presented objects is an inherited property of the underlying additive category $\lMod{Q,A}$ and does not depend on the exact structure. 

\proofoftag{2} We now prove that $\mathscr{F}\subseteq{}_{Q,A}\mathcal{K}^\theta$. By Theorem \ref{thm:bifibrant_description} every cofibrant object $X$ takes values in $\lPrj{A}^\theta$ which is precisely the class $\theta$ since $\theta$ is assumed to satisfy Setup \ref{stp:HJ+theta}, and hence it is saturated. 
Putting everything together we get that $\mathscr{F}=\mathcal{C}_\theta\cap\operatorname{fp}(\lMod{Q,A})$ is described as
\[
\{X\in\lMode{Q,A}{\theta}\,|\,X(q)\in \theta,\text{ for every }q\in Q\}\,\cap\,\{X\in\lMod{Q,A}\,|\,\supp{X}<\infty\text{ and } X(q)\in\lsmod{A},\text{ for every } q\in Q\}.
\]
Moreover, since $\theta\cap\lsmod{A}=\operatorname{add}(\varrho)$, we conclude that $\mathscr{F}\subseteq{}_{Q,A}\mathcal{K}^\theta$.

For the converse inclusion, assume that $X$ is $\theta$-strictly perfect, then it is finitely presented. To prove that it is also cofibrant, notice that $X$ admits a finite filtration (since $\supp{X}<\infty$) by stalks $S_q(R)$ with $q\in\supp{X}$ and $R\in \varrho$, by \cite[Lem.~3.5]{Slaftsos_silting}. That is, there is a finite series of monomorphisms
\[
0=X^0\rightarrowtail X^1\rightarrowtail X^2\rightarrowtail\cdots\rightarrowtail X^\ell=X,
\]
such that, $X^{i+1}/X^i\,\cong\, S_{q_i}(R)$. We prove by finite induction that $X$ is cofibrant. For $i=0$, we get that $X^1$ is isomorphic to a stalk, and thus it is cofibrant. Assume now that for some $j<\ell$ the object $X^j$ is cofibrant. By definition, there exists an exact sequence
\[
0\longrightarrow X^j\longrightarrow X^{j+1}\longrightarrow S_{q_j}(R)\longrightarrow0
\]
for some $R$ in $\varrho$, with $X^j$ and $S_{q_j}(R)$ in $\mathcal{C_\theta}$. Since  $\mathcal{C}_\theta$ is closed under extensions, $X^{j+1}$ is also cofibrant. This proves that $X$ is cofibrant, and completes the proof.
\end{proof}

We define an object to be \textbf{$\theta$-perfect} in $\QST{Q}{A}$ if its cofibrant replacement is $\theta$-strictly perfect. We denote the subcategory of $\theta$-perfect objects by $\QST{Q}{A}^{\text{perf}}$. For $\varrho=\{A\}$ one recovers the classical definitions of \cite{HJ-TAMS}.

\begin{thm}\label{thm:hj_conj}
    Assume that $Q$ has no cycles. Let also $\varrho$ be a set of finitely presented $A$-modules, $\theta=\operatorname{Add}(\varrho)$, and assume that $\lMod{Q,A}$ is endowed with the objectwise $\theta$-exact structure. Then,
    \[
    \QST{Q}{A}^{\text{perf}}\subseteq\QST{Q}{A}^c\quad\text{and}\quad \operatorname{Kar}(\QST{Q}{A}^{\text{perf}})=\QST{Q}{A}^c.
    \]
    If moreover, $A$ is an artin algebra over $\Bbbk$, then one has
    \[
    \QST{Q}{A}^{\text{perf}}=\QST{Q}{A}^c.
    \]
\end{thm}

\begin{proof}
    To prove the first claim, one could use Theorem \ref{thm:compact=perfect}. However, one does not  a priori know that $\lMode{Q,A}{\theta}$ is a locally coherent exact category. Thus we need to show that $\mathscr{F}={}_{Q,A}\mathcal{K}^\theta$ is an idempotent complete Frobenius exact category also in this case. Also $\operatorname{fp}(\lMod{Q,A})$ does not a priori have enough projectives. It does so, if $A$ is (left) artinian or more generally (left) coherent.

    We show that $\mathscr{F}={}_{Q,A}\mathcal{K}^\theta$ is an idempotent complete Frobenius exact category with projective-injective objects  
    \begin{equation}\label{eq:projinj_in_F}
    \mathcal{P}:=\operatorname{add}\{F_q(R)\in\lMode{Q,A}{\theta}\,|\,q\in Q\text{ and }R\in \varrho\}=\lPrj{Q,A}^\theta\cap\operatorname{fp}(\lMod{Q,A}).
    \end{equation}
    Notice that $\mathscr{F}$ is extension closed, since both $\mathcal{C}_\theta$ and $\operatorname{fp}(\lMod{Q,A})$ are, and thus it inherits an exact structure from the ambient exact category $\lMode{Q,A}{\theta}$. Indeed, $\mathcal{C}_\theta$ is extension closed as the left class of a cotorsion pair. Moreover, given a conflation  $0\to X\to Y\to Z\to 0$  in $\lMode{Q,A}{\theta}$ with $X,Z$ in $\operatorname{fp}(\lMod{Q,A})$, by \cite[Lem.~5.2]{HJ-TAMS}, $Y$ has finite support, and also, it takes values in $\lsmod{A}$, since finitely presented $A$-modules are closed under extensions for any ring $A$. One also has that $\mathscr{F}$ is idempotent complete. The subcategory of cofibrant objects is idempotent complete since $\lMode{Q,A}{\theta}$ is idempotent complete as an efficient exact category; it is weakly idempotent complete with coproducts (\cite[Lem.~1.4]{SaorinStovicek}) and hence idempotent complete by \cite[Rmk.~7.3]{Buhler}. Moreover, $\operatorname{fp}(\lMod{Q,A})$ is idempotent complete as well. The idempotents split in the ambient abelian category $\lMod{Q,A}$ and finitely presented objects are closed under retracts \cite[Lem.~5.9]{adamek2010algebraic}.

    To prove that $\mathscr{F}$ has enough projectives and enough injectives, we use the approximations in \cite[Constr.~4.20]{SV_Q_shaped_homotopy}. In particular, for any $X\in\mathscr{F}$ one has 
    \begin{equation}\label{eq:approx_in_F}\begin{aligned}
        0\longrightarrow \mathbb{K}(X)\longrightarrow  \mathbb{F}(X)=\bigoplus_{q\in\supp{X}}F_q(X(q))\longrightarrow X\longrightarrow0\\
        0\longrightarrow X\longrightarrow  \mathbb{G}(X)=\bigoplus_{q\in\supp{X}}G_q(X(q))\longrightarrow\mathbb{C}(X)\longrightarrow0
    \end{aligned}\end{equation}
    where by \cite[Rmk.~4.22(b)]{SV_Q_shaped_homotopy} the objects $\mathbb{F}(X),\,\mathbb{G}(X),\,\mathbb{K}(X)$ and $\mathbb{C}(X)$ have finite support and take values in $\operatorname{add}\{X(q)\,|\,q\in Q\}\subseteq\operatorname{add}(\varrho)$. The projective and injective objects coincide since $G_q\cong F_{\mathbb{S}^{-1}_q}$ by \cite[Lem.~3.4]{HJ-TAMS}. 
    
    We now prove the description \eqref{projinj_in_F} of the projective-injective objects in $\mathscr{F}$. Take an object $X\in\mathcal{P}$. It is clear that $X$ is finitely presented, and moreover it is projective in $\lMode{Q,A}{\theta}$ thus projective also in the Frobenius category $\mathcal{C}_\theta$. We conclude that it is a projective-injective object in $\mathscr{F}$. Conversely, by \eqref{approx_in_F}, one has that a projective-injective object in $\mathscr{F}$ is a summand of an object $\mathbb{F}(X)$, which, as we already proved above, are finite direct sums of functors with finite support that take values in $\operatorname{add}(\varrho)$. Thus they lie in $\mathcal{P}$. For the last equality in \eqref{projinj_in_F}, notice that \cite[Prop.~4.6(a)]{SV_Q_shaped_homotopy} yields $\lPrj{Q,A}^\theta=\operatorname{Add}\{F_q(R)\,|\,q\in Q,\,R\in \varrho\}$. A finitely presented projective in $\lPrj{Q,A}^\theta$ is a summand of a finite subcoproduct, proving the equality. 

    The proof of 
    \[
    \QST{Q}{A}^{\text{perf}}\subseteq\QST{Q}{A}^c\quad\text{and}\quad \operatorname{Kar}(\QST{Q}{A}^{\text{perf}})=\QST{Q}{A}^c.
    \]
    is identical with the corresponding part in the proof of Theorem \ref{thm:compact=perfect}, since it does not depend on $\lMode{Q,A}{\theta}$ being locally coherent, neither on $\operatorname{fp}(\lMod{Q,A})$ having enough projectives. 

    Assume now that $A$ is an artin algebra. We will prove that $\operatorname{fp}(\lMod{Q,A})$ is a Krull-Schmidt category. Let $X$ be an arbitrary object in $\operatorname{fp}(\lMod{Q,A})$. By Lemma \ref{lem:perf=compacts}, it has finite support and take values $X(q)$, $q\in Q$ in finitely presented (and thus also finitely generated) $A$-modules. Since $A$ is an artin algebra, each $X(q)$ has finite length. We define the integer
    \[
    L(X):=\sum_{q\in\supp{X}}\ell_A(X(q))<\infty
    \]
    of all the lengths $\ell_A(X(q))$ of the values $X(q)$ of $X$. Note that for two subrepresentations $Y,Z\subseteq X$ such that $Y\subsetneq Z$, one has that at some $q\in\supp{X}$, $Y(q)$ is a proper subrepresentation of $Z(q)$ and thus $L(Y)<L(Z)$. We will show that $X$ has finite length in $\lMod{Q,A}$. Choose an ascending series of subrepresentations
    \begin{equation}\label{eq:ascending_series}
    0=Y_0\subsetneq Y_1\subsetneq Y_2\subsetneq\ldots\subsetneq Y_n\subsetneq\ldots\subset X.
    \end{equation}
    Evaluation \eqref{ascending_series} at an object $q\in Q$, yields an ascending series of $A$-modules, that stabilises, since $A$ is artinian. Moreover, we have that 
    \[
    0=L(Y_0)<L(Y_1)<\ldots<L(Y_m)\le L(X),
    \]
    for some positive integer $m$. We conclude that the ascending series in \eqref{ascending_series} has to be finite. Choose a finite series of strict subresentations,
    \begin{equation}\label{eq:composition_series}
        0=X_0\subsetneq X_1\subsetneq\ldots\subsetneq X_r=X
    \end{equation}
    such that the $r$ is the largest possible integer of strict inclusions and no non-trivial subrepresentation can be inserted between two terms of the series. We will show that for every $0\le i\le r$, the object $X_{i+1}/X_i$ is simple. Suppose for a contradiction that $X_{i+1}/X_i\neq0$ is not simple. Then there exists some non-zero subobject $W\subsetneq X_{i+1}/X_i$. Considering the canonical projection
    \[
    \pi\colon X_{i+1}\longrightarrow X_{i+1}/X_i,
    \]
    the preimage $\pi^{-1}(W)$ is a subrepresentation of $X_{i+1}$ and one has $X_i\subsetneq \pi^{-1}(W)\subsetneq X_{i+1}$. We reached a contradiction by the construction of \eqref{composition_series}. We conclude that $X$ has finite length. Finally, by \cite[Lem.~5.1~\&~Thm.~5.5]{Krause_KS}, $X$ admits a decomposition into a finite direct sum of indecomposable objects with local endomorphism rings. Since $X$ was arbitrary, $\operatorname{fp}(\lMod{Q,A})$ is a Krull-Schmidt category, and thus an application of Theorem \ref{thm:compact=perfect}, gives 
    \[
    \QST{Q}{A}^{\text{perf}}=\QST{Q}{A}^c.
    \]
\end{proof}

 For the special case that $\varrho=\{A\}$, the exact structure is the abelian one, $\QST{Q}{A}$ is the $Q$-shaped derived category $\QSD{Q}{A}$, and Theorem \ref{thm:hj_conj} proves the Conjecture of Holm and J{\o}rgensen for artin rings $A$.

\subsection*{Tool and computational resource disclosure}
Consistent with the Leiden Declaration’s recommendation to disclose the use of automated tools transparently, I used ChatGPT (OpenAI) for language editing, including corrections to grammar, spelling, clarity, and sentence structure. Moreover, the initial idea underlying the construction of Example \ref{exa:applications-fg} arose from a suggestion generated by artificial intelligence. I subsequently developed and refined this idea into the form presented here. The tool was not used to generate further ideas, arguments, analyses, or conclusions. I reviewed and approved all suggested changes, and the academic content remains entirely my own.

\bibliographystyle{alpha}
\bibliography{references}

\end{document}